\documentclass{amsart}%
\usepackage{amssymb}
\usepackage{amsmath}
\usepackage{amsfonts}
\usepackage[margin=1.7in]{geometry}
\usepackage[all,cmtip]{xy}
\usepackage{diagxy}
\usepackage[backref]{hyperref}%
\usepackage{graphicx}

\newtheorem{theorem}{Theorem}
\theoremstyle{plain}
\newtheorem{definition}{Definition}
\newtheorem*{acknowledgement}{Acknowledgement}

\newtheorem{construction}{Construction}

\newtheorem{lemma}{Lemma}

\newtheorem{proposition}[theorem]{Proposition}

\numberwithin{equation}{section}
\theoremstyle{remark}
\newtheorem{remark}[lemma]{Remark}
\newtheorem*{openproblem}{Open Problem}
\theoremstyle{definition}
\newtheorem{example}{Example}
\begin{document}
\title{Operator ideals in Tate objects}
\author{O. Braunling,\quad M. Groechenig,\quad J. Wolfson}

\address{Department of Mathematics, Universit\"{a}t Freiburg, FRG}
\email{oliver.braeunling@math.uni-freiburg.de}
\address{Department of Mathematics, Imperial College London, UK}
\email{m.groechenig@imperial.ac.uk}
\address{Department of Mathematics, University of Chicago, USA}
\email{wolfson@math.uchicago.edu}

\thanks{O.B.\ was supported by GK1821 \textquotedblleft Cohomological Methods in
Geometry\textquotedblright. M.G.\ was partially supported by EPSRC Grant
No.\ EP/G06170X/1. J.W.\ was partially supported by an NSF Post-doctoral
Research Fellowship under Grant No.\ DMS-1400349. Our research was supported
in part by NSF Grant No.\ DMS-1303100 and EPSRC Mathematics Platform grant EP/I019111/1.}

\begin{abstract}
Tate's central extension originates from 1968 and has since found many
applications to curves. In the 80s Beilinson found an $n$-dimensional
generalization: cubically decomposed algebras, based on ideals of bounded and
discrete operators in ind-pro limits of vector spaces. Kato and Beilinson
independently defined `($n$-)Tate categories' whose objects are formal
iterated ind-pro limits in general exact categories. We show that the
endomorphism algebras of such objects often carry a cubically decomposed
structure, and thus a (higher) Tate central extension. Even better, under very
strong assumptions on the base category, the $n$-Tate category turns out to be
just a category of projective modules over this type of algebra.

\end{abstract}
\maketitle

In his 1980 paper \textquotedblleft Residues and ad\`{e}les\textquotedblright%
\ \cite{MR565095} A. A. Beilinson introduced the following algebraic
structure, without giving it a name:

\begin{definition}
\label{def_BeilNFoldAlg}A \emph{Beilinson }$n$\emph{-fold cubical algebra} is

\begin{itemize}
\item an associative $k$-algebra $A$;

\item two-sided ideals $I_{i}^{+},I_{i}^{-}$ such that $I_{i}^{+}+I_{i}^{-}=A
$ for $i=1,\ldots,n$;

\item call $I_{tr}:=\bigcap_{i=1,\ldots,n}I_{i}^{+}\cap I_{i}^{-}$ the
\emph{trace-class} operators of $A$.
\end{itemize}
\end{definition}

In his 1987 paper \textquotedblleft How to glue perverse
sheaves\textquotedblright\ \cite{MR923134} he introduced the exact category
$1$-$\mathsf{Tate}_{\aleph_{0}}^{el}\mathcal{C}$ for any given exact category
$\mathcal{C}$. It was suggestively denoted by $\underleftrightarrow{\lim
}\,\mathcal{C}$ in loc. cit. We shall recall its definition in
\S \ref{sect_TateCategories}.

Although these two papers do not cite each other, some ideas in them can be
viewed as two sides of the same coin. In the present paper we establish a
rigorous connection between them. In fact, the main idea is that the latter
category $-$ under a number of assumptions $-$ are just the projective modules
over the former type of algebras. But this really requires some assumptions
$-$ in general it is quite far from the truth.\medskip

Define $n$-$\mathsf{Tate}_{\aleph_{0}}^{el}(\mathcal{C}):=\mathsf{Tate}%
_{\aleph_{0}}^{el}(\,(n-1)$-$\mathsf{Tate}_{\aleph_{0}}(\mathcal{C})\,)$ and
$n$-$\mathsf{Tate}_{\aleph_{0}}(\mathcal{C})$ as the idempotent completion of
$n$-$\mathsf{Tate}_{\aleph_{0}}^{el}(\mathcal{C})$. We write $P_{f}(R)$ for
the category of finitely generated projective right $R$-modules.

\begin{theorem}
\label{intro_Thm1}Let $\mathcal{C}$ be an idempotent complete split exact category.

\begin{enumerate}
\item For every object $X\in n$-$\mathsf{Tate}_{\aleph_{0}}^{el}(\mathcal{C})$
its endomorphism algebra canonically carries the structure of a Beilinson
$n$-fold cubical algebra.

\item If there is a countable family of objects $\{S_{i}\}$ in $\mathcal{C}$
such that every object in $\mathcal{C}$ is a direct summand of some countable
direct sum of objects from $\{S_{i}\}$, then there exists (non-canonically) an
object $X\in n$-$\mathsf{Tate}_{\aleph_{0}}^{el}(\mathcal{C})$ such that%
\[
n\text{-}\mathsf{Tate}_{\aleph_{0}}(\mathcal{C})\overset{\sim}{\longrightarrow
}P_{f}(R)\qquad\text{with}\qquad R:=\operatorname*{End}(X)\text{.}%
\]

\item Under this equivalence, the ideals $I_{i}^{\pm}$ correspond to certain
categorical ideals, which can be defined even if $\mathcal{C}$ is not split exact.
\end{enumerate}
\end{theorem}

See Theorem \ref{thm_CSplitExactIdealsAddUpGetNFoldCubicalBeilAlgebra} and
Theorem \ref{thm_body_TateObjsAreProjModules} in the paper for details. In
other words: In some sense the approaches of the 1980 paper and the 1987 paper
are essentially equivalent. If $\mathcal{C}$ is not split exact, the ideals
$I_{i}^{+},I_{i}^{-}$ still exist, see Theorem
\ref{theorem_CIdempCompleteHaveTateIdeals} in the text. However, the property
$I_{i}^{+}+I_{i}^{-}=A$ can fail to hold; Example
\ref{Example_FinAbelianPGroups} will give a counter-example.\medskip

V. G. Drinfeld has also introduced a category fitting into the same context,
his notion of \textquotedblleft Tate $R$-modules\textquotedblright\ for a
given ring $R$ \cite{MR2181808}. We call it $\mathsf{Tate}^{Dr}(R)$ and give
the definition later. In loc. cit. these appear without a restriction on the
cardinality. However, if we restrict to countable cardinality, then Theorem
\ref{intro_Thm1} also implies:

\begin{theorem}
Let $R$ be a commutative ring. Then there is an exact equivalence of
categories%
\[
\mathsf{Tate}_{\aleph_{0}}^{Dr}(R)\overset{\sim}{\longrightarrow}%
P_{f}(E)\text{,}%
\]
where $E$ is the Beilinson $1$-fold cubical algebra%
\[
E:=\operatorname*{End}\nolimits_{\mathsf{Tate}_{\aleph_{0}}^{Dr}(R)}\left(
\left.  R((t))\right.  \right)  \text{,}%
\]
where \textquotedblleft$R((t))$\textquotedblright\ is understood to be the
`Tate $R$-module \`{a} la Drinfeld' denoted by $R((t))$ in Drinfeld's paper
\cite{MR2181808}.
\end{theorem}

See Theorem \ref{thm_body_ModulesForTateObjsALaDrinfeld} in the paper. This
also reveals a certain additional structure on endomorphisms of Drinfeld's
Tate $R$-modules, which appears not to have been studied so far at all.

Beilinson has originally considered the category $1$-$\mathsf{Tate}%
_{\aleph_{0}}^{el}\mathcal{C}$, i.e. without idempotent completion. Our
previous paper \cite[\S 3.2.7]{TateObjectsExactCats} shows:

\begin{theorem}
The category $1$-$\mathsf{Tate}_{\aleph_{0}}^{el}(\mathcal{C})$ can fail to be
idempotent complete. In particular, one cannot improve Theorem
\ref{intro_Thm1} to%
\[
n\text{-}\mathsf{Tate}_{\aleph_{0}}^{el}(\mathcal{C})\overset{\sim
}{\longrightarrow}P_{f}(R)\text{,}%
\]
i.e. without the idempotent completion, regardless which ring $R$ is taken.
\end{theorem}

This follows simply since $P_{f}(-)$ is always an idempotent complete
category. For some constructions the categories $n$-$\mathsf{Tate}_{\aleph
_{0}}(\mathcal{C})$ are too small since the admissible Ind- and Pro-limits are
only allowed to be taken over \textit{countable} diagrams. This happens for
example when writing down the ad\`{e}les of a curve over an uncountable base
field as a 1-Tate object. In our previous paper \cite{TateObjectsExactCats} we
have therefore constructed categories $n$-$\mathsf{Tate}_{\kappa}%
(\mathcal{C})$, constraining the size of limits by a general infinite cardinal
$\kappa$. Examples due to J.\ \v{S}\v{t}ov\'{\i}\v{c}ek and J. Trlifaj
\cite[Appendix B]{TateObjectsExactCats} demonstrate the following

\begin{theorem}
Even if $\mathcal{C}$ is split exact and idempotent complete, the category $1
$-$\mathsf{Tate}_{\kappa}(\mathcal{C})$ for $\kappa>\aleph_{0}$ can fail to be
split exact. In particular, one cannot improve Theorem \ref{intro_Thm1} to
general cardinalities $\kappa$.
\end{theorem}

A key application of our results are to ad\`{e}les of schemes, as introduced
by A. N. Parshin and Beilinson \cite{MR0419458}, \cite{MR565095}. A detailed
account was given by A. Huber \cite{MR1138291}. We state the next result in
the language of these papers, but the reader will also find the necessary
notation and background explained in the main body of the present text:

\begin{theorem}
\label{thm_intro_OnAdeles}Let $k$ be a field and $X/k$ a finite type scheme of
pure dimension $n$. Let $\mathcal{F}$ be any quasi-coherent sheaf and
$\triangle\subseteq S\left(  X\right)  _{n}$ a subset.

\begin{enumerate}
\item Then the Beilinson-Parshin ad\`{e}les $A(\triangle,\mathcal{F})$ can be
viewed as an elementary $n$-Tate object in finite-dimensional $k$-vector
spaces, i.e. so that%
\[
A(\triangle,\mathcal{F})\in n\text{-}\mathsf{Tate}^{el}\left(  \mathsf{Vect}%
_{f}\right)  \text{.}%
\]

\item The ring $\operatorname*{End}\left(  A(\triangle,\mathcal{O}%
_{X})\right)  $ carries the structure of an $n$-fold cubical\ Beilinson
algebra as in Definition \ref{def_BeilNFoldAlg}.

\item If $\triangle=\{(\eta_{0}>\cdots>\eta_{n})\}$ is a singleton and
$\operatorname*{codim}\nolimits_{X}\overline{\{\eta_{i}\}}=i$, there is a
canonical isomorphism $\operatorname*{End}\nolimits_{\left.  n\text{-}%
\mathsf{Tate}^{el}\right.  }\left(  A(\triangle,\mathcal{O}_{X})\right)  \cong
E_{\triangle}^{\operatorname*{Beil}}$, where $E_{\triangle}%
^{\operatorname*{Beil}}$ denotes Beilinson's original cubical algebra from
\cite[\S 3, \textquotedblleft$\mathrm{E}_{\mathrm{\Delta}}$\textquotedblright%
]{MR565095} (defined without Tate categories).
\end{enumerate}
\end{theorem}

See Theorem \ref{thm_body_AdelesAreSlicedOverVectSpaces} in the paper $-$ in a
way this result is the counterpart of a recent result of Yekutieli
\cite[Theorem 0.4]{MR3317764}, who uses topologies instead of Tate objects
however. Theorem \ref{thm_intro_OnAdeles} does not follow from Theorem
\ref{intro_Thm1} since ad\`{e}les with very few exceptions hinge on forming
uncountably infinite limits. Trying to generalize (1), one may also view the
ad\`{e}les as $n$-Tate objects over other categories, e.g. finite abelian
groups if $k$ gets replaced by the integers $\mathbf{Z}$, or coherent sheaves
with zero-dimensional support. However, for these variations parts (2) and (3)
of the theorem would be false. We refer the reader to
\S \ref{sect_ApplicationsToAdeles} for counter-examples\medskip

Historically, J. Tate's paper \cite{MR0227171} introduced the first example of
a Beilinson $n$-fold cubical algebra, but only for the case $n=1$.\ He
developed a formalism of traces for his trace-class operators, lifting the
trace of finite-dimensional vector spaces. We can generalize this to exact categories:

An \textit{exact trace} is a natural notion of a formalism of traces for a
general exact category, see \S \ref{sect_TraceClassOps} for details.

\begin{theorem}
Suppose $\mathcal{C}$ is an idempotent complete exact category and
$\operatorname*{tr}\nolimits_{(-)}$ an exact trace on $\mathcal{C}$ with
values in an abelian group $Q$. Then for every object $X\in\left.
n\text{-}\mathsf{Tate}(\mathcal{C})\right.  $ and $I_{tr}:=I_{tr}%
(\operatorname*{End}(X))$ its trace-class operators, there is a canonically
defined trace%
\[
\tau_{X}:I_{tr}/[I_{tr},I_{tr}]\rightarrow Q\text{,}%
\]
such that for a short exact sequence $A\hookrightarrow B\twoheadrightarrow
A/B$ and $f\in I_{tr}(B)$ so that $f\mid_{A}$ factors over $A$, we have%
\[
\tau_{B}(f)=\tau_{A}(f\mid_{A})+\tau_{A/B}(\overline{f})\text{.}%
\]
If $X\in\mathcal{C}$, this trace agrees with the given trace, $\tau
_{X}=\operatorname*{tr}\nolimits_{X}$.
\end{theorem}

See Prop. \ref{Prop_TraceExists} for the full statement, which is more
detailed and gives a unique characterization of $\tau$ in terms of the input
trace. We also get:

\begin{theorem}
Let $\mathcal{C}$ be an idempotent complete exact category. Then for every
trace-class morphism $\varphi\in I_{tr}(X,X)$ some sufficiently high power
$\varphi^{\circ r}$ (or a sufficiently long word made from several such
morphisms) will factor through an object in $\mathcal{C}$.
\end{theorem}

This generalizes a property which Tate had baptized `finite-potent' and which
plays a key role in his construction of the trace.

Tate and Beilinson used the $n$-fold cubical algebras to produce
(higher)\ central extensions. The classical example is Tate's central
extension, which encodes the residue of a rational $1$-form. Ultimately, these
constructions can be translated into Lie (and Hochschild) homology classes.
Under mild assumptions, we can construct these classes also for the
endomorphism algebras of $n$-Tate objects.

\begin{theorem}
Let $\mathcal{C}$ be a $k$-linear abelian category with a $k$-valued exact
trace. For every $n$-sliced object\footnote{See the main body of the text for
definitions.} $X\in\left.  n\text{-}\mathsf{Tate}(\mathcal{C})\right.  $ the
endomorphism algebra $E:=\operatorname*{End}(X)$ is a Beilinson $n$-fold
cubical algebra and

\begin{enumerate}
\item its Lie algebra $\mathfrak{g}_{X}:=E_{Lie}$ carries a canonical
Beilinson-Tate Lie cohomology class,%
\[
\phi_{\operatorname*{Beil}}\in H_{\operatorname*{Lie}}^{n+1}(\mathfrak{g}%
_{X},k)\text{;}%
\]

\item as well as canonical Hochschild and cyclic homology functionals%
\[
\phi_{HH}:HH_{n}(E)\rightarrow k\qquad\text{resp.}\qquad\phi_{HC}%
:HC_{n}(E)\rightarrow k\text{.}%
\]

\end{enumerate}
\end{theorem}

See Theorem \ref{Thm_BTClassesExist} for details. For $n=1$ the class
$\phi_{\operatorname*{Beil}}$ just happens to define a central extension as a
Lie algebra. Of course, the classical examples are all special cases of this
construction. We provide some examples in \S \ref{sect_TateExtension}.

Tate categories and Beilinson cubical algebras have already found quite
diverse applications. Ranging from residue symbols in \cite{MR0227171},
\cite{MR565095}, glueing sheaves \cite{MR923134}, over models for
infinite-dimensional vector bundles \cite{MR2181808}\footnote{Drinfeld's paper
proposes several notions, one of them being Tate $R$-modules, which are
closest to the subject of this paper. The other are flat Mittag-Leffler
bundles, which are however also related via admissible Ind-objects of
projectives. See \cite{MR3223352}, \cite[Appendix]{TateObjectsExactCats}.}, to
higher local compactness and Fourier theory \cite{MR1804933}, \cite{MR2473773}%
, \cite{MR2866188}, e.g. for the representation theory of algebraic groups
over higher local fields \cite{MR1800352}, \cite{MR2100671}, \cite{MR2192062},
\cite{MR2214248}, \cite{MR2285231}, \cite{MR2760661}.\medskip

Quite recently, B. Hennion has introduced Tate categories for stable $\infty
$-categories \cite{HennionTateObjsInStableInftyCats},
\cite{HennionHighDimFormalLoopSpace}. It would be interesting to study the
counterparts of our results in this context. Higher local fields can be
regarded as $n$-Tate objects and in \cite{MR2314612} D. V. Osipov has already
related ad\`{e}les to categories similar to $n$-Tate categories and studied
endomorphism rings in this context. In a quite different direction, A.
Yekutieli \cite{MR1213064}, \cite{MR3317764} develops the use of
semi-topological algebraic structures to describe ad\`{e}les. These also give
rise to an $n$-fold cubical algebra, but in a different way based on picking
coefficient fields for the individual layers of the involved higher local
fields. The relation to his approach is explained in \cite{bgwGeomAnalAdeles}.

\begin{acknowledgement}
We thank Theo B\"{u}hler and Amnon Yekutieli for very helpful correspondence.
The paper is heavily based on ideas of Mikhail Kapranov.
\end{acknowledgement}

\section{\label{sect_TateCategories}Tate categories}

For every exact category $\mathcal{C}$ one can form the corresponding
categories of admissible Ind-objects $\mathsf{Ind}^{a}\mathcal{C}$ or
admissible Pro-objects $\mathsf{Pro}^{a}\mathcal{C}$, perhaps with some
conditions on the allowed cardinality of diagrams, denoted by a subscript as
in $\mathsf{Ind}_{\aleph_{0}}^{a}$ or more generally $\mathsf{Ind}_{\kappa
}^{a}$ for $\kappa$ an infinite cardinal. See \cite{MR2872533},
\cite{TateObjectsExactCats} for definitions and basic properties. Enlarging
$\mathcal{C}$ in both of these ways, we arrive at the commutative square of
inclusion functors%
\[%
%
%
\bfig\Square(0,0)[\mathcal{C}`\mathsf{Ind}^{a}\mathcal{C}`\mathsf{Pro}%
^{a}\mathcal{C}`\mathsf{Ind}^{a}\mathsf{Pro}^{a}\mathcal{C}\text{.};```]
\efig
\]

\begin{definition}
Let $\mathcal{C}$ be an exact category.

\begin{enumerate}
\item The category $\mathsf{Tate}^{el}\mathcal{C}$ is the smallest
extension-closed full sub-category of $\mathsf{Ind}^{a}\mathsf{Pro}%
^{a}\mathcal{C}$ which contains both $\mathsf{Ind}^{a}\mathcal{C}$ and
$\mathsf{Pro}^{a}\mathcal{C}$.

\item $\mathsf{Tate}(\mathcal{C})$ denotes the idempotent completion of
$\mathsf{Tate}^{el}\mathcal{C}$.

\item Define $n$-$\mathsf{Tate}^{el}(\mathcal{C}):=\mathsf{Tate}^{el}\left(
\,(n-1)\text{-}\mathsf{Tate}(\mathcal{C})\,\right)  $ and $n$-$\mathsf{Tate}%
(\mathcal{C}):=n$-$\mathsf{Tate}^{el}(\mathcal{C})^{ic}$ as its idempotent completion.
\end{enumerate}
\end{definition}

All of these categories come with a natural exact category structure so that
all basic tools of homological algebra are available, they have derived
categories, $K$-theory, etc\ldots\ Versions of the Tate category were first
introduced by K. Kato in 1980 in an IHES preprint, published only much later
\cite{MR1804933}, and independently by A. Beilinson under the suggestive name
$\underleftrightarrow{\lim}\,\mathcal{C}$ \cite{MR923134}. The equivalence of
these two approaches was established by L. Previdi. There is also a slightly
different variant due to V. Drinfeld \cite{MR2181808}. We refer the reader to
\cite{MR2872533}, \cite{TateObjectsExactCats} for extensive discussions of
these categories and comparison results. The `$C_{n} $ categories' of D. V.
Osipov \cite{MR2314612} a based on similar ideas. The definition of Tate
categories which we give here is due to \cite{TateObjectsExactCats}.

\begin{example}
[Kapranov]If $\mathcal{C}$ is the abelian category of finite-dimensional
$k$-vector spaces, $\mathsf{Tate}^{el}\mathcal{C}$ is equivalent to the exact
category of locally linearly compact topological $k$-vector spaces.
\end{example}

\begin{example}
If $R$ is a commutative ring and $P_{f}\left(  R\right)  $ the exact category
of finitely generated projective $R$-modules, $\mathsf{Tate}_{\aleph_{0}}%
P_{f}\left(  R\right)  $ is equivalent to the category of countably generated
\textquotedblleft Tate $R$-modules\textquotedblright\ in the sense of Drinfeld
\cite{MR2181808}. Without the restriction on countable generation, the latter
is in general only a full sub-category of the former. Both are proven in
\cite[Thm. 5.30]{TateObjectsExactCats}.
\end{example}

\begin{example}
We refer the reader to the works of Kato \cite{MR1804933}, Kapranov
\cite[Appendix]{MR1800352} and Previdi \cite{MR2872533} for a discussion of
Tate categories for non-additive categories. These will not appear in the
present paper.
\end{example}

Every object in the category $\mathsf{Tate}^{el}\mathcal{C}$ comes with a
notion of `lattices'.

\begin{definition}
Let $X\in\mathsf{Tate}^{el}\mathcal{C}$ be an object. We call a
sub-object\footnote{In this paper, the symbols $\hookrightarrow$ and
$\twoheadrightarrow$ denote admissible monics and epics with respect to the
exact structure of a category. Moreover, a sub-object always refers to an
admissible sub-object in the sense that the inclusion is an admissible monic.}
$L\hookrightarrow X$ a \emph{lattice} (or \emph{Tate lattice} if we wish to
contrast the notion to other concepts of lattices) if $L\in\mathsf{Pro}%
^{a}\mathcal{C}$ and $X/L\in\mathsf{Ind}^{a}\mathcal{C}$. The set of all
lattices in $X$ is the \emph{Sato Grassmannian }$Gr(X)$.
\end{definition}

There are two basic properties which control most of the behaviour of this
concept of lattices: If $L^{\prime}\hookrightarrow L\hookrightarrow X$ are two
lattices in $X$, then $L/L^{\prime}$ lies in the category $\mathcal{C}$
\cite[Proposition 6.6]{TateObjectsExactCats} . Furthermore, if $\mathcal{C}$
is idempotent complete, any two lattices have a common sub-lattice and a
common over-lattice \cite[Theorem 6.7]{TateObjectsExactCats}.

\section{\label{sect_ClassicalMotivatingExample}The motivating classical
example}

The following algebraic structure was introduced by Beilinson \cite{MR565095}
for the purpose of generalizing Tate's 1968 construction of the
one-dimensional residue symbol \cite{MR0227171} to higher dimensions. The
constructions in loc. cit. produce a kind of generalized residue symbol for
any such algebraic structure. The importance of this structure extends far
beyond just the residue symbol. In a way, it axiomatizes essential algebraic
features of the endomorphism algebra of a well-behaved $n$-Tate object. Before
addressing this, let us recall the definition:

\begin{definition}
\label{BT_DefCubicallyDecompAlgebra}\cite[\S 1]{MR565095} Let $k$ be a field.
An \emph{(}$n$\emph{-fold)} \emph{cubically decomposed algebra} over $k$ is
the datum $(A,(I_{i}^{\pm}),\tau)$:

\begin{itemize}
\item an associative $k$-algebra $A$;

\item two-sided ideals $I_{i}^{+},I_{i}^{-}$ such that $I_{i}^{+}+I_{i}^{-}=A
$ for $i=1,\ldots,n$;

\item writing $I_{i}^{0}:=I_{i}^{+}\cap I_{i}^{-}$ and $I_{tr}:=I_{1}^{0}%
\cap\cdots\cap I_{n}^{0}$, a $k$-linear map (called \emph{trace})%
\[
\tau:I_{tr}/[I_{tr},A]\rightarrow k\text{.}%
\]

\end{itemize}
\end{definition}

We next recall the original key example for this structure, coming straight
from geometry. Suppose $X/k$ is a reduced scheme of finite type and pure
dimension $n$. We shall use the same notation as in \cite{MR565095}. Notably,
$S\left(  X\right)  _{\bullet}$ denotes the simplicial set of flags of points
(i.e. $S\left(  X\right)  _{n}=\{(\eta_{0}>\cdots>\eta_{n})\}$ with $\eta
_{i}\in X$ and $x>y$ means that $\overline{\{x\}}\ni y$). Further, given
$\triangle\subseteq S\left(  X\right)  _{n}$ we write $\left.  _{\eta_{0}%
}\triangle\right.  :=\{(\eta_{1}>\cdots>\eta_{n})\mid(\eta_{0}>\cdots>\eta
_{n})\in\triangle\}$. Finally, $A(\triangle,\mathcal{M}) $ denotes the
\emph{Beilinson-Parshin ad\`{e}les} for $\triangle\subseteq S\left(  X\right)
_{n}$. This means that for any \textit{coherent} sheaf $\mathcal{M}$ we define%
\begin{align*}
A(\triangle,\mathcal{M}):=  &
{\textstyle\prod\nolimits_{\eta\in\triangle}}
\underset{i}{\underleftarrow{\lim}}\mathcal{M}\underset{\mathcal{O}_{X}%
}{\otimes}\mathcal{O}_{X,\eta}/\mathfrak{m}_{\eta}^{i}\text{\qquad(in the case
}n=0\text{)}\\
A(\triangle,\mathcal{M}):=  &
{\textstyle\prod\nolimits_{\eta\in X}}
\underset{i}{\underleftarrow{\lim}}A(\left.  _{\eta}\triangle\right.
,\mathcal{M}\underset{\mathcal{O}_{X}}{\otimes}\mathcal{O}_{X,\eta
}/\mathfrak{m}_{\eta}^{i})\text{\qquad(in the case }n\geq1\text{)}%
\end{align*}
and for a \textit{quasi-coherent} sheaf $\mathcal{M}$ we define $A(\triangle
,\mathcal{M}):=\underrightarrow{\operatorname*{colim}}_{\mathcal{M}^{\prime}%
}A(\triangle,\mathcal{M}^{\prime})$ and the colimit is taken over the category
of coherent sub-sheaves of $\mathcal{M}$ with inclusions as morphisms. These
colimits and limits are usually taken in the bi-complete category of
$\mathcal{O}_{X}$-module sheaves. We follow this viewpoint here as well, at
least for the moment. Later, in \S \ref{sect_ApplicationsToAdeles} we shall
address a novel perspective using Tate categories instead.

\begin{definition}
[Beilinson \cite{MR565095}]\label{def_BeilLattice}Let $\triangle=\{(\eta
_{0}>\cdots>\eta_{i})\}\in S\left(  X\right)  _{i}$ be given and $M$ a
finitely generated $\mathcal{O}_{\eta_{0}}$-module. Then a \emph{(Beilinson)
lattice} in $M$ is a finitely generated $\mathcal{O}_{\eta_{1}}$-module
$L\subseteq M$ such that $\mathcal{O}_{\eta_{0}}\cdot L=M$. Now and later on,
we shall use the abbreviation%
\[
M_{\triangle}:=A(\triangle,M)
\]
for $M$ a quasi-coherent sheaf on $X$.
\end{definition}

Whenever we are given a $\triangle$ as above, define $\triangle^{\prime
}:=\{(\eta_{1}>\cdots>\eta_{n})\}$, removing the initial entry.

\begin{definition}
[Beilinson \cite{MR565095}]\label{def_HigherAdeleOperatorIdeals}Let $M_{1}$
and $M_{2}$ be finitely generated $\mathcal{O}_{\eta_{0}}$-modules.

\begin{enumerate}
\item Let $\operatorname*{Hom}\nolimits_{\varnothing}(M_{1},M_{2}%
):=\operatorname*{Hom}\nolimits_{k}(M_{1},M_{2})$ be the set of all $k$-linear
maps. Then we define%
\[
\operatorname*{Hom}\nolimits_{\triangle}(M_{1},M_{2})\subseteq
\operatorname*{Hom}\nolimits_{k}(M_{1\triangle},M_{2\triangle})
\]
to be the sub-$k$-module of all $f\in\operatorname*{Hom}\nolimits_{k}%
(M_{1\triangle},M_{2\triangle})$ such that for all (Beilinson) lattices
$L_{1}\hookrightarrow M_{1},L_{2}\hookrightarrow M_{2}$ there exist
(Beilinson) lattices $L_{1}^{\prime}\hookrightarrow M_{1},L_{2}^{\prime
}\hookrightarrow M_{2}$ with%
\[
L_{1}^{\prime}\hookrightarrow L_{1},\qquad L_{2}\hookrightarrow L_{2}^{\prime
},\qquad f(L_{1\triangle^{\prime}}^{\prime})\hookrightarrow L_{2\triangle
^{\prime}},\qquad f(L_{1\triangle^{\prime}})\hookrightarrow L_{2\triangle
^{\prime}}^{\prime}%
\]
and for all such $L_{1},L_{1}^{\prime},L_{2},L_{2}^{\prime}$ the induced $k
$-linear map%
\[
\overline{f}:(L_{1}/L_{1}^{\prime})_{\triangle^{\prime}}\rightarrow
(L_{2}^{\prime}/L_{2})_{\triangle^{\prime}}%
\]
lies in $\operatorname*{Hom}\nolimits_{\triangle^{\prime}}(L_{1}/L_{1}%
^{\prime},L_{2}^{\prime}/L_{2})$.

\item Let $I_{1\triangle}^{+}(M_{1},M_{2})$ be those morphisms $f\in
\operatorname*{Hom}\nolimits_{\triangle}(M_{1},M_{2})$ such that there exists
a lattice $L\hookrightarrow M_{2}$ with $f(M_{1\triangle})\hookrightarrow
L_{\triangle^{\prime}}$.

\item Dually, $I_{1\triangle}^{-}(M_{1},M_{2})$ is formed of those such that
there exists a lattice $L\hookrightarrow M_{1}$ with $f(L_{\triangle^{\prime}%
})=0$.

\item For $i\geq2$ we let $I_{i\triangle}^{+}(M_{1},M_{2})$ be those
$f\in\operatorname*{Hom}\nolimits_{\triangle}(M_{1},M_{2})$ such that for all
lattices $L_{1},L_{1}^{\prime},L_{2},L_{2}^{\prime}$ as in part (3) the
condition%
\[
\overline{f}\in I_{(i-1)\triangle^{\prime}}^{+}(L_{1}/L_{1}^{\prime}%
,L_{2}^{\prime}/L_{2})
\]
holds. Analogously, define $I_{i\triangle}^{-}(M_{1},M_{2})$ to be those with
$\overline{f}\in I_{(i-1)\triangle^{\prime}}^{-}(L_{1}/L_{1}^{\prime}%
,L_{2}^{\prime}/L_{2})$.
\end{enumerate}
\end{definition}

With these definitions in place we are ready to formulate the principal source
of algebras as in Definition \ref{def_BeilNFoldAlg}:

\begin{theorem}
[{Beilinson, \cite[\S 3]{MR565095}}]%
\label{Thm_BeilFlagInSchemeGivesCubDecompAlgebra}Suppose $X/k$ is a reduced
finite type scheme of pure dimension $n$. Let $\eta_{0}>\cdots>\eta_{n}\in
S\left(  X\right)  _{n}$ be a flag with $\operatorname*{codim}\nolimits_{X}%
\overline{\{\eta_{i}\}}=i$. Then%
\[
E_{\triangle}^{\operatorname*{Beil}}:=\operatorname*{Hom}\nolimits_{\triangle
}(\mathcal{O}_{\eta_{0}},\mathcal{O}_{\eta_{0}})\qquad\subseteq
\operatorname*{End}\nolimits_{k}(\mathcal{O}_{X\triangle},\mathcal{O}%
_{X\triangle})
\]
is an associative sub-algebra. Define $I_{i\triangle}^{\pm}\subseteq
E_{\triangle}^{\operatorname*{Beil}}$ by $I_{i\triangle}^{\pm}(\mathcal{O}%
_{\eta_{0}},\mathcal{O}_{\eta_{0}})$ for $1\leq i\leq n$. Then $(E_{\triangle
}^{\operatorname*{Beil}},(I_{i\triangle}^{\pm}),\operatorname*{tr})$ is an
$n$-fold cubically decomposed algebra. Here \textquotedblleft%
$\operatorname*{tr}$\textquotedblright\ refers to Tate's trace for
finite-potent morphisms, see \cite{MR0227171} for the definition. In
particular, $E_{\triangle}^{\operatorname*{Beil}}$ is an example of the
algebras in Definition \ref{def_BeilNFoldAlg}.
\end{theorem}

There are a number of other examples leading to cubically decomposed algebras:

\begin{example}
[Yekutieli]Every topological higher local field (TLF) \cite{MR1213064},
\cite{MR1352568} carries Yekutieli's canonical cubically decomposed algebra
structure \cite[Thm. 0.4]{MR3317764}. If the base field $k$ is perfect, one
can show that the ad\`{e}les decompose as a kind of restricted product of
TLFs. Yekutieli's cubically decomposed algebra then turns out to be isomorphic
to Beilinson's. See \cite{bgwGeomAnalAdeles} for details.
\end{example}

\begin{example}
Higher infinite matrix algebras also carry a cubically decomposed structure.
This is probably the simplest non-trivial example \cite{MR3207578}.
\end{example}

So how can we connect Beilinson's Theorem with the category $\mathsf{Tate}%
^{el}\mathcal{C}$?

\section{Operator ideals in Tate categories}

First of all, we will show that the condition of Definition
\ref{def_HigherAdeleOperatorIdeals}, part (1), naturally comes up in the
context of lattices of Tate objects. This requires some preparation. We need
to establish some features which would be entirely obvious if we dealt with
$k$-vector spaces and the notion of lattices from Definition
\ref{def_BeilLattice}.

It was shown in \cite{TateObjectsExactCats} that Pro-objects are left
filtering in $\mathsf{Tate}^{el}(\mathcal{C})$ and Ind-objects right
filtering. The following result strengthens these two facts:

\begin{proposition}
\label{Prop_LatticeLeftFiltAndRightFilt}Let $\mathcal{C}$ be an exact category.

\begin{enumerate}
\item Every morphism $Y\overset{a}{\longrightarrow}X$ in $\mathsf{Tate}%
^{el}(\mathcal{C})$ with $Y\in\mathsf{Pro}^{a}(\mathcal{C})$ can be factored
as $Y\overset{\tilde{a}}{\rightarrow}L\hookrightarrow X$ with $L$ a lattice in
$X$.

\item Every morphism $X\overset{a}{\longrightarrow}Y$ in $\mathsf{Tate}%
^{el}(\mathcal{C})$ with $Y\in\mathsf{Ind}^{a}(\mathcal{C})$ can be factored
as $X\twoheadrightarrow X/L\overset{\tilde{a}}{\rightarrow}Y$ with $L$ a
lattice in $X$.
\end{enumerate}
\end{proposition}

\begin{proof}
A complete proof is given in \cite[Proposition 2.7]{bgwRelativeTateObjects}.
\end{proof}

\begin{lemma}
\label{Lemma_LatticeInsideLattice}Suppose $\mathcal{C}$ is an exact category.
Let $X,X^{\prime}\in\mathsf{Tate}^{el}\mathcal{C}$ and $\varphi\in
\operatorname*{Hom}(X,X^{\prime})$.

\begin{enumerate}
\item For every lattice $L\hookrightarrow X$ there exists a lattice
$L^{\prime}\hookrightarrow X^{\prime}$ admitting a factorization as depicted
below on the left.

\item For every lattice $L^{\prime}\hookrightarrow X^{\prime}$ there exists a
lattice $L\hookrightarrow X$ admitting a factorization as depicted on the
right.%
\[%
%
%
\bfig\square/^{ (}->`.>`>`^{ (}.>/[L`X`L^{\prime}`X^{\prime};`{\varphi
}^{\prime}`{\varphi}`]
\efig
\qquad\qquad%
%
%
\bfig\square/^{ (}.>`.>`>`^{ (}->/[L`X`L^{\prime}`X^{\prime};`{\varphi
}^{\prime}`{\varphi}`]
\efig
\]

\end{enumerate}
\end{lemma}

\begin{proof}
This immediately follows from Prop. \ref{Prop_LatticeLeftFiltAndRightFilt}.
(1) Here $L\rightarrow X^{\prime}$ is a morphism from a Pro-object, so it
factors through a lattice $L^{\prime}$ of $X^{\prime}$. (2) Here $X\rightarrow
X^{\prime}/L^{\prime}$ is a morphism to an Ind-object, so it factors
$X\twoheadrightarrow X/L\rightarrow X^{\prime}/L^{\prime}$ for some lattice
$L$.
\end{proof}

\begin{lemma}
[Cartesian sandwich]\label{Lemma_CartesianSandwich}Suppose $\mathcal{C}$ is
idempotent complete. Let $X_{1},X_{2}\in\mathsf{Tate}^{el}(\mathcal{C})$ and
$L\hookrightarrow X_{1}\oplus X_{2}$ a lattice. Then there exist lattices
$L_{i}^{\prime}\subseteq L_{i}$ of $X_{i}$ so that%
\[
L_{1}^{\prime}\oplus L_{2}^{\prime}\subseteq L\subseteq L_{1}\oplus
L_{2}\text{.}%
\]

\end{lemma}

\begin{proof}
Consider the composed morphism $L\hookrightarrow X_{1}\oplus X_{2}%
\twoheadrightarrow X_{i}$. By Prop. \ref{Prop_LatticeLeftFiltAndRightFilt}
this factors through a lattice of $X_{i}$, say $L_{i}\hookrightarrow X_{i}$.
Now we already have%
\[%
%
%
\bfig\btriangle/>`^{ (}->`^{ (}->/[L`L_{1}\oplus L_{2}`X_{1}\oplus X_{2};``]
\efig
\]
By \cite[Lemma 6.9]{TateObjectsExactCats} the downward arrow is an admissible
monic (this crucially makes use of the assumption that $\mathcal{C}$ is
idempotent complete). Dually, consider the composition $X_{i}\hookrightarrow
X_{1}\oplus X_{2}\twoheadrightarrow\left(  X_{1}\oplus X_{2}\right)  /L$.
Since Ind-objects are right filtering \cite[Prop. 5.10 (2)]%
{TateObjectsExactCats}, there must be a lattice $L_{i}^{\prime}$ so that the
map factors as $X_{i}\twoheadrightarrow X_{i}/L_{i}^{\prime}\rightarrow\left(
X_{1}\oplus X_{2}\right)  /L$. Thus, the composition $L_{1}^{\prime}\oplus
L_{2}^{\prime}\hookrightarrow X_{1}\oplus X_{2}\twoheadrightarrow\left(
X_{1}\oplus X_{2}\right)  /L$ is zero and by the universal property of kernels
we get a canonical morphism $L_{1}^{\prime}\oplus L_{2}^{\prime}\rightarrow
L$. Again by \cite[Lemma 6.9]{TateObjectsExactCats} the corresponding morphism
of the lattice quotients must be an admissible epic, so that this is an
admissible monic.
\end{proof}

\begin{definition}
\label{Def_BoundedDiscreteFiniteDimOne}Let $\mathcal{C}$ be an exact category.
For objects $X,X^{\prime}\in\mathsf{Tate}^{el}\mathcal{C}$ call a morphism
$\varphi:X\rightarrow X^{\prime}$

\begin{enumerate}
\item \emph{bounded} if there exists a lattice $L^{\prime}\subseteq X^{\prime
}$ so that $\varphi$ factors as $X\rightarrow L^{\prime}\hookrightarrow
X^{\prime}$;

\item \emph{discrete} if there exists a lattice $L\subseteq X$ so that
$L\hookrightarrow X\overset{\varphi}{\rightarrow}X^{\prime}$ is the zero morphism;

\item \emph{finite} if it is both bounded and discrete.
\end{enumerate}

Denote by $I^{+}\left(  X,X^{\prime}\right)  $, $I^{-}\left(  X,X^{\prime
}\right)  $, $I^{0}\left(  X,X^{\prime}\right)  $ the subsets of
$\operatorname*{Hom}\left(  X,X^{\prime}\right)  $ of bounded, discrete and
finite morphisms respectively.
\end{definition}

\begin{lemma}
\label{Lemma_PropertiesBoundedOrDiscreteMorphisms}Suppose $\mathcal{C}$ is
idempotent complete and $s\in\{+,-,0\}$. Then for arbitrary objects
$X,X^{\prime},X^{\prime\prime}$ the following are true:

\begin{enumerate}
\item $I^{s}\left(  X,X^{\prime}\right)  $ is a subgroup with respect to addition.

\item $I^{s}\left(  X,X^{\prime}\right)  $ is a categorical ideal, i.e. the
composition of any morphism with a morphism in $I^{s}$ lies in $I^{s}$. Thus,
the composition of morphisms factors as
\begin{align*}
I^{s}\left(  X^{\prime},X^{\prime\prime}\right)  \otimes\operatorname*{Hom}%
\left(  X,X^{\prime}\right)   & \longrightarrow I^{s}\left(  X,X^{\prime
\prime}\right) \\
\operatorname*{Hom}\left(  X^{\prime},X^{\prime\prime}\right)  \otimes
I^{s}\left(  X,X^{\prime}\right)   & \longrightarrow I^{s}\left(
X,X^{\prime\prime}\right)  \text{.}%
\end{align*}

\item In the ring $\operatorname*{End}(X)$ the subgroup $I^{s}\left(
X,X\right)  $ is a two-sided ideal.

\item Every morphism in $I^{0}\left(  X,X^{\prime}\right)  $ factors through a
morphism from an Ind- to a Pro-object. Every product of at least two morphisms
in $I^{0}$ factors through an object in $\mathcal{C}$.
\end{enumerate}
\end{lemma}

\begin{example}
\label{Example_NeedLongWordsToHaveFiniteRank}Let $\mathcal{C}:=\mathsf{Vect}%
_{f}$ be the category of finite-dimensional $k$-vector spaces. Define
$X:=k[t]\oplus k[[t]]$ and an endomorphism%
\[
\varphi:k[t]\oplus k[[t]]\longrightarrow k[t]\oplus k[[t]]\text{,}%
\qquad(a,b)\longmapsto(0,a)\text{.}%
\]
It is easy to see that $\varphi\in I^{0}(X,X)$, but $\varphi$ does not factor
through an object in $\mathcal{C}$. This shows that Lemma
\ref{Lemma_PropertiesBoundedOrDiscreteMorphisms} (4) cannot be strengthened in
this alluring fashion. Note that this phenomenon is already present in Tate's
original work \cite{MR0227171}. It is the reason why he has to work with
`finite-potent' morphisms rather than finite rank ones.
\end{example}

\begin{proof}
(1) Suppose $\varphi_{1},\varphi_{2}\in I^{+}\left(  X,X^{\prime}\right)  $
are given and they factor as%
\[
\varphi_{1}:X\rightarrow L_{1}^{\prime}\hookrightarrow X^{\prime}%
\qquad\text{and}\qquad\varphi_{2}:X\rightarrow L_{2}^{\prime}\hookrightarrow
X^{\prime}\text{.}%
\]
Then by the directedness of the Sato Grassmannian \cite[Thm. 6.7]%
{TateObjectsExactCats} there exists a lattice $L_{3}^{\prime}$ so that
$L_{i}^{\prime}\subseteq L_{3}^{\prime}$ for $i=1,2$ and thus without loss of
generality we may assume $L_{i}^{\prime}=L_{3}^{\prime}$ in the above
factorizations, so the claim is clear. The same works for $I^{-}(X,X^{\prime
})$ by taking a common sub-lattice. By $I^{0}\left(  X,X^{\prime}\right)
=I^{+}\left(  X,X^{\prime}\right)  \cap I^{-}\left(  X,X^{\prime}\right)  $ it
also works for $I^{0}$. (2) For $\varphi\in I^{+}$ and $\psi$ arbitrary it is
trivial that $\varphi\circ\psi$ factors through a lattice, namely the same as
$\varphi$ does. In the reverse direction $\psi\circ\varphi$ we have a
factorization as depicted on the left in%
\[%
%
%
\bfig\dtriangle/<-`^{ (}->`>/[L^{\prime}`X`X^{\prime};``\varphi]
\square(500,0)/```>/[``X^{\prime}`X^{\prime\prime};```\psi]
\efig
\qquad\qquad%
%
%
\bfig\dtriangle/<-`^{ (}->`>/[L^{\prime}`X`X^{\prime};``\varphi]
\square(500,0)/.>``^{ (}.>`>/[L^{\prime}`L^{\prime\prime}`X^{\prime}%
`X^{\prime\prime};{\psi\mid_{L^{\prime}}}```\psi]
\efig
\]
since $\varphi\in I^{+}$. Then by Lemma \ref{Lemma_LatticeInsideLattice} (1)
there exists a lattice $L^{\prime\prime}$ in $X^{\prime\prime}$ so that we get
a further factorization as depicted above on the right. Thus, we have a
factorization $X\rightarrow L^{\prime\prime}\hookrightarrow X^{\prime\prime}$
with $L^{\prime\prime}$ a lattice, so $\psi\circ\varphi$ is also bounded. For
$\varphi\in I^{-}$ and $\psi$ arbitrary, the proof is analogous. This time
$\psi\circ\varphi$ trivially sends a lattice to zero, namely the same one as
$\varphi$ and the reverse direction $\varphi\circ\psi$ requires an argument,
very analogous to the above one for $I^{+}$: Let $L^{\prime}$ be a lattice
which is sent to zero by $\varphi$ as shown left in%
\[%
%
%
\bfig\square/```>/[``X`X^{\prime};```\psi]
\btriangle(500,0)/^{ (}->`>`>/[L^{\prime}`X^{\prime}`X^{\prime\prime
};`0`\varphi]
\efig
\qquad\qquad%
%
%
\bfig\square/.>`^{ (}.>``>/[L`L^{\prime}`X`X^{\prime};{\psi\mid_{L }}```\psi]
\btriangle(500,0)/^{ (}->`>`>/[L^{\prime}`X^{\prime}`X^{\prime\prime
};`0`\varphi]
\efig
\]
According to Lemma \ref{Lemma_LatticeInsideLattice} (2) there exists a lattice
$L$ in $X$ so that we can complete the diagram as depicted on the right. Hence
$\varphi\circ\psi$ sends a lattice to zero. (3) is immediate from (2). For (4)
let $X\overset{\varphi}{\longrightarrow}X^{\prime}$ be a morphism in
$I^{0}\left(  X,X^{\prime}\right)  $. By boundedness we find a factorization
through a lattice $L^{\prime}$ so that we get the diagram on the left%
\[%
%
%
\bfig\square/``^{ (}->`>/[`L^{\prime}`X`X^{\prime};```\varphi]
\morphism(0,0)/>/<500,500>[X`L^{\prime};]
\efig
\qquad\qquad%
%
%
\bfig\square/.>`{ (}->`^{ (}->`>/[L`L^{\prime}`X`X^{\prime};```\varphi]
\morphism(0,0)/>/<500,500>[X`L^{\prime};]
\efig
\]
and by discreteness a lattice $L$ so that $L\rightarrow X\rightarrow
X^{\prime}$ is zero, as depicted on the right. Since $L^{\prime}%
\hookrightarrow X^{\prime}$ is a monomorphism, the induced upper horizontal
arrow must be the zero map itself. By the universal property of kernels this
yields a factorization%
\[%
%
%
\bfig\Dtrianglepair/{ (}->`>`>`->>`<-/[L`X`L^{\prime}`X/L;`0`\varphi``]
\efig
\]
Thus, we have obtained a factorization $X\twoheadrightarrow X/L\rightarrow
L^{\prime}\hookrightarrow X^{\prime}$ as desired. If we compose any two
morphisms, we may equivalently look at the composition of these
factorizations,%
\[
X\twoheadrightarrow X/L\rightarrow L^{\prime}\hookrightarrow X^{\prime
}\twoheadrightarrow X^{\prime}/L^{\prime\prime}\rightarrow L^{\prime\prime
}\hookrightarrow X^{\prime\prime\prime}\text{,}%
\]
then $L^{\prime}\rightarrow X^{\prime}/L^{\prime\prime}$ is a morphism from a
Pro-object to an Ind-object. Since Pro-objects are left filtering in
$\mathsf{Tate}^{el}\mathcal{C}$, this factors through a Pro-sub-object of
$X^{\prime}/L^{\prime\prime}$, which is therefore also an Ind-object. However,
by \cite[Prop. 5.9]{TateObjectsExactCats} an object can only be simultaneously
an Ind- and Pro-object if it actually lies in $\mathcal{C}$.
\end{proof}

\begin{definition}
Let $X$ be an elementary Tate object. If there exists a lattice
$i:L\hookrightarrow X$ which admits a splitting $s:X\rightarrow L$ so that
$si=\operatorname*{id}_{L}$, we call $X$ a \emph{sliced Tate object}.
\end{definition}

The following result is the categorical analogue of the decomposition used by
J. Tate in his original article \cite[Prop. 1]{MR0227171}.

\begin{proposition}
\label{Prop_SlicedTateObjIdealsSumUp}Let $\mathcal{C}$ be an exact category.

\begin{enumerate}
\item If $X$ is a sliced elementary Tate object, $\operatorname*{End}\left(
X\right)  =I^{+}\left(  X,X\right)  +I^{-}\left(  X,X\right)  $. More
generally, there is a short exact sequence of abelian groups%
\[
0\longrightarrow I^{0}(X,X)\longrightarrow I^{+}(X,X)\oplus I^{-}%
(X,X)\longrightarrow\operatorname*{End}(X)\longrightarrow0\text{,}%
\]
functorial in morphisms between sliced elementary Tate objects.

\item If $\mathcal{C}$ is split exact and idempotent complete, every
elementary Tate object in $\mathsf{Tate}_{\aleph_{0}}^{el}\mathcal{C}$ is
sliced. In particular, each $\operatorname*{End}(X)$ is a unital one-fold
cubical algebra in the sense of Definition \ref{def_BeilNFoldAlg}.
\end{enumerate}
\end{proposition}

\begin{proof}
(1) Define $P^{+}:=is:X\rightarrow L\hookrightarrow X$. This is clearly a
bounded morphism. Define $P^{-}:=\operatorname*{id}_{X}-P^{+}$. We find that
after precomposing by $L\hookrightarrow X$ the two morphisms $X\overset
{\operatorname*{id}}{\rightarrow}X$, $X\overset{P^{+}}{\rightarrow}X$ agree,
thus $P^{-}\mid_{L}=0$, i.e. $P^{-}$ is a discrete morphism. Finally,
$\operatorname*{id}_{X}=P^{+}+P^{-}$ by construction. Thus, any morphism
$\varphi\in\operatorname*{End}\left(  X\right)  $ can be written as
$\varphi=P^{+}\varphi+P^{-}\varphi$ and since bounded and discrete morphisms
form ideals, by Lemma \ref{Lemma_PropertiesBoundedOrDiscreteMorphisms},
$P^{+}\varphi$ is bounded and $P^{-}\varphi$ discrete. (2) By \cite[Prop.
5.23]{TateObjectsExactCats} the split exactness of $\mathcal{C}$ implies that
$\mathsf{Tate}_{\aleph_{0}}^{el}\mathcal{C}$ is also split exact. Hence, we
can pick any lattice (always exists), obtain $L\hookrightarrow X$, and the
split exactness enforces the existence of a splitting.
\end{proof}

\begin{example}
\label{Example_FinAbelianPGroups}Let $p$ be a prime number and $\mathcal{C}$
the abelian category of finite abelian $p$-groups. We shall show that
$\mathsf{Tate}^{el}\mathcal{C}$ contains both sliced and non-sliced objects.
Specifically, both%
\[
\text{\textquotedblleft}\mathbf{F}_{p}((t))\text{\textquotedblright}%
=\underset{i}{\underrightarrow{\operatorname*{colim}}}\underset{j}%
{\underleftarrow{\lim}}\,t^{-i}\mathbf{F}_{p}[t]/t^{j}\qquad\text{and}%
\qquad\text{\textquotedblleft}\mathbf{Q}_{p}\text{\textquotedblright%
}=\underset{i}{\underrightarrow{\operatorname*{colim}}}\underset
{j}{\underleftarrow{\lim}}\,p^{-i}\mathbf{Z}/p^{j}%
\]
are elementary Tate objects.

\begin{enumerate}
\item The former is sliced via $\mathbf{F}_{p}((t))\simeq\mathbf{F}%
_{p}[[t]]\oplus t^{-1}\mathbf{F}_{p}[t^{-1}]$ while the latter is not. To see
this, note that $\mathsf{Ind}^{a}\mathcal{C}$ is equivalent to the category of
$p$-primary torsion abelian groups. Hence, if there exists a splitting
$\mathbf{Q}_{p}\simeq I\oplus P$ with $I\in\mathsf{Ind}^{a}\mathcal{C}$, we
must have $I=0$ since $\mathbf{Q}_{p}$ has no non-trivial torsion elements at
all. Thus, $\mathbf{Q}_{p}\in\mathsf{Pro}^{a}\mathcal{C} $, forcing that
$\mathbf{Q}_{p}/\mathbf{Z}_{p}\in\mathcal{C}$, but this is clearly absurd.

\item Proposition \ref{Prop_SlicedTateObjIdealsSumUp} fails for $\mathbf{Q}%
_{p}$. Suppose not. Then there exist $p\in I^{+}(\mathbf{Q}_{p},\mathbf{Q}%
_{p})$ and $q\in I^{-}(\mathbf{Q}_{p},\mathbf{Q}_{p})$ so that
$\operatorname*{id}=p+q$. Then $pq$ and $qp$ lie in $I^{0}(\mathbf{Q}%
_{p},\mathbf{Q}_{p})$, so by Lemma
\ref{Lemma_PropertiesBoundedOrDiscreteMorphisms} they both factor through a
morphism from an Ind- to a Pro-object. So they factor through torsion
elements. Thus, $pq=qp=0$. As a result, $p=p(p+q)=p^{2}$ and analogously for
$q$. So these must be idempotents. Thus, in the idempotent completion we get a
direct sum splitting $\mathbf{Q}_{p}\simeq\operatorname*{im}p\oplus
\operatorname*{im}q$. Since $q$ kills a lattice, say $L$, the map
$\mathbf{Q}_{p}\overset{q}{\twoheadrightarrow}\operatorname*{im}q$ descends to
$\mathbf{Q}_{p}/L\twoheadrightarrow\operatorname*{im}q$, forcing
$\operatorname*{im}q$ to be an Ind-object and therefore zero. Again, we obtain
$\mathbf{Q}_{p}\in\mathsf{Pro}^{a}\mathcal{C}$, which is absurd.
\end{enumerate}
\end{example}

We recall that in an additive category a morphism $p$ is an epic if for any
composition%
\[
X\overset{p}{\longrightarrow}Y\overset{f}{\longrightarrow}Z
\]
which is zero, $f$ must already have been zero. Now suppose we want to find a
definition for `locally epic'. Then lattices take over the r\^{o}le of a basis
of open neighbourhoods of the neutral element. Hence, it makes sense to use
the definition of epic morphisms, but restrict both the assumption as well as
the conclusion to lattices. This leads to the following concept.

\begin{definition}
Let $p:X\rightarrow Y$ be a morphism of elementary Tate objects.

\begin{enumerate}
\item We call $p$ \emph{submersive} if for any morphism $f$ and lattice
$L\hookrightarrow X$ so that the diagonal arrow in%
\[%
%
%
\bfig\node l(0,0)[L]
\node x(500,0)[X]
\node y(500,500)[Y]
\node z(1400,500)[Z]
\node ll(0,500)[L^{\prime}]
\arrow/^{ (}->/[l`x;]
\arrow[x`y;p]
\arrow[y`z;f]
\arrow/-->/[l`z;0]
\arrow/^{ (}.>/[ll`y;]
\efig
\]
is zero, there exists a lattice $L^{\prime}\hookrightarrow Y$ (drawn with a
dotted arrow) so that $L^{\prime}\hookrightarrow Y\rightarrow Z$ is
zero.\newline$\left.  \qquad\right.  $\emph{(Slogan: \textquotedblleft
vanishing on a lattice can be pushed forward\textquotedblright)}

\item Symmetrically, call $p$ \emph{immersive} if for any morphism $f$ and
lattice $L\hookrightarrow X$ so that the diagonal arrow in
\[%
%
%
\bfig\node l(0,0)[X/L]
\node x(500,0)[X]
\node y(500,500)[Y]
\node z(1400,500)[Z]
\node ll(0,500)[Y/{L^{\prime}}]
\arrow/->>/[x`l;]
\arrow[y`x;p]
\arrow[z`y;f]
\arrow/-->/[z`l;0]
\arrow/.>>/[y`ll;]
\efig
\]
is zero, there exists a lattice $L^{\prime}\hookrightarrow Y$ (whose quotient
is drawn with a dotted arrow) so that $Z\rightarrow Y\twoheadrightarrow
Y/L^{\prime}$ is zero.\newline$\left.  \qquad\right.  $\emph{(Slogan:
\textquotedblleft vanishing modulo a lattice can be pulled
back\textquotedblright)}
\end{enumerate}
\end{definition}

These two definitions are almost dual. One transforms one into the other by
going to the opposite category and interchanging Ind- and Pro-objects.

\begin{lemma}
\label{Lemma_AdmMonicsAreImmersiveAdmEpicsSubmersive}Let $\mathcal{C}$ be an
idempotent complete exact category.

\begin{enumerate}
\item Every admissible monic $p:Y\hookrightarrow X$ is immersive.

\item Every admissible epic $p:X\twoheadrightarrow Y$ is submersive.
\end{enumerate}
\end{lemma}

For example for an arbitrary lattice in an elementary Tate object, the
inclusion is immersive and the respective quotient morphism submersive:%
\[%
%
%
\bfig\morphism(0,0)/^{ (}->/<600,0>[L`X;\mathrm{immersive}]
\morphism(600,0)/->>/<600,0>[X`X/L;\mathrm{submersive}]
\efig
\]
We shall show in Example \ref{example_NonAdmissibleMonoNotImmersive} that the
lemma can fail if we remove the word `admissible'.

For a morphism $f:X\rightarrow Y$ in ${\mathsf{Tate}^{el}(\mathcal{C})}$, and
$L\hookrightarrow X$ a lattice, the notation $f(L)=0$ is shorthand for the
statement that the diagram
\[
\xymatrix{ L \ar[r] \ar[d] & X \ar[d] \\ 0 \ar[r] & Y }
\]
commutes. As a first step towards the proof of the lemma, say for
$p:X\twoheadrightarrow Y$ epic, we observe that for $Y\in{\mathsf{Ind}%
^{a}(\mathcal{C})}$ the statement is automatically true, since then we have
that $0\hookrightarrow Y$ is a lattice, and we certainly have $g(0)=0$.

The general case relies on the following lemma.

\begin{lemma}
\label{lemma:fact}Let $\mathcal{C}$ be an idempotent complete exact category.
Let $g:M\rightarrow N$ be a morphism of admissible Pro-objects $M$,
$N\in\mathsf{Pro}^{a}(\mathcal{C})$, which is sent to the zero morphism by the
exact functor ${\mathsf{Pro}^{a}(\mathcal{C})}\rightarrow{\mathsf{Pro}%
^{a}(\mathcal{C})}/{\mathcal{C}}$. Then there exists a commutative triangle
\[
\xymatrix{ & U \ar[rd] & \\ M \ar[rr] \ar@{->>}[ru]^g & & N, }
\]
where $U\in\mathcal{C}$, and $g$ is an admissible epic in ${\mathsf{Pro}%
^{a}(\mathcal{C})}$.
\end{lemma}

\begin{proof}
We use that $\mathcal{C}\subset{\mathsf{Pro}^{a}(\mathcal{C})}$ is right
s-filtering \cite[Prop. 4.2. (2)]{TateObjectsExactCats}. By an observation of
B\"{u}hler this implies that the class $\Sigma_{m}$ of admissible
monomorphisms in ${\mathsf{Pro}^{a}(\mathcal{C})}$ with cokernel in
$\mathcal{C}$ satisfies a calculus of right fractions, see \cite[Prop.
2.19]{TateObjectsExactCats} for the broader context. Moreover, we also know
from B\"{u}hler that ${\mathsf{Pro}^{a}(\mathcal{C})}[\Sigma_{m}^{-1}%
]\cong{\mathsf{Pro}^{a}(\mathcal{C})}/{\mathcal{C}}$, \cite[Prop.
2.19]{TateObjectsExactCats}. Since $g:M\rightarrow N$ and $0:M\rightarrow N$
induce the same map in ${\mathsf{Pro}^{a}(\mathcal{C})}/{\mathcal{C}}$, we see
that there exists a commutative diagram
\[
\xymatrix{ & M \ar[ld] \ar[rd]^g & \\ M & M' \ar[l] \ar[u]^h \ar[r] \ar[d] & N \\ & M, \ar[lu] \ar[ru]_0 & }
\]
where $h\colon M^{\prime}\hookrightarrow M$ is an admissible monic with
cokernel $Q\in\mathcal{C}$. The commutativity of the diagram implies that the
horizontal arrow $M^{\prime}\rightarrow N$ is zero. Therefore, we obtain by
the universal property of cokernels a factorization
\[
\xymatrix{ & Q \ar[rd] & \\ M \ar[rr]^g \ar@{->>}[ru] & & N }
\]
as required to conclude the proof of the assertion.
\end{proof}

We are now ready to prove that admissible epimorphisms are submersive.

\begin{proof}
[Proof of Lemma \ref{Lemma_AdmMonicsAreImmersiveAdmEpicsSubmersive}]We shall
only treat the case of an admissible epic, and leave the necessary
modifications for the monic case to the reader. By Lemma
\ref{Lemma_LatticeInsideLattice} we have a commutative diagram
\[
\xymatrix{ L \ar[r] \ar[d] & X \ar[d]^f \\ M \ar[r] \ar[d] & Y \ar[d]^g \\ N \ar[r] & Z }
\]
where the horizontal arrows are inclusions of lattices. We also know that the
inclusion ${\mathsf{Ind}^{a}(\mathcal{C})}\hookrightarrow{\mathsf{Tate}%
^{el}(\mathcal{C})}$ is right s-filtering \cite[Cor. 2.3]{bgwTensor}, and the
quotient category is equivalent to ${\mathsf{Pro}^{a}(\mathcal{C}%
)}/{\mathcal{C}}$ \cite[Prop. 5.34]{TateObjectsExactCats}. Inclusions of
lattices are sent to isomorphisms in ${\mathsf{Tate}^{el}(\mathcal{C}%
)}/{\mathsf{Ind}^{a}(\mathcal{C})}$. Hence, we obtain that the composition
$f\circ p$ is sent to $0$ in ${\mathsf{Tate}^{el}(\mathcal{C})}/{\mathsf{Ind}%
^{a}(\mathcal{C})}$. However, exact functors send admissible epimorphisms to
admissible epimorphisms; and every admissible epimorphism is an epimorphism in
the categorical sense. The relation $f\circ p=0\circ p$ in ${\mathsf{Tate}%
^{el}(\mathcal{C})}/{\mathsf{Ind}^{a}(\mathcal{C})}$ implies now that $f=0$ in
${\mathsf{Tate}^{el}(\mathcal{C})}/{\mathsf{Ind}^{a}(\mathcal{C})}$.

We have shown above that the morphism $M\rightarrow N$ in ${\mathsf{Pro}%
^{a}(\mathcal{C})}$ is sent to $0$ in ${\mathsf{Pro}^{a}(\mathcal{C}%
)}/{\mathcal{C}}$. By Lemma \ref{lemma:fact}, this yields a factorization
$M\twoheadrightarrow Q\rightarrow N$ with $Q\in\mathcal{C}$. Let $L^{\prime} $
be the kernel of the admissible epimorphism $M\twoheadrightarrow Q$. By
construction $L^{\prime}\subset Y$ is a lattice, and $f(L^{\prime})=0$. This
concludes the proof.
\end{proof}

\begin{example}
A submersive morphism does not need to be an epimorphism. For example, for
$\mathcal{C}:=\mathsf{Vect}_{f}$ the zero morphism $k[[t]]\overset
{0}{\rightarrow}k[t]$ is a submersion. This makes sense topologically since we
would think of $k[t]$ as having zero-dimensional tangent spaces. It is however
also a finite immersion, which appears rather strange from the point of view
of topological intuition.
\end{example}

\begin{example}
\label{example_NonAdmissibleMonoNotImmersive}Let us construct a
(non-admissible!) monomorphism which is not immersive. Let $\mathcal{C}%
:=\mathsf{Vect}_{f}$ be the category of finite-dimensional $k$-vector spaces.
We have a morphism%
\[
\varphi:k[t]\longrightarrow k((t))\text{,}%
\]
the obvious inclusion. This morphism is monic, but it is not an admissible
monic since otherwise a Pro-object would have an Ind-object as a sub-object.
We claim that this morphism is not immersive. Suppose it is. Take
$Z\rightarrow Y$ to be the identity $\operatorname*{id}\nolimits_{k[t]}$ and
$L:=k[[t]]$, which is clearly a lattice in $k((t))$. The immersion property
now implies that $k[t]$ must be a lattice in itself. In particular, it must be
a Pro-object, which is absurd.
\end{example}

\begin{lemma}
\label{Lemma_OpenMorProps}Submersive morphisms have the following properties:

\begin{enumerate}
\item Isomorphisms are submersive.

\item The composition of submersive morphisms is submersive.
\end{enumerate}
\end{lemma}

\begin{proof}
(1) is trivial, just transport the lattice along the isomorphism. (2) Let
$p,q$ be composable submersive morphisms. Let $f$ be an arbitrary morphism and
$L$ a lattice that gets sent to zero by $(f\circ q)\circ p$, i.e. the lower
diagonal arrow in%
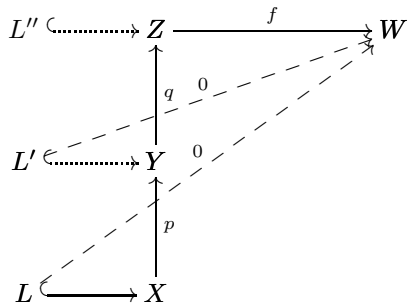
\[%
%
%
\bfig\node l(0,0)[L]
\node x(500,0)[X]
\node y(500,500)[Y]
\node z(1400,1000)[W]
\node w(500,1000)[Z]
\node ll(0,500)[L^{\prime}]
\node lll(0,1000)[L^{\prime\prime}]
\arrow/^{ (}->/[l`x;]
\arrow[x`y;p]
\arrow[y`w;q]
\arrow[w`z;f]
\arrow/-->/[l`z;0]
\arrow/-->/[ll`z;0]
\arrow/^{ (}.>/[ll`y;]
\arrow/^{ (}.>/[lll`w;]
\efig
\]
The submersiveness of $p$ (for the morphism $f\circ q$) guarantees the
existence of a lattice $L^{\prime}$ so that $f\circ q$ sends it to zero. Now
the submersiveness of $q$ yields the existence of a lattice $L^{\prime\prime}$
which is being sent to zero by $f$. But this is all we had to show.
\end{proof}

\begin{lemma}
\label{Lemma_CoOpenMorProps}Immersive morphisms have the following properties:

\begin{enumerate}
\item Isomorphisms are immersive.

\item The composition of immersive morphisms is immersive.
\end{enumerate}
\end{lemma}

\begin{proof}
The proof is essentially dual to the proof of Lemma \ref{Lemma_OpenMorProps},
just reverse the direction of all arrows.
\end{proof}

\begin{lemma}
[{\cite[Prop. 5.23]{TateObjectsExactCats}}]%
\label{Lemma_CIdemComplSplitExactAleph0}If $\mathcal{C}$ is idempotent
complete and split exact, $\mathsf{Tate}_{\aleph_{0}}^{el}\mathcal{C}$ is
split exact.
\end{lemma}

\begin{lemma}
\label{Lemma_LiftDiscreteOrBoundedness}Suppose we are given one of the squares%
\[%
%
%
\bfig\square(0,0)/>`>`<-{) }`>/<600,600>[X_1`X_2`Y_1`Y_2;\mathrm
{discrete}`\mathrm{submersive}``f]
\efig
\qquad\text{resp.}\qquad%
%
%
\bfig\square(0,0)/>`->>`<-`>/<600,600>[X_1`X_2`Y_1`Y_2;\mathrm{bounded}%
``\mathrm{immersive}`g]
\efig
\]
Then $f$ is discrete (resp. $g$ bounded).
\end{lemma}

For this statement to be true the monic (resp. epic) would need not be admissible.

\begin{proof}
If $X_{1}\rightarrow X_{2}$ is discrete, there is a lattice $L$ so that the
upper row in%
\[%
%
%
\bfig\morphism(0,600)/^{ (}->/<600,0>[L`X_1;]
\square(600,0)/>`>`<-{) }`>/<600,600>[X_1`X_2`Y_1`Y_2;\mathrm{discrete}%
`\mathrm{submersive}``g]
\efig
\]
is the zero morphism. Since the right-hand side upward arrow is a
monomorphism, it follows that $L\hookrightarrow X_{1}\rightarrow
Y_{1}\rightarrow Y_{2}$ must already be zero. Now being submersive implies
that there is a lattice $L^{\prime}$ in $Y_{1}$ so that $L^{\prime
}\hookrightarrow Y_{1}\overset{g}{\rightarrow}Y_{2}$ is zero. Hence, $g$ is
discrete. The argument for the other square is dual.
\end{proof}

We collect a few more useful properties.

\begin{lemma}
\label{Lemma_InterplayOpenDiscrete}Suppose $\mathcal{C}$ is idempotent complete.

\begin{enumerate}
\item If $p:X\rightarrow Y$ is submersive, either no lattice $L\hookrightarrow
X$ is sent to zero, or $Y$ is an Ind-object.

\item Submersive discrete morphisms are precisely the morphisms $X\rightarrow
Y$ with $Y$ an Ind-object.

\item If $p:X\rightarrow Y$ is immersive, either it does not factor through
any lattice in $Y$, or $X$ is a Pro-object.

\item Immersive bounded morphisms are precisely the morphisms $X\rightarrow Y
$ with $X$ a Pro-object.
\end{enumerate}
\end{lemma}

\begin{proof}
(1) If a lattice $L$ exists that $p$ sends to zero, being submersive gives a
lattice $L^{\prime}$ in%
\[%
%
%
\bfig\node l(0,0)[L]
\node x(500,0)[X]
\node y(500,500)[Y]
\node z(1400,500)[Y]
\node ll(0,500)[L^{\prime}]
\arrow/^{ (}->/[l`x;]
\arrow[x`y;p]
\arrow[y`z;1]
\arrow/-->/[l`z;0]
\arrow/^{ (}.>/[ll`y;]
\efig
\]
so that $L^{\prime}\hookrightarrow Y$ is the zero map. So the zero object is a
lattice, which forces $Y$ to be an Ind-object. (2)\ if $p:X\rightarrow Y $ is
discrete, a lattice is sent to zero, so just use (1). Conversely, if $Y$ is an
Ind-object, by Prop. \ref{Prop_LatticeLeftFiltAndRightFilt} the morphism $p$
factors through a lattice quotient $p:X\twoheadrightarrow X/L\rightarrow Y$.
In particular $p$ sends $L$ to zero and so $p$ is discrete. As $Y$ is an
Ind-object, the zero object is a lattice, so $p$ is clearly submersive. (3)
and (4) are dual.
\end{proof}

Finally, we can show that the boundedness of a morphism is preserved under
passing to sub-objects or quotients, and analogously for discreteness and finiteness.

\begin{proposition}
\label{Prop_IfMorphismFactorsPreservesBasicProperties}Suppose $\mathcal{C}$ is
idempotent complete. Let $\varphi:X\rightarrow X$ be a bounded (resp.
discrete, finite) morphism and $Y\hookrightarrow X$ an admissible monic such
that $\varphi\mid_{Y}$ factors over $Y$, i.e.%
\begin{equation}%
%
%
\bfig\node x(0,0)[X]
\node xp(500,0)[X]
\node n(0,500)[Y]
\node np(500,500)[Y]
\arrow/^{ (}->/[n`x;]
\arrow/>/[x`xp;\varphi]
\arrow/^{ (}->/[np`xp;]
\arrow/>/[n`np;\varphi]
\efig
\label{laa1}%
\end{equation}
Then

\begin{enumerate}
\item the restriction $\varphi\mid_{Y}:Y\rightarrow Y$ is also bounded (resp.
discrete, finite), and

\item the quotient map $\overline{\varphi}:X/Y\rightarrow X/Y$ is also bounded
(resp. discrete, finite).
\end{enumerate}

If $\varphi$ is discrete and $L^{\prime}\hookrightarrow L\hookrightarrow X$
are lattices so that $\varphi$ factors as%
\[
\varphi:L/L^{\prime}\longrightarrow L/L^{\prime}\text{,}%
\]
then there exist lattices $L_{1}^{\prime}\hookrightarrow L_{1}\hookrightarrow
Y$ and $L_{2}^{\prime}\hookrightarrow L_{2}\hookrightarrow X/Y$ so that
$\varphi\mid_{Y}$ and $\overline{\varphi}$ factor as%
\begin{align*}
\varphi\mid_{Y}:L_{1}/L_{1}^{\prime}  & \longrightarrow L_{1}/L_{1}^{\prime}\\
\overline{\varphi}:L_{2}/L_{2}^{\prime}  & \longrightarrow L_{2}/L_{2}%
^{\prime}%
\end{align*}
and%
\[
L_{1}/L_{1}^{\prime}\hookrightarrow L/L^{\prime}\twoheadrightarrow L_{2}%
/L_{2}^{\prime}%
\]
is short exact.
\end{proposition}

\begin{proof}
\textit{(1, Bounded)} As $\varphi:X\rightarrow X$ is bounded, it factors over
a lattice, say $L$. Thus, $Y\hookrightarrow X\overset{\varphi}{\rightarrow}X$
factors over $L$ in the target, but by the commutativity of Diagram \ref{laa1}
this means that $Y\overset{\varphi}{\rightarrow}Y\hookrightarrow X$ factors
over $L$ in the target. Hence, we get the diagram%
\[%
%
%
\bfig\node l(0,0)[X/L]
\node x(500,0)[X]
\node y(500,500)[Y]
\node z(1400,500)[Y]
\node ll(0,500)[Y/{L^{\prime}}]
\arrow/->>/[x`l;]
\arrow/^{ (}->/[y`x;p]
\arrow[z`y;\varphi]
\arrow/-->/[z`l;0]
\arrow/.>>/[y`ll;]
\efig
\]
By Lemma \ref{Lemma_AdmMonicsAreImmersiveAdmEpicsSubmersive} the admissible
monic $p$ is immersive. Thus, a lattice $L^{\prime}$ as in the above diagram
exists, showing that $\varphi\mid_{L}$ is bounded.\newline\textit{(1,
Discrete)} This is simpler. As $\varphi:X\rightarrow X$ is discrete, there
exists a lattice $L\hookrightarrow X$ so that $L\hookrightarrow X\overset
{\varphi}{\rightarrow}X$ is zero. By Lemma \ref{Lemma_LatticeInsideLattice}
there exists a lattice $L^{\prime}\hookrightarrow Y$ such that under
$Y\hookrightarrow X$ it maps to $L\hookrightarrow X$, and then the composition%
\[
L^{\prime}\rightarrow L\rightarrow X\rightarrow X
\]
is zero. Thus, by commutativity $L^{\prime}\hookrightarrow Y\rightarrow
Y\hookrightarrow X$ is zero, and by the defining property of monics, the
composition $L^{\prime}\hookrightarrow Y\rightarrow Y$ must already be zero.
Since $L^{\prime}$ is a lattice, it follows that $\varphi\mid_{L}$ is
discrete.\newline\textit{(1, Finite)} Just combine both statements.\newline%
\textit{(2, Bounded)} Consider the commutative diagram%
\[%
%
%
\bfig\node x(0,0)[X/Y]
\node xp(500,0)[X/Y.]
\node n(0,500)[X]
\node np(500,500)[X]
\arrow/>>/[n`x;]
\arrow/>/[x`xp;{\overline{\varphi}}]
\arrow/>>/[np`xp;]
\arrow/>/[n`np;\varphi]
\efig
\]
As $\varphi$ is bounded, there exists a lattice $L\hookrightarrow X$ so that
$\varphi:X\rightarrow L\hookrightarrow X$. By Lemma
\ref{Lemma_LatticeInsideLattice} there exists a lattice $L^{\prime
}\hookrightarrow X/Y$ so that $X\twoheadrightarrow X/Y$ restricted to $L$
factors over $L^{\prime}\hookrightarrow X/Y$. In other words, $X\overset
{\varphi}{\rightarrow}X\twoheadrightarrow X/Y\twoheadrightarrow(X/Y)/L^{\prime
}$ is zero. By the commutativity of the diagram,%
\[
X\twoheadrightarrow X/Y\overset{\overline{\varphi}}{\rightarrow}%
X/Y\twoheadrightarrow(X/Y)/L^{\prime}%
\]
must be zero as well. Since the first morphism is an epic, we deduce that
$X/Y\overset{\overline{\varphi}}{\rightarrow}X/Y\twoheadrightarrow
(X/Y)/L^{\prime}$ is already the zero map. By the universal property of
cokernels, this means that there is a factorization $\overline{\varphi
}:X/Y\rightarrow L^{\prime}\hookrightarrow X/Y$, i.e. $\overline{\varphi}$ is
bounded.\newline\textit{(2, Discrete)} As $\varphi$ is discrete, there exists
a lattice $L\hookrightarrow X$ such that $L\hookrightarrow X\overset{\varphi
}{\rightarrow}X$ is zero. Thus, we obtain that the diagonal morphism in%
\[%
%
%
\bfig\node l(0,0)[L]
\node x(500,0)[X]
\node y(500,500)[X/Y]
\node z(1400,500)[X/Y]
\node ll(0,500)[L^{\prime}]
\arrow/^{ (}->/[l`x;]
\arrow/>>/[x`y;p]
\arrow[y`z;{\overline{\varphi}}]
\arrow/-->/[l`z;0]
\arrow/^{ (}.>/[ll`y;]
\efig
\]
is zero. Following Lemma \ref{Lemma_AdmMonicsAreImmersiveAdmEpicsSubmersive}
the admissible epic $p$ is submersive, i.e. there exists a lattice $L^{\prime
}\hookrightarrow X/Y$ such that $L^{\prime}\hookrightarrow X/Y\overset
{\overline{\varphi}}{\rightarrow}X/Y$ is zero. But this just means that
$\overline{\varphi}$ is discrete, too.\newline\textit{(2, Finite)} Just
combine the last two statements.\newline\textit{(Lattices)} Finally, combine
the above constructions for a discrete morphism and lattices $L^{\prime
}\hookrightarrow L\hookrightarrow X$ such that $\varphi$ factors as%
\[
\varphi:L/L^{\prime}\longrightarrow L/L^{\prime}\text{.}%
\]
We see that they construct lattices $L_{1}^{\prime},L_{1}\hookrightarrow Y$ so
that $L_{1}\hookrightarrow L$ and $L_{1}^{\prime}\hookrightarrow L^{\prime}$
under $Y\hookrightarrow X$; without loss of generality use the
(co-)directedness of the Sato Grassmannian \cite[Thm. 6.7]%
{TateObjectsExactCats} to achieve that $L_{1}^{\prime}\hookrightarrow L_{1}$
holds, by replacing $L_{1}$ by a common over-lattice of the two constructed
lattices if necessary. Proceed similarly for the quotient $X/Y$.
\end{proof}

\section{\label{sect_GeneralTateObjects}General Tate objects}

In this section we extend (in a trivial way) the previous definitions to
non-elementary Tate objects.\medskip

Let $\mathcal{C}$ be an exact category. We recall that its idempotent
completion $\mathcal{C}^{ic}$ is the category whose objects are pairs $(X,p)$
with $X\in\mathcal{C}$ and $p:X\rightarrow X$ with $p^{2}=p$ an idempotent.
Its morphisms are%
\begin{align}
\operatorname*{Hom}\nolimits_{\mathcal{C}^{ic}}((X,p),(Y,q))  & =\left\{
f\in\operatorname*{Hom}\nolimits_{\mathcal{C}}(X,Y)\mid qfp=f\right\}
\label{lcw1}\\
& =\left\{  f\mid\exists g\in\operatorname*{Hom}\nolimits_{\mathcal{C}%
}(X,Y)\text{ so that }f=qgp\right\}  \text{.}\nonumber
\end{align}
We refer to \cite[\S 6]{MR2606234} or \cite[Ch. II]{MR2791358} for a detailed
construction and basic properties of the idempotent completion. Recall that
$\mathsf{Tate}(\mathcal{C}):=(\mathsf{Tate}^{el}\mathcal{C})^{ic} $. We will
now define all basic types of morphisms between general Tate objects by simply
requiring that the morphism of the underlying elementary Tate objects has the
relevant property. This is, by the way, the same mechanism which is employed
to equip $\mathcal{C}^{ic}$ with an exact structure: A kernel-cokernel
sequence in $\mathcal{C}^{ic}$ is called exact iff it is a direct summand of
an exact sequence in $\mathcal{C}$. See \cite[Prop. 6.13]{MR2606234} for more
on this. In particular, admissible monics and epics in $\mathcal{C}^{ic}$ are
represented by admissible monics and epics in $\mathcal{C}$.

\begin{lemma}
\label{Lemma_Tate_ICOfSplitExactIsSplitExact}Let $\mathcal{C}$ be a split
exact category and $\mathcal{C}^{ic}$ its idempotent completion. Then
$\mathcal{C}^{ic}$ is also split exact.
\end{lemma}

\begin{proof}
Suppose $0\rightarrow A\rightarrow B\rightarrow C\rightarrow0$ is an exact
sequence in $\mathcal{C}^{ic}$. Then by definition \cite[\S 6, cf. Prop.
6.13]{MR2606234} it arises as a direct summand of an exact sequence in
$\mathcal{C}$, viewed as a sequence in $\mathcal{C}^{ic}$. Thus, there is an
exact sequence in $\mathcal{C}$ so that%
\[
0\longrightarrow A\oplus A^{\prime}\overset{i}{\longrightarrow}B\oplus
B^{\prime}\overset{j}{\longrightarrow}C\oplus C^{\prime}\longrightarrow0
\]
is exact in $\mathcal{C}$. Since $\mathcal{C}$ is split exact, there exists a
left splitting $\pi:B\oplus B^{\prime}\rightarrow A\oplus A^{\prime}$ so that
$\pi i=1$. It is now easy to check that $B\rightarrow B\oplus B^{\prime
}\overset{\pi}{\rightarrow}A\oplus A^{\prime}\rightarrow A$, where the outer
arrows are the inclusion and projection from the direct summands, is a left
splitting of the original exact sequence.
\end{proof}

\begin{definition}
For objects $X,X^{\prime}\in\mathsf{Tate}(\mathcal{C})$ we say that
$\varphi:X\rightarrow X^{\prime}$ is

\begin{enumerate}
\item bounded,

\item discrete,

\item finite,

\item immersive,

\item submersive,
\end{enumerate}

if, when unwinding the definition of idempotent completion, we have $X=(Y,p)$
and $X^{\prime}=(Y^{\prime},q)$ and $\varphi:Y\rightarrow Y^{\prime}$ (so that
$q\varphi p=\varphi$) is a morphism of elementary Tate objects so that
$\varphi$ has the named property.
\end{definition}

\begin{lemma}
\label{Lemma_BasicLemmataForGeneralTateObjects}Let $\mathcal{C}$ be an
idempotent complete exact category. In $\mathsf{Tate}(\mathcal{C})$

\begin{enumerate}
\item Lemma \ref{Lemma_PropertiesBoundedOrDiscreteMorphisms} remains valid,
i.e. bounded, discrete and finite morphisms form categorical ideals,

\item Lemma \ref{Lemma_AdmMonicsAreImmersiveAdmEpicsSubmersive} remains valid,
i.e. admissible monics (resp. epics) are immersive (resp. submersive) as before.

\item Lemma \ref{Lemma_OpenMorProps} remains valid, i.e. submersions behave as before.

\item Lemma \ref{Lemma_CoOpenMorProps} remains valid, i.e. immersions behave
as before.

\item Lemma \ref{Lemma_CIdemComplSplitExactAleph0} remains valid, i.e. if
$\mathcal{C}$ is split exact, $\mathsf{Tate}_{\aleph_{0}}\mathcal{C}$ is split exact.

\item Lemma \ref{Lemma_LiftDiscreteOrBoundedness} remains valid, i.e. given%
\[%
%
%
\bfig\square(0,0)/>`>`<-{) }`>/<600,600>[X_1`X_2`Y_1`Y_2;\mathrm
{discrete}`\mathrm{submersive}``f]
\efig
\qquad\text{resp.}\qquad%
%
%
\bfig\square(0,0)/>`->>`<-`>/<600,600>[X_1`X_2`Y_1`Y_2;\mathrm{bounded}%
``\mathrm{immersive}`g]
\efig
\]
$f$ is discrete (resp. $g$ bounded).
\end{enumerate}
\end{lemma}

\begin{proof}
Nothing really happens. We give some details nonetheless: (1) The ideal
property follows from the corresponding property for elementary Tate objects
since $\operatorname*{Hom}\nolimits_{\mathcal{C}^{ic}}(X,X^{\prime})$ is a
subgroup of $\operatorname*{Hom}\nolimits_{\mathcal{C}}(Y,Y^{\prime})$ for
$X=:(Y,p)$ and $X^{\prime}=:(Y^{\prime},q)$, Equation \ref{lcw1}. (2), (3),
(4) similar. (5) Use Lemma \ref{Lemma_Tate_ICOfSplitExactIsSplitExact} for
split exactness. For (6) note that we have such squares in $\mathsf{Tate}%
(\mathcal{C})$ only if they come from a square of elementary Tate objects with
morphisms with the same properties, so Lemma
\ref{Lemma_LiftDiscreteOrBoundedness} applies to this square, implying that
$f$ is discrete in $\mathsf{Tate}^{el}\mathcal{C}$ and then so is in
$\mathsf{Tate}(\mathcal{C})$. Analogously for $g$.
\end{proof}

\begin{remark}
Clearly our approach is based on drawing parallels to similar concepts in
functional analysis. For example our notion of bounded morphisms is not too
remote from the concept of a compact operator. The same remark applies to
trace-class operators. The idea to look at higher local fields, i.e. special
cases of $n$-Tate objects over vector spaces, from a functional analytic
perspective has already been pursued in the work of A. C\'{a}mara
\cite{MR3161556} and \cite{MR3227342}.
\end{remark}

\section{Cubical structure}

In Beilinson's definition, that is Definition
\ref{def_HigherAdeleOperatorIdeals}, an interesting continuity condition
appears. One looks at all $k$-linear maps \textquotedblleft such that for all
lattices $L_{1}\hookrightarrow M_{1},L_{2}\hookrightarrow M_{2}$ there exist
lattices $L_{1}^{\prime}\hookrightarrow M_{1},L_{2}^{\prime}\hookrightarrow
M_{2}$ such that%
\[
L_{1}^{\prime}\hookrightarrow L_{1},\qquad L_{2}\hookrightarrow L_{2}^{\prime
},\qquad f(L_{1\triangle^{\prime}}^{\prime})\hookrightarrow L_{2\triangle
^{\prime}},\qquad f(L_{1\triangle^{\prime}})\hookrightarrow L_{2\triangle
^{\prime}}^{\prime}%
\]
holds.\textquotedblright

In order to relate this to Tate objects, we first need to show that the very
definition of morphisms of Tate objects implies this kind of behaviour
automatically. This is not entirely obvious from the outset due to the rather
different style of definition of lattices:

\begin{lemma}
\label{Lemma_CanFactorTateMorThroughLatticePairs}Suppose $\mathcal{C}$ is an
idempotent complete exact category and $X_{1},X_{2}\in\mathsf{Tate}%
^{el}\mathcal{C}$. Let $f\in\operatorname*{Hom}\left(  X_{1},X_{2}\right)  $
be an arbitrary morphism. For all lattices $L_{1}\hookrightarrow X_{1}%
,L_{2}\hookrightarrow X_{2}$ there exist lattices $L_{1}^{\prime}%
,L_{2}^{\prime}$ and a \emph{double lattice factorization}%
\begin{equation}%
%
%
\bfig\node a(0,0)[L_1^{\prime}]
\node b(0,400)[L_1]
\node c(0,800)[X_1]
\node g(800,0)[L_2]
\node h(800,400)[L_2^{\prime}]
\node i(800,800)[X_2]
\arrow/.>/[a`g;{\exists}]
\arrow/.>/[b`h;{\exists}]
\arrow/>/[c`i;f]
\arrow/{ (}->/[a`b;{}]
\arrow/{ (}->/[b`c;{}]
\arrow/{ (}->/[g`h;{}]
\arrow/{ (}->/[h`i;{}]
\efig
\label{lFullFactorization}%
\end{equation}
and for all such $L_{1},L_{1}^{\prime},L_{2},L_{2}^{\prime}$ we get an induced
morphism%
\[
\overline{f}:L_{1}/L_{1}^{\prime}\rightarrow L_{2}^{\prime}/L_{2}%
\qquad\text{in}\qquad\operatorname*{Hom}\nolimits_{\mathcal{C}}\left(
L_{1}/L_{1}^{\prime},L_{2}^{\prime}/L_{2}\right)
\]
in the category $\mathcal{C}$.
\end{lemma}

We keep the notation $\overline{f}$ for later use.

\begin{proof}
From the assumptions we just get the diagram depicted on the left in:%
\[%
%
%
\bfig\node b(0,400)[L_1]
\node c(0,800)[X_1]
\node g(800,0)[L_2]
\node i(800,800)[X_2]
\arrow/>/[c`i;f]
\arrow/{ (}->/[b`c;{}]
\arrow/{ (}->/[g`i;{}]
\efig
\qquad\qquad%
%
%
\bfig\node b(0,400)[L_1]
\node c(0,800)[X_1]
\node g(800,0)[L_2]
\node h(800,400)[L_2^{\prime}]
\node i(800,800)[X_2]
\arrow/.>/[b`h;{\exists}]
\arrow/>/[c`i;f]
\arrow/{ (}->/[b`c;{}]
\arrow/{ (}->/[g`h;{}]
\arrow/{ (}->/[h`i;{}]
\efig
\]
By Lemma \ref{Lemma_LatticeInsideLattice} (1) the restriction $f\mid_{L_{1}} $
factors through some lattice of $X_{2}$, say $\tilde{L}_{2}$. By the
directedness of the Sato Grassmannian \cite[Thm. 6.7]{TateObjectsExactCats} we
can find a common over-lattice of both $\tilde{L}_{2}$ and $L_{2}$, call it
$L_{2}^{\prime}$, so that we arrive at the diagram on the right. By Lemma
\ref{Lemma_LatticeInsideLattice} (2) there exists some lattice $\tilde{L}_{1}$
of $X_{1}$ so that $f\mid_{\tilde{L}_{1}}$ factors through $L_{2}$. By the
codirectedness of the Sato Grassmannian \cite[Thm. 6.7]{TateObjectsExactCats}
we can find a common sub-lattice of both $L_{1}$ and $\tilde{L}_{1}$, call it
$L_{1}^{\prime}$, so that we arrive at the Diagram \ref{lFullFactorization}.
Finally, this induces a canonical morphism $\overline{f}:L_{1}/L_{1}^{\prime
}\longrightarrow L_{2}^{\prime}/L_{2}$ and by \cite[Prop. 6.6]%
{TateObjectsExactCats} quotients of nested lattices lie in the base category,
i.e. both source and target of $\overline{f}$ lie in the sub-category
$\mathcal{C}$.
\end{proof}

Later, we will need to understand how the composition of morphisms leads to
the composition of such induced morphisms $\overline{f}$. In order to do this,
we need to be able to find intermediate double lattice factorizations. The
best we can hope for in this direction is the following existence result:

\begin{lemma}
\label{Lemma_FindIntermediateDoubleLatticePair}Suppose $\mathcal{C}$ is
idempotent complete. Let $X_{1}\overset{f}{\longrightarrow}X_{2}\overset
{g}{\longrightarrow}X_{3}$ be arbitrary morphisms between elementary Tate
objects. Then for every double lattice factorization as in Diagram
\ref{lFullFactorization} for the composite $g\circ f$ we can find lattices
$\tilde{L}_{1}$ in $X_{1}$, $L_{2},L_{2}^{\prime}$ in $X_{2}$ and $\tilde
{L}_{3}$ in $X_{3}$ so that
\begin{equation}%
%
%
\bfig\node b(0,0)[L_1/{L_1^{\prime}}]
\node c(0,400)[L_1/{\tilde{L}_1}]
\node f(600,400)[L_2/{L_2^{\prime}}]
\node h(1200,0)[{L_3^{\prime}}/L_3]
\node i(1200,400)[{\tilde{L}_3}/L_3]
\arrow/>/[b`h;{\overline{g\circ f}}]
\arrow/>/[c`f;{\overline{f}}]
\arrow/>/[f`i;{\overline{g}}]
\arrow/<<-/[b`c;{}]
\arrow/{ (}->/[h`i;{}]
\efig
\label{lFullFactorization2}%
\end{equation}
commutes.
\end{lemma}

\begin{proof}
For the beginning, let $f,g$ be arbitrary morphisms. Suppose we are given a
double lattice factorization for $g\circ f$, i.e.%
\[%
%
%
\bfig\node a(0,0)[L_1^{\prime}]
\node b(0,400)[L_1]
\node c(0,800)[X_1]
\node g(800,0)[L_3.]
\node h(800,400)[L_3^{\prime}]
\node i(800,800)[X_3]
\arrow/>/[a`g;{ }]
\arrow/>/[b`h;{ }]
\arrow/>/[c`i;g \circ f]
\arrow/{ (}->/[a`b;{}]
\arrow/{ (}->/[b`c;{}]
\arrow/{ (}->/[g`h;{}]
\arrow/{ (}->/[h`i;{}]
\efig
\]
In general there is no reason why it should be possible to factor the two
lower horizontal arrows over lattices in $X_{2}$. Thus, we first need to
refine a given factorization. Using Lemma \ref{Lemma_LatticeInsideLattice} (1)
there exists a lattice $L_{2}$ in $X_{2}$, and\ (using the Lemma again) a
lattice $\tilde{L}_{3}$ so that the diagram depicted below on the left
commutes:%
\[%
%
%
\bfig\node a(0,0)[L_1^{\prime}]
\node b(0,250)[L_1]
\node bb(500,500)[L_2]
\node c(0,750)[X_1]
\node f(500,750)[X_2]
\node bbb(1000,500)[\tilde{L}_{3}]
\node g(1000,0)[L_3]
\node h(1000,250)[L_3^{\prime}]
\node i(1000,750)[X_3]
\arrow/>/[a`g;{}]
\arrow/>/[b`h;{}]
\arrow/>/[c`f;f]
\arrow/>/[f`i;g]
\arrow/.>/[b`bb;{}]
\arrow/.>/[bb`bbb;{}]
\arrow/{ (}.>/[bb`f;{}]
\arrow/{ (}.>/[bbb`i;{}]
\arrow/{ (}->/[a`b;{}]
\arrow/{ (}->/[b`c;{}]
\arrow/{ (}->/[g`h;{}]
\arrow/{ (}->/[h`bbb;{}]
\efig
\qquad\qquad%
%
%
\bfig\node a(0,0)[L_1^{\prime}]
\node b(0,250)[L_1]
\node bb(500,500)[L_2]
\node c(0,750)[X_1]
\node f(500,750)[X_2]
\node bbb(1000,500)[\tilde{L}_{3}]
\node g(1000,0)[L_3]
\node h(1000,250)[L_3^{\prime}]
\node i(1000,750)[X_3]
\node ee(500,-250)[L_2^{\prime}]
\node eee(0,-250)[\tilde{L}_1]
\arrow/>/[a`g;{}]
\arrow/>/[b`h;{}]
\arrow/>/[c`f;f]
\arrow/>/[f`i;g]
\arrow/>/[b`bb;{}]
\arrow/>/[bb`bbb;{}]
\arrow/.>/[ee`g;{}]
\arrow/.>/[eee`ee;{}]
\arrow/{ (}->/[bb`f;{}]
\arrow/{ (}->/[bbb`i;{}]
\arrow/{ (}->/[a`b;{}]
\arrow/{ (}->/[b`c;{}]
\arrow/{ (}->/[g`h;{}]
\arrow/{ (}->/[h`bbb;{}]
\arrow/{ (}.>/[ee`bb;{}]
\arrow/{ (}.>/[eee`a;{}]
\efig
\]
Here we may have without loss of generality replaced $\tilde{L}_{3}$ in the
diagram by a common over-lattice of $L_{3}^{\prime}$ and $\tilde{L}_{3}$ so
that the diagram still commutes (use directedness of the Sato Grassmannian).
Now consider the bottom horizontal arrow in this diagram. Analogous to the
previous refinement, using Lemma \ref{Lemma_LatticeInsideLattice} (2) we find
a lattice $L_{2}^{\prime}$ in $X_{2}$ which (after possibly replacing
$L_{2}^{\prime}$ by a common sub-lattice with $L_{2}$) fits in the diagram
depicted above on the right. Repeating this step again for $\tilde{L}_{1}$
yields the full diagram on the right. Taking quotients we get Diagram
\ref{lFullFactorization2}.
\end{proof}

The following definition is a fairly precise imitation (even regarding the
naming of the variables) of the continuity condition employed by Beilinson in
his ad\`{e}le paper, compare with Definition
\ref{def_HigherAdeleOperatorIdeals}, or see the original paper \cite{MR565095}.

\begin{definition}
\label{Def_CubicalStructureForTateObjects}Suppose $\mathcal{C}$ is idempotent
complete. Let $X_{1},X_{2}\in\left.  n\text{-}\mathsf{Tate}^{el}%
\mathcal{C}\right.  $ be elementary $n$-Tate objects.

\begin{enumerate}
\item Let $I_{1}^{s}\left(  X_{1},X_{2}\right)  $ for $s\in\{+,-,0\}$ denote
the bounded, discrete and finite morphisms in $\operatorname*{Hom}\left(
X_{1},X_{2}\right)  $ respectively, exactly as in Definition
\ref{Def_BoundedDiscreteFiniteDimOne}.

\item For $i=2,\ldots,n$ let $I_{i}^{s}\left(  X_{1},X_{2}\right)  $ denote
the morphisms $f\in\operatorname*{Hom}\left(  X_{1},X_{2}\right)  $ such that
for all lattices $L_{1},L_{1}^{\prime},L_{2},L_{2}^{\prime}$ and double
lattice factorizations as in Diagram \ref{lFullFactorization} we have%
\[
\overline{f}\in I_{(i-1)}^{s}(L_{1}/L_{1}^{\prime},L_{2}^{\prime}%
/L_{2})\text{.}%
\]

\item We define%
\[
I_{tr}\left(  X_{1},X_{2}\right)  :=\bigcap_{i=1,\ldots,n}I_{i}^{0}\left(
X_{1},X_{2}\right)  \text{,}%
\]
its elements will be called \emph{trace-class} morphisms.
\end{enumerate}
\end{definition}

As in \S \ref{sect_GeneralTateObjects} this immediately implies a reasonable
definition for general (non-elementary) Tate objects:

\begin{definition}
If $(X_{1},p_{1})$ and $(X_{2},p_{2})$ are general Tate objects, define
$I_{i}^{s}\left(  X_{1},X_{2}\right)  $ to consist of those morphisms
$f:(X_{1},p_{1})\rightarrow(X_{2},p_{2})$ such that the underlying morphism of
elementary Tate objects $X_{1}\rightarrow X_{2}$ lies in $I_{i}^{s}\left(
X_{1},X_{2}\right)  $ in the above sense.
\end{definition}

\begin{theorem}
\label{theorem_CIdempCompleteHaveTateIdeals}Suppose $\mathcal{C}$ is
idempotent complete and $X,X^{\prime},X^{\prime\prime}\in\left.
n\text{-}\mathsf{Tate}^{el}\mathcal{C}\right.  $ or $\left.  n\text{-}%
\mathsf{Tate}(\mathcal{C})\right.  $.

\begin{enumerate}
\item The $I_{i}^{s}\left(  -,-\right)  $ for $i=1,\ldots,n$ are categorical
ideals. This means that the composition of morphisms factors as
\begin{align*}
I_{i}^{s}\left(  X^{\prime},X^{\prime\prime}\right)  \otimes
\operatorname*{Hom}\left(  X,X^{\prime}\right)   & \longrightarrow I_{i}%
^{s}\left(  X,X^{\prime\prime}\right) \\
\operatorname*{Hom}\left(  X^{\prime},X^{\prime\prime}\right)  \otimes
I_{i}^{s}\left(  X,X^{\prime}\right)   & \longrightarrow I_{i}^{s}\left(
X,X^{\prime\prime}\right)  \text{.}%
\end{align*}

\item In the ring $\operatorname*{End}(X)$ the $I_{i}^{s}\left(  X,X\right)  $
are two-sided ideals.

\item Every composition of $\geq2^{n}$ morphisms from $I_{tr}\left(
-,-\right)  $ factors through an object in $\mathcal{C}$. For words in $<2^{n}
$ letters this is in general false.
\end{enumerate}
\end{theorem}

The following argument is very close in spirit to the handling of lattices by
A. Yekutieli in \cite{MR3317764}. However, we encounter a number of additional
technical issues because of the less concrete notion of lattice we work with.
There is also a similar study in the case of vector spaces and $n$-local
fields by D. V. Osipov \cite{MR2314612}.

\begin{proof}
(1) We only show this for elementary Tate objects since the general case
follows directly along the same lines as the proofs in
\S \ref{sect_GeneralTateObjects}. We will reduce this to the case of a single
Tate category, notably Lemma \ref{Lemma_PropertiesBoundedOrDiscreteMorphisms}.
Let $f,g$ be composable morphisms as depicted in the top row of the diagram
below. In order to prove that $g\circ f$ lies in $I_{i}^{s}$ (for some $s$ and
$i$), the condition to check reduces to proving a property $g\circ f$ induced
to a certain iterated subquotient of lattices. The lattice subquotients arise
from an inductive choice. More precisely: Starting with $m:=1$, consider any
double lattice factorization of the composition as in diagram
\begin{equation}%
%
%
\bfig\node a(0,0)[L_1^{\prime}]
\node b(0,250)[L_1]
\node c(0,500)[X_1]
\node f(500,500)[X_2]
\node g(1000,0)[L_3]
\node h(1000,250)[L_3^{\prime}]
\node i(1000,500)[X_3]
\arrow/>/[a`g;{}]
\arrow/>/[b`h;{}]
\arrow/>/[c`f;f]
\arrow/>/[f`i;g]
\arrow/{ (}->/[a`b;{}]
\arrow/{ (}->/[b`c;{}]
\arrow/{ (}->/[g`h;{}]
\arrow/{ (}->/[h`i;{}]
\efig
\label{lDoubleLatticeFactorizationD}%
\end{equation}
with $X_{1},X_{2},X_{3}$ being elementary $\left(  n-m+1\right)  $-Tate
objects. By Lemma \ref{Lemma_FindIntermediateDoubleLatticePair} we can
construct a commutative diagram%
\[%
%
%
\bfig\node b(0,0)[L_1/{L_1^{\prime}}]
\node c(0,400)[L_1/{\tilde{L}_1}]
\node f(600,400)[L_2/{L_2^{\prime}}]
\node h(1200,0)[{L_3^{\prime}}/L_3]
\node i(1200,400)[{\tilde{L}_3}/L_3]
\arrow/>/[b`h;{\overline{g\circ f}}]
\arrow/>/[c`f;{\overline{f}}]
\arrow/>/[f`i;{\overline{g}}]
\arrow|m|/<<-/[b`c;{\mathrm{submersive}}]
\arrow|m|/{ (}->/[h`i;{\mathrm{immersive}}]
\efig
\]
where the left and right outer arrows are an admissible epic (resp. monic) and
thus are submersive (resp. immersive) by Lemma
\ref{Lemma_BasicLemmataForGeneralTateObjects}. Now continue with a picking
another double lattice factorization as in Equation
\ref{lDoubleLatticeFactorizationD}, but this time with $m_{new}:=m_{old}+1$
and using the top row of the above diagram in place of $X_{1}\overset
{f}{\longrightarrow}X_{2}\overset{g}{\longrightarrow}X_{3}$. Note that the
objects in this new row are quotients of nested lattices, so by \cite[Prop.
6.6]{TateObjectsExactCats} they are elementary $\left(  n-m\right)  $-Tate
objects. Repeat this until we reach $m=i$. For the rest of the proof
$\overline{g\circ f}$ will refer to the respective morphism coming from the
last step in this inductive procedure, i.e. when $m=i$. In particular,
$\overline{g\circ f}$ is a morphism between $\left(  n-i\right)  $-Tate
objects and from now on the word lattice will only refer to lattices in such.
No more interplay of lattices of varying Tate categories will be needed, let
us also rename the entries of the above diagram into neutral terms%
\[%
%
%
\bfig\node b(0,0)[Z_1]
\node c(0,400)[W_1]
\node f(600,400)[A]
\node h(1200,0)[Z_2]
\node i(1200,400)[W_2]
\arrow/>/[b`h;{\overline{g\circ f}}]
\arrow/>/[c`f;{\overline{f}}]
\arrow/>/[f`i;{\overline{g}}]
\arrow|m|/<<-/[b`c;{\mathrm{submersive}}]
\arrow|m|/{ (}->/[h`i;{\mathrm{immersive}}]
\efig
\]
Now by assumption one of $\overline{f}$ or $\overline{g}$ lies in $I^{s}$, so
by Lemma \ref{Lemma_PropertiesBoundedOrDiscreteMorphisms} the entire top row
lies in $I^{s}$. Then by Lemma \ref{Lemma_LiftDiscreteOrBoundedness} it
follows that the bottom row lies in $I^{s}$ as well. (2) trivially follows
from (1). For (3) first note that it suffices to show this for elementary
$n$-Tate objects. Now we show the claim by induction on $n$. For $n=1$ Lemma
\ref{Lemma_PropertiesBoundedOrDiscreteMorphisms} gives the claim. Hence,
assume the case $n-1$ has been dealt with and suppose $f_{j}\in I_{tr}\left(
-,-\right)  $ for $j=1,\ldots,2^{n}$ are given and composable so that
$f_{1}\circ\cdots\circ f_{2^{n}}$ makes sense. By a minimal variation of the
argument for Lemma \ref{Lemma_PropertiesBoundedOrDiscreteMorphisms} there is a
factorization of $f_{j}\circ f_{j+1}$ as%
\begin{equation}
X_{1}\twoheadrightarrow X_{1}/L\overset{\overline{f_{j}\circ f_{j+1}}%
}{\longrightarrow}L^{\prime}\hookrightarrow X_{2}\text{,}\label{lcw83}%
\end{equation}
where $L$ is a lattice in $X_{1}$ and $L^{\prime}$ a lattice in $X_{2}$.
Following the argument of Lemma
\ref{Lemma_PropertiesBoundedOrDiscreteMorphisms} further, the composition of
any two morphisms having a factorization as in Equation \ref{lcw83}, factors
through an object in $\left(  n-1\right)  $-$\mathsf{Tate}(\mathcal{C})$.
Thus, for every second index there is a factorization $\overline{f_{1}f_{2}%
},\overline{f_{3}f_{4}},\overline{f_{5}f_{6}},\ldots:X_{\ast}\rightarrow
C_{\ast}\rightarrow X_{\ast}$ with `$\ast$' replaced by suitable indices and
with $C_{j}\in\left(  n-1\right)  $-$\mathsf{Tate}(\mathcal{C})$. Now if we
compose these $2^{n}/2=2^{n-1}$ morphisms, by induction it factors over an
object in $\mathcal{C}$. To see that one cannot do with less than $2^{n}$
morphisms, we ask the reader to adapt Example
\ref{Example_NeedLongWordsToHaveFiniteRank} accordingly.
\end{proof}

\begin{definition}
\label{def_GoodIdempotents}Suppose we are given $A:=\operatorname*{End}(X)$ in
the situation of Theorem \ref{theorem_CIdempCompleteHaveTateIdeals}. Pairwise
commuting elements $P_{i}^{+}\in A$ (with $i=1,\ldots,n$) such that the
following conditions are met:

\begin{itemize}
\item $P_{i}^{+2}=P_{i}^{+}$.

\item $P_{i}^{+}A\subseteq I_{i}^{+}$.

\item $P_{i}^{-}A\subseteq I_{i}^{-}\qquad$(and we define $P_{i}%
^{-}:=\mathbf{1}_{A}-P_{i}^{+}$)
\end{itemize}

will be called a \emph{system of good idempotents}. We shall call an
(elementary) $n$-Tate object $n$\emph{-sliced} if $A=\operatorname*{End}(X)$
admits a system of good idempotents.
\end{definition}

A very explicit example for good idempotents will be given in Example
\ref{example_LaurentSeriesOverRHaveVeryEasyTateStructure}.

\begin{proposition}
\label{cor_SlicedTateObjectsHaveEndNFoldCubical}For every $n$-sliced object
$X\in\left.  n\text{-}\mathsf{Tate}^{el}(\mathcal{C})\right.  $ or $\left.
n\text{-}\mathsf{Tate}(\mathcal{C})\right.  $ we have%
\begin{equation}
I_{i}^{+}\left(  X,X\right)  +I_{i}^{-}\left(  X,X\right)
=\operatorname*{End}(X)\label{lcwx1}%
\end{equation}
for all $i=1,\ldots,n$. Moreover, $\operatorname*{End}(X)$ is a Beilinson
$n$-fold cubical algebra as in Definiton \ref{def_BeilNFoldAlg}.
\end{proposition}

\begin{proof}
By Theorem \ref{theorem_CIdempCompleteHaveTateIdeals} we have the necessary
ideals $I_{i}^{\pm}$. In order to meet all axioms of Definiton
\ref{def_BeilNFoldAlg}, it suffices to prove Equation \ref{lcwx1}. However,
this can be done using the idempotents, by an immediate generalization of the
proof of Prop. \ref{Prop_SlicedTateObjIdealsSumUp}.
\end{proof}

\begin{lemma}
\label{lemma_TraceClassOfObjectInCIsAll}If $X\in\mathcal{C}$ then every
endomorphism is trace-class, i.e.%
\[
I_{tr}(X)=\operatorname*{End}\nolimits_{\left.  n\text{-}\mathsf{Tate}%
(\mathcal{C})\right.  }(X)=\operatorname*{End}\nolimits_{\mathcal{C}%
}(X)\text{.}%
\]

\end{lemma}

\begin{proof}
The first equality holds since \textit{every} sub-object of $X\in\mathcal{C} $
is a lattice with respect to the top $1$-Tate structure, $\left.
n\text{-}\mathsf{Tate}(\mathcal{C})\right.  =\mathsf{Tate}(\left.
(n-1)\text{-}\mathsf{Tate}(\mathcal{C})\right.  )$. Then for all quotients
$N/N^{\prime}$ of such lattices $N^{\prime}\hookrightarrow N\hookrightarrow
X$, we still have $N/N^{\prime}\in\mathcal{C}$. Thus, inductively, every
endomorphism is trace-class. Then second equality holds since the embeddings
$\mathcal{C}\hookrightarrow\left.  \mathsf{Tate}(\mathcal{C})\right.  $ are
all fully faithful.
\end{proof}

\section{The countable split exact case}

With the previous results we have seriously approached arriving at a structure
as in Definition \ref{def_BeilNFoldAlg}. Suppose $\mathcal{C}$ is an
idempotent complete exact category. Now suppose $X$ is an elementary $n$-Tate
object. We may present it as%
\begin{equation}
X=\underset{L_{1}}{\underrightarrow{\operatorname*{colim}}}\underset
{L_{1}^{\prime}}{\underleftarrow{\lim}}\,\frac{L_{1}}{L_{1}^{\prime}}%
\text{,}\label{lcab1}%
\end{equation}
where $L_{1}^{\prime}\hookrightarrow L_{1}\hookrightarrow X$ are a nested pair
of lattices. One could also write this as%
\begin{equation}
X=\underset{L_{1}\in Gr(X)}{\underrightarrow{\operatorname*{colim}}}%
\underset{L_{1}^{\prime}\in Gr(X),L_{1}^{\prime}\subseteq L_{1}}%
{\underleftarrow{\lim}}\,\frac{L_{1}}{L_{1}^{\prime}}\text{,}\label{lcab2}%
\end{equation}
where $Gr(X)$ denotes the Sato Grassmannian of all lattices in $X$.

\begin{remark}
Let us look at the situation of a general $n$-Tate object, without the
cardinality hypothesis. We can always write $X$ as an Ind-diagram (of
Pro-objects) or Pro-diagram (of Ind-objects)%
\begin{equation}
X=\underset{L_{1}}{\underrightarrow{\operatorname*{colim}}}L_{1}\text{,}\qquad
X=\underset{L_{1}}{\underleftarrow{\lim}}\,\frac{X}{L_{1}}\label{lcab3}%
\end{equation}
where $L_{1}$ runs over the partially ordered set of lattices in $X$. The
first presentation follows trivially from \cite{TateObjectsExactCats} and
would work for a general exact category $\mathcal{C}$; for the second one
needs the dual viewpoint developed in \cite[\S \ Duality]{bgwTensor},
requiring $\mathcal{C}$ to be idempotent complete. This asymmetry stems from
\cite{TateObjectsExactCats} defining the Tate category as a sub-category of
$\mathsf{Ind}^{a}\mathsf{Pro}^{a}\mathcal{C}$ so that some `preferred
viewpoint' is built into the theory. Working from the outset with the opposite
category would shove the idempotent completeness assumption to the first
presentation and remove it from the second. The presentation in Equation
\ref{lcab2} can be obtained by first using the left-hand side presentation in
Equation \ref{lcab3}, and then employing the right-hand side presentation for
each $L_{1}$ individually. The co-directedness of the Sato Grassmannian and
\cite[Corollary 2]{bgwTensor} imply that instead of $L_{1}=\underset
{L_{1}^{\prime}}{\underleftarrow{\lim}}\,\frac{L_{1}}{L_{1}^{\prime}}$ with
$L_{1}^{\prime}$ running through lattices of the Pro-object $L_{1}$, we may
alternatively run through the lattices $L_{1}^{\prime}$ of $X$ which are
contained in $L_{1}$.
\end{remark}

By \cite[Prop. 6.6]{TateObjectsExactCats} any such quotient $L_{1}%
/L_{1}^{\prime}$ is an $\left(  n-1\right)  $-Tate object. This observation
generalizes to the case where $X$ is a general $n$-Tate object by refining our
presentation to%
\[
X=P\,\underset{L_{1}}{\underrightarrow{\operatorname*{colim}}}\underset
{L_{1}^{\prime}}{\underleftarrow{\lim}}\,\frac{L_{1}}{L_{1}^{\prime}}\text{,}%
\]
where $P$ denotes an idempotent. By induction it follows that%
\[
X=P\,\underset{L_{1}}{\underrightarrow{\operatorname*{colim}}}\underset
{L_{1}^{\prime}}{\underleftarrow{\lim}}\,P_{L_{1},L_{1}^{\prime}}%
\,\underset{L_{2}}{\underrightarrow{\operatorname*{colim}}}\underset
{L_{2}^{\prime}}{\underleftarrow{\lim}}\,P_{L_{1},L_{1}^{\prime},L_{2}%
,L_{2}^{\prime}}\cdots\underset{L_{n}}{\underrightarrow{\operatorname*{colim}%
}}\underset{L_{n}^{\prime}}{\underleftarrow{\lim}}\,P_{\left(  \ldots\right)
}\,\frac{L_{n}}{L_{n}^{\prime}}\text{,}%
\]
where $L_{1}^{\prime}\hookrightarrow L_{1}$ are nested lattices of the
elementary $n$-Tate object underlying $X$, $L_{2}^{\prime}\hookrightarrow
L_{2}$ are nested lattices of the elementary $\left(  n-1\right)  $-Tate
object underlying $L_{1}/L_{1}^{\prime}$, $L_{3}^{\prime}\hookrightarrow
L_{3}$ are nested lattices of the elementary $\left(  n-2\right)  $-Tate
object underlying $L_{2}/L_{2}^{\prime}$,\ldots, and finally $L_{n}%
/L_{n}^{\prime}$ is an object of $\mathcal{C}$. The letters $P,P_{L_{1}%
,L_{1}^{\prime}},\ldots$ denote idempotents cutting out the respective Tate objects.

The results of the last section on double lattice factorizations, notably
Lemma \ref{Lemma_CanFactorTateMorThroughLatticePairs}, tell us that any
morphism%
\[
f:X_{1}\longrightarrow X_{2}%
\]
of $n$-Tate objects stems from a system of compatible morphisms $L_{n}%
/L_{n}^{\prime}\rightarrow N_{n}/N_{n}^{\prime}$ in the category $\mathcal{C}$
so that $f$ is induced from assembling these morphisms into%
\[
\underset{L_{1}}{\underrightarrow{\operatorname*{colim}}}\underset
{L_{1}^{\prime}}{\underleftarrow{\lim}}\cdots\underset{L_{n}}{\underrightarrow
{\operatorname*{colim}}}\underset{L_{n}^{\prime}}{\underleftarrow{\lim}%
}\,\frac{L_{n}}{L_{n}^{\prime}}\overset{f}{\longrightarrow}\underset{N_{1}%
}{\underrightarrow{\operatorname*{colim}}}\underset{N_{1}^{\prime}%
}{\underleftarrow{\lim}}\cdots\underset{N_{n}}{\underrightarrow
{\operatorname*{colim}}}\underset{N_{n}^{\prime}}{\underleftarrow{\lim}%
}\,\frac{N_{n}}{N_{n}^{\prime}}\text{.}%
\]
If we take over from Lemma \ref{Lemma_CanFactorTateMorThroughLatticePairs} the
notation that the induced morphism of a double lattice factorization is
$\overline{f}$, the morphisms $L_{n}/L_{n}^{\prime}\rightarrow N_{n}%
/N_{n}^{\prime}$ here are nothing but \textquotedblleft$n$-fold overline
$f$\textquotedblright.

\begin{theorem}
\label{thm_CSplitExactIdealsAddUpGetNFoldCubicalBeilAlgebra}Suppose
$\mathcal{C}$ is a split and idempotent complete exact category and
$X\in\left.  n\text{-}\mathsf{Tate}_{\aleph_{0}}^{el}(\mathcal{C})\right.  $
or $\left.  n\text{-}\mathsf{Tate}_{\aleph_{0}}(\mathcal{C})\right.  $, i.e.
$X$ is a countable $n$-Tate object. Then $X$ is $n$-sliced and
$\operatorname*{End}(X)$ carries the structure of an $n$-fold
cubical\ Beilinson algebra in the sense of Definition \ref{def_BeilNFoldAlg}.
\end{theorem}

\begin{proof}
For a $1$-Tate object this is literally Prop.
\ref{Prop_SlicedTateObjIdealsSumUp}. In general, by Prop.
\ref{cor_SlicedTateObjectsHaveEndNFoldCubical}, it suffices to find a system
of good idempotents. Proceed by induction in $n$. Write an elementary $n$-Tate
object as%
\begin{equation}
X=\underset{L_{1}}{\underrightarrow{\operatorname*{colim}}}\underset
{L_{1}^{\prime}}{\underleftarrow{\lim}}\frac{L_{1}}{L_{1}^{\prime}}%
\text{.}\label{lcw7}%
\end{equation}
As any quotient $L_{1}/L_{1}^{\prime}$ is an $\left(  n-1\right)  $-Tate
object and we assume our claim, i.e. the existence of a system of good
idempotents, for $n-1$, the idempotents supply a direct sum decomposition%
\[
L_{1}/L_{1}^{\prime}=\bigoplus_{s_{1},\ldots,s_{n-1}\in\{\pm\}}P_{1}^{s_{1}%
}\cdots P_{n-1}^{s_{n-1}}\left(  L_{1}/L_{1}^{\prime}\right)  \text{.}%
\]
If $L_{1}\hookrightarrow L_{2}$ is a larger lattice in $X$ and $L_{2}^{\prime
}\hookrightarrow L_{1}^{\prime}$ a smaller lattice, the split exactness allows
one to find a direct sum decomposition%
\[
\frac{L_{2}}{L_{2}^{\prime}}\simeq\frac{L_{1}}{L_{1}^{\prime}}\oplus\left(
\text{another }\left(  n-1\right)  \text{-Tate object}\right)  \text{.}%
\]
We can use the same idempotents $P_{1}^{\pm},\ldots,P_{n-1}^{\pm}$ to
decompose the new summand. As our indexing categories are countable, we can
exhaust $X$ in this fashion to get a choice of $n-1$ good idempotents on all
of $X$. Finally, on all of $X$, we get a further idempotent $P_{n}^{\pm}$,
just by splitting the entire presentation of Equation \ref{lcw7} as%
\[
0\longrightarrow\underset{L_{1}^{\prime}}{\underleftarrow{\lim}}\,\frac
{\tilde{L}}{L_{1}^{\prime}}\longrightarrow\underset{L_{1}}{\underrightarrow
{\operatorname*{colim}}}\underset{L_{1}^{\prime}}{\underleftarrow{\lim}%
}\,\frac{L_{1}}{L_{1}^{\prime}}\longrightarrow\underset{L_{1}}%
{\underrightarrow{\operatorname*{colim}}}\,\frac{L_{1}}{\tilde{L}%
}\longrightarrow0\text{,}%
\]
where $\tilde{L}$ is some fixed lattice of $X$. This gives us a full system of
$n$ good idempotents and thus proves our claim. If $X$ is a general Tate
object, $(X,p)$, use the idempotents $pP_{i}^{\pm}p$ of the underlying
elementary Tate object $X$ instead.
\end{proof}

\begin{openproblem}
What is the correct analogue of this theorem in the context of Hennion's Tate
categories for stable $\infty$-categories?
\cite{HennionTateObjsInStableInftyCats}
\end{openproblem}

We close this section by presenting an example where it is particularly easy
to find a system of good idempotents.

\begin{example}
\label{example_LaurentSeriesOverRHaveVeryEasyTateStructure}Let $R$ be a ring,
possibly non-commutative. Define the ring of formal Laurent series by
$R((t)):=R[[t]][t^{-1}]$. Then $R((t_{1}))((t_{2}))\ldots((t_{n}))$
canonically has a representative in $\left.  n\text{-}\mathsf{Tate}%
^{el}(\mathsf{Mod}_{R})\right.  $ via%
\begin{align*}
& X:=\text{\textquotedblleft}R((t_{1}))((t_{2}))\ldots((t_{n}%
))\text{\textquotedblright}\\
& \qquad\qquad=\underset{i_{n}}{\underrightarrow{\operatorname*{colim}}%
}\underset{j_{n}}{\underleftarrow{\lim}}\cdots\underset{i_{1}}%
{\underrightarrow{\operatorname*{colim}}}\underset{j_{1}}{\underleftarrow
{\lim}}\,\frac{1}{t_{1}^{i_{1}}\cdots t_{n}^{i_{n}}}R[t_{1},\ldots
,t_{n}]/(t_{1}^{j_{1}},\ldots,t_{n}^{j_{n}})\text{.}%
\end{align*}
When evaluating the colimits and limits in $\mathsf{Mod}_{R}$, we get the
usual $R$-module%
\[
R((t_{1}))((t_{2}))\ldots((t_{n}))
\]
and with a little more work one obtains the ring structure on it. Since the
(co)limits are taken over projective $R$-modules so that this object actually
could be constructed on the left in%
\[
\left.  n\text{-}\mathsf{Tate}_{\aleph_{0}}^{el}(P_{f}(R))\right.
\longrightarrow\left.  n\text{-}\mathsf{Tate}^{el}(\mathsf{Mod}_{R}%
)\text{,}\right.
\]
we deduce that the implications of Thm.
\ref{thm_CSplitExactIdealsAddUpGetNFoldCubicalBeilAlgebra} apply to this
particular object. For $i=1,\ldots,n$ define idempotents%
\[
P_{i}^{+}\,\sum a_{m_{1},\ldots,m_{n}}t_{1}^{m_{1}}\cdots t_{n}^{m_{n}}%
:=\sum\delta_{m_{i}\geq0}a_{m_{1},\ldots,m_{n}}t_{1}^{m_{1}}\cdots
t_{n}^{m_{n}}%
\]
for $a_{m_{1},\ldots,m_{n}}\in R$ and $P_{i}^{-}:=1-P_{i}^{+}$. This is a
system of good idempotents for $X$ in the sense of Definition
\ref{def_GoodIdempotents}. A description of the ideals $I_{(-)}^{\pm}$ is easy
to give; the statements would be analogous to those in Lemma
\ref{Lemma_IdealsAndIndProLimits} in \S \ref{sect_ApplicationsToAdeles}.
\end{example}

\begin{remark}
A different approach has been introduced by A. Yekutieli. He developed the
concept of semi-topological rings in \cite{MR1213064}, \cite{MR1352568}. If
one prefers Yekutieli's semi-topological theory over $n$-Tate objects, an
analogous construction is possible: If $R$ is a semi-topological ring, e.g.
with the discrete topology, Yekutieli shows that $R((t_{1}))((t_{2}%
))\ldots((t_{n}))$ also possesses a canonical structure as a semi-topological
ring itself. He has established a result in the style of Theorem
\ref{thm_CSplitExactIdealsAddUpGetNFoldCubicalBeilAlgebra} in \cite{MR3317764}%
. Whichever way one proceeds, one needs a replacement for classical
topological concepts: For example, it is known that for $n\geq2$ an $n$-local
field like $k((t_{1}))\ldots((t_{n}))$ is \textit{not} a topological field. A.
N. Parshin and I. B. Fesenko have resolved this issue by using sequential
topologies, cf. \cite{MR1850194}. K. Kato's version of Tate categories was
also introduced to address exactly this issue, we refer to the introduction of
\cite{MR1804933}. See \cite{bgwGeomAnalAdeles} for a comparison of these
different viewpoints.
\end{remark}

\section{\label{sect_RelationToProjectiveModules}Relation to projective
modules}

For a possibly non-commutative ring $R$ we denote by $P_{f}(R)$ the category
of finitely generated projective \textsl{right} $R$-modules. First, let us
recall the following:

\begin{theorem}
\label{thm_body_DrinfeldModulesSitInsideTateCat}(\cite[Thm. 5.30]%
{TateObjectsExactCats}) Suppose $R$ is a commutative ring.

\begin{enumerate}
\item Then $\mathsf{Tate}^{Dr}(R)$ admits a canonical fully faithful embedding
as a sub-category of $\mathsf{Tate}(P_{f}(R))$.

\item When restricting to countable cardinality, this becomes an equivalence,%
\[
\mathsf{Tate}_{\aleph_{0}}^{Dr}(R)\overset{\sim}{\longrightarrow}%
\mathsf{Tate}_{\aleph_{0}}(P_{f}(R))
\]

\end{enumerate}
\end{theorem}

We shall also need a result identifying split exact categories and projective
module categories, a type of projective generator argument. It applies to a
wide range of situations, so let us state it in this generality.

\begin{lemma}
\label{Lemma_SplitExactCatWithStandardObjectsAreProjModules}Let $\mathcal{C}$
be an idempotent complete, split exact category such that every object is a
direct summand of some fixed object $S\in\mathcal{C}$. Then the functor%
\begin{align}
\mathcal{C}  & \longrightarrow P_{f}(\operatorname*{End}\nolimits_{\mathcal{C}%
}(S))\label{lFunctorMorita}\\
Z  & \longmapsto\operatorname*{Hom}\nolimits_{\mathcal{C}}(S,Z)\nonumber
\end{align}
is an exact equivalence of exact categories.
\end{lemma}

\begin{proof}
Firstly, define $E:=\operatorname*{End}\nolimits_{\mathcal{C}}(S)$ and note
that for any $Z\in\mathcal{C}$ the group $\operatorname*{Hom}%
\nolimits_{\mathcal{C}}(S,Z)$ becomes a right $E$-module by the composition of
morphisms, i.e.%
\[
\operatorname*{Hom}\nolimits_{\mathcal{C}}(S,Z)\times\operatorname*{Hom}%
\nolimits_{\mathcal{C}}(S,S)\longrightarrow\operatorname*{Hom}%
\nolimits_{\mathcal{C}}(S,Z)\text{.}%
\]
This produces a functor $\mathcal{C}\rightarrow\mathsf{Mod}(E)$. Thus, for any
objects $Z_{1},Z_{2}\in\mathcal{C}$ we get an induced map of homomorphism
groups%
\begin{equation}
\operatorname*{Hom}\nolimits_{\mathcal{C}}(Z_{1},Z_{2})\longrightarrow
\operatorname*{Hom}\nolimits_{\mathsf{Mod}(E)}(\operatorname*{Hom}%
\nolimits_{\mathcal{C}}(S,Z_{1}),\operatorname*{Hom}\nolimits_{\mathcal{C}%
}(S,Z_{2}))\text{.}\label{lcw64}%
\end{equation}
For $Z_{1}=Z_{2}:=S$ it is an isomorphism. It also sends idempotents to
idempotents, so the object $S$ and all its direct summands are sent to $E$ and
direct summands of it. However, by assumption any object of $\mathcal{C}$ is
of this shape, so this gives an alternative description of the same functor,
and therefore implies that the map in Equation \ref{lcw64} is an isomorphism
for arbitrary $Z_{1},Z_{2}$. So the functor is fully faithful. Moreover, we
see that every object is sent to a direct summand of the free $E $-module $E$,
so the image of the functor consists of finitely generated projective
$E$-modules, which shows that the functor is well-defined. Conversely, every
finitely generated projective $E$-module $M$ is a direct summand of $E^{\oplus
n}$ for some $n$. Since Equation \ref{lcw64} is also an isomorphism for
$S^{\oplus n}$, the idempotent defining $M$ comes from an idempotent of
$S^{\oplus n}$. As $\mathcal{C}$ is idempotent complete, this idempotent
possesses a kernel. This shows that the functor is essentially surjective. As
a result, we have an equivalence of categories and since both are split exact,
this equivalence is necessarily exact. In either case, the only short exact
sequences are the split ones. This finishes the proof.
\end{proof}

\begin{remark}
Let us quickly address the uniqueness of such a presentation. Suppose
$S,S^{\prime}\in\mathcal{C}$ are objects both satisfying the assumptions of
the lemma. For example, $S^{\prime}:=S\oplus S$. Then the equivalences of
categories are related by the functor%
\begin{align*}
P_{f}(\operatorname*{End}\nolimits_{\mathcal{C}}(S))  & \longrightarrow
P_{f}(\operatorname*{End}\nolimits_{\mathcal{C}}(S^{\prime}))\\
M  & \longmapsto M\otimes_{\operatorname*{End}\nolimits_{\mathcal{C}}%
(S)}\operatorname*{Hom}\nolimits_{\mathcal{C}}(S^{\prime},S)\text{,}%
\end{align*}
which itself is an exact equivalence. Note that this is precisely the shape of
a Morita equivalence: $\operatorname*{Hom}\nolimits_{\mathcal{C}}(S^{\prime
},S) $ is the Morita bimodule with the rings $\operatorname*{End}%
\nolimits_{\mathcal{C}}(S)$ and $\operatorname*{End}\nolimits_{\mathcal{C}%
}(S^{\prime})$ acting from the left- and right respectively. Exchanging the
roles of $S$ and $S^{\prime}$ yields the Morita bimodule for the reverse direction.
\end{remark}

Suppose $\mathcal{C}$ is a split exact category. Moreover, suppose there is a
collection $\{S_{i}\}_{i\in\mathbf{N}}$ of objects $S_{i}\in\mathcal{C}$ such
that every object $X\in\mathcal{C}$ is a direct summand of some countable
direct sum of these $S_{i}$. Then $\mathcal{C}^{\prime}:=\mathsf{Tate}%
_{\aleph_{0}}\mathcal{C}$ is idempotent complete and split exact by Lemma
\ref{Lemma_BasicLemmataForGeneralTateObjects}. Moreover, there is the
canonical object%

\begin{equation}
Y:=\widehat{%
{\textstyle\prod_{\mathbf{N}}}
S}\oplus\widehat{%
{\textstyle\bigoplus_{\mathbf{N}}}
S}\text{,}\label{la32}%
\end{equation}
defined by%
\[
\widehat{%
{\textstyle\prod_{\mathbf{N}}}
S}:=%
{\textstyle\prod_{\mathbf{N}}}
{\textstyle\prod_{i\in\mathbf{N}}}
S_{i}\text{,}\qquad\text{and}\qquad\widehat{%
{\textstyle\bigoplus_{\mathbf{N}}}
S}:=%
{\textstyle\bigoplus_{\mathbf{N}}}
{\textstyle\bigoplus_{i\in\mathbf{N}}}
S_{i}\text{,}%
\]
viewed as a $\mathsf{Pro}^{a}$-object (respectively $\mathsf{Ind}^{a}%
$-object), and it follows from \cite[Prop. 5.24]{TateObjectsExactCats} that
every object $X\in\mathsf{Tate}_{\aleph_{0}}^{el}\mathcal{C}$ is a direct
summand of $Y$. Then of course the same holds in the idempotent completion
$\mathcal{C}^{\prime}$. As a result, we have shown the assumptions of our
argument, but for $\mathsf{Tate}_{\aleph_{0}}\mathcal{C}$ instead of
$\mathcal{C}$ and for the family $\{S_{i}\}_{i\in\mathbf{N}}$ we can take the
single object $Y$. We may now iterate this procedure to obtain the following.

\begin{definition}
\label{def_doublebracketstdobject}Let $\mathcal{C}$ be a split exact category.
For any object $X\in\mathcal{C}$ define%
\[
X((t)):=%
{\textstyle\prod_{\mathbf{N}}}
X\oplus%
{\textstyle\bigoplus_{\mathbf{N}}}
X\qquad\in\mathsf{Tate}^{el}\mathcal{C}\text{.}%
\]

\end{definition}

This is just a special case of Equation \ref{la32} in the case of a single object.

\begin{definition}
\label{Definition_StandardObject}Let $\mathcal{C}$ be a split exact category
and $\{S_{i}\}_{i\in\mathbf{N}}$ a collection of objects such that every
$X\in\mathcal{C}$ is a direct summand of a countable direct sum of objects in
$\{S_{i}\}$. Then we call%
\[
S:=\left(  \widehat{%
{\textstyle\prod_{\mathbf{N}}}
S}\oplus\widehat{%
{\textstyle\bigoplus_{\mathbf{N}}}
S}\right)  ((t_{2}))\cdots((t_{n}))
\]
a \emph{standard object} for $n$-$\mathsf{Tate}_{\aleph_{0}}\mathcal{C}$.
\end{definition}

\begin{theorem}
\label{thm_body_TateObjsAreProjModules}Let $\mathcal{C}$ be a idempotent
complete and split exact category with a countable collection $\{S_{i}\}$ of
objects as in Definition \ref{Definition_StandardObject}.

\begin{enumerate}
\item Then every object $X\in n$-$\mathsf{Tate}_{\aleph_{0}}\mathcal{C}$ is a
direct summand of a standard object.

\item There is an exact equivalence of exact categories%
\[
n\text{-}\mathsf{Tate}_{\aleph_{0}}(\mathcal{C})\overset{\sim}{\longrightarrow
}P_{f}\left(  R\right)
\]
for $R:=\operatorname*{End}\nolimits_{n\text{-}\mathsf{Tate}_{\aleph_{0}%
}(\mathcal{C})}(S)$ and $S$ any standard object.

\item Under this equivalence, the ideals $I_{i}^{\pm}$ in $R$ correspond to
the categorical ideals of Definition \ref{Def_CubicalStructureForTateObjects}.
\end{enumerate}
\end{theorem}

\begin{proof}
The first claim is just \cite[Prop. 7.4]{TateObjectsExactCats} and the second
is a direct consequence thanks to Lemma
\ref{Lemma_SplitExactCatWithStandardObjectsAreProjModules}. Part (3) is
obvious for the standard object and then use that every object is a direct
summand of the latter.
\end{proof}

Let us now adapt this result to the case of `Tate modules \textit{\`{a} la
Drinfeld}', we refer to \cite[\S 5.4]{TateObjectsExactCats} for a definition
and background information. This type of object has been introduced by
Drinfeld in his paper \cite{MR2181808} as a candidate for the local sections
of a reasonable notion of infinite-dimensional vector bundles over a scheme.

\begin{theorem}
\label{thm_body_ModulesForTateObjsALaDrinfeld}Let $R$ be a commutative ring.
Then there is an exact equivalence of categories%
\[
\mathsf{Tate}_{\aleph_{0}}^{Dr}(R)\overset{\sim}{\longrightarrow}%
P_{f}(E)\text{,}%
\]
where $E$ is the Beilinson $1$-fold cubical algebra%
\[
E:=\operatorname*{End}\nolimits_{\mathsf{Tate}_{\aleph_{0}}^{Dr}(R)}\left(
\left.  R((t))\right.  \right)  \text{,}%
\]
where \textquotedblleft$R((t))$\textquotedblright\ is understood as the Tate
module \`{a} la Drinfeld with this name in Drinfeld's paper \cite{MR2181808}.
\end{theorem}

\begin{proof}
We claim that we have exact equivalences of categories, namely%
\[
\mathsf{Tate}_{\aleph_{0}}^{Dr}(R)\overset{\sim}{\longrightarrow}%
\mathsf{Tate}_{\aleph_{0}}(P_{f}(R))\overset{\sim}{\longrightarrow}%
P_{f}\left(  E\right)  \text{,}%
\]
where the first equivalence stems from Theorem
\ref{thm_body_DrinfeldModulesSitInsideTateCat}. The latter exists since
$P_{f}(R) $ is an idempotent complete split exact category so that Theorem
\ref{thm_body_TateObjsAreProjModules} is applicable with%
\[
E:=\operatorname*{End}\nolimits_{\mathsf{Tate}_{\aleph_{0}}\left(
P_{f}(R)\right)  }\left(  \left.  R((t))\right.  \right)  \text{.}%
\]
To justify this, note that every finitely generated projective $R$-module is a
direct summand of a finitely generated free module $R^{\oplus n}$ for $n$
large enough, so $R((t))$, as in Definition \ref{def_doublebracketstdobject},
is indeed a standard object. However, now using the equivalence of Theorem
\ref{thm_body_DrinfeldModulesSitInsideTateCat} again, the full faithfulness
yields an isomorphism of rings
\[
E\cong\operatorname*{End}\nolimits_{\mathsf{Tate}_{\aleph_{0}}^{Dr}(R)}\left(
\left.  R((t))\right.  \right)  \text{,}%
\]
where now $\left.  R((t))\right.  $ is to be understood as the (rather: one
choice of a) Tate module \`{a} la Drinfeld corresponding to the Tate object
with the same name. However, Drinfeld himself introduced the corresponding
object already in his original paper \cite[\S 3, especially Example
3.2.2]{MR2181808}, and in fact it is also called $R((t))$ in loc. cit. This
establishes the claim.
\end{proof}

\begin{openproblem}
It would seem interesting to study the analogous problem without the
restriction to countable cardinality. This probably would lead to a very
complicated picture. Kaplansky has shown that a projective module over a ring
must necessarily be a direct sum of countably generated modules. Over the
years it has become increasingly clear that this perhaps surprising appearance
of questions of cardinality permeate the entire field \cite{MR2900662}. See
for example \cite{MR2900444} or \cite{MR3223352} for intricacies in the
context of Drinfeld's ideas.
\end{openproblem}

\section{\label{sect_TraceClassOps}Trace-class operators}

Suppose $A$ is an $n$-fold cubical algebra as in Definition
\ref{def_BeilNFoldAlg}. Then we call the intersection of ideals%
\[
I_{tr}:=\bigcap_{i=1,\ldots,n}I_{i}^{+}\cap I_{i}^{-}%
\]
the ideal of \emph{trace-class} operators in $A$. Let us say a few things
about the historical precursors of this concept: While the name is inspired
from a vaguely related definition in functional analysis, the present format
originates from Tate's 1968 article on residues for curves \cite{MR0227171}.
He considers a $1$-fold cubical algebra of $k$-linear maps and wants to define
a trace on $I_{tr}$, mimicking the usual trace. Sadly, the maps in his ideal
$I_{tr}$ need not have finite rank, so a priori it is not clear whether a
notion of trace exists for them at all. Tate then follows the principle that
any nilpotent map should have trace zero, no matter whether it has finite rank
or not. From this he distills the concept of a `finite-potent' map $-$ a map
for which some finite power has finite rank. Tate manages to develop a
well-defined trace for such maps. Nonetheless, this trace has some fairly
mysterious properties. Most notably it is not always linear, as was
conjectured by Tate and later shown by F. Pablos Romo \cite{MR2319783}, see
also \cite{MR2360831}, \cite{MR3177048} for a fairly complete analysis of this
issue. However, in all applications of Tate's trace one usually only needs it
for trace-class operators, i.e. maps in $I_{tr}$, rather than all
finite-potent maps. Restricted to $I_{tr}$, Tate's generalized trace becomes
linear and very well-behaved. In this section we shall generalize this concept
to Tate categories. Just as Tate's original work takes the classical finite
rank trace as input, we shall also use a notion of trace on the input category
$\mathcal{C}$ as the starting point for the construction:

\begin{definition}
Let $\mathcal{C}$ be an exact category. An \emph{exact trace} on $\mathcal{C}
$ with values in an abelian group $Q$ is for each object $X\in\mathcal{C}$ a
group homomorphism%
\[
\operatorname*{tr}\nolimits_{X}:\operatorname*{End}\nolimits_{\mathcal{C}%
}(X)\longrightarrow Q
\]
so that the following properties hold:

\begin{enumerate}
\item \emph{(Zero on commutators)} For $f,g\in\operatorname*{End}%
\nolimits_{\mathcal{C}}(X)$ we have $\operatorname*{tr}\nolimits_{X}(fg-gf)=0$.

\item \emph{(Additivity)} For a short exact sequence $A\hookrightarrow
B\twoheadrightarrow A/B$ and $f\in\operatorname*{End}\nolimits_{\mathcal{C}%
}(B)$ so that $f\mid_{A}$ factors over $A$, we have%
\begin{equation}
\operatorname*{tr}\nolimits_{B}(f)=\operatorname*{tr}\nolimits_{A}(f\mid
_{A})+\operatorname*{tr}\nolimits_{A/B}(f)\text{.}\label{ladd1}%
\end{equation}

\end{enumerate}
\end{definition}

\begin{example}
For $\mathcal{C}:=\mathsf{Vect}_{f}$ the usual trace of a $k$-linear
endomorphism is an exact trace with values in the base field $k$.
\end{example}

\begin{example}
[Hattori-Stallings trace]Suppose $R$ is any unital associative ring, not
necessarily commutative. Let $P_{f}(R)$ be the category of finitely generated
right $R$-modules. For any $X\in P_{f}(R)$ the \emph{Hattori-Stallings trace}
is the morphism%
\begin{align*}
\operatorname*{tr}\nolimits_{X}:\operatorname*{End}\nolimits_{R}(X)  &
\longrightarrow R/[R,R]\\
X\otimes X^{\vee}  & \longrightarrow R/[R,R]\\
x\otimes x^{\vee}  & \longmapsto x^{\vee}(x)\text{.}%
\end{align*}
It is an exact trace. In fact, it is known to be universal on the category
$P_{f}(R)$, i.e. any exact trace with values in an abelian group $Q$ arises as
the composition of the Hattori-Stallings trace with a morphism
$R/[R,R]\rightarrow Q$. See \cite{MR0175950}, \cite{MR0202807} for the
original papers, \cite[\S 1]{MR564418} for a review. If $R:=k$ is a field, we
recover the classical trace.
\end{example}

\begin{example}
In Tate's theory in \cite{MR0227171} every nilpotent endomorphism has trace
zero. This need not hold in our context $-$ for entirely trivial reasons. We
give an explicit counter-example nonetheless: Take $\mathcal{C}:=P_{f}%
(\mathbf{Z}/2^{10})$. Define a trace $tr_{\mathbf{Z}/2^{10}}%
:\operatorname*{End}(\mathbf{Z}/2^{10})\rightarrow\mathbf{Z}/2^{10}$ as the
identity. Since $\mathcal{C}$ is split exact, the axiom regarding additivity
for exact sequences determines a unique continuation of $tr_{\mathbf{Z}%
/2^{10}}$ to the entire category $\mathcal{C}$. Clearly, multiplication with
$2$ is a nilpotent endomorphism of $\mathbf{Z}/2^{10}$, yet has trace $2$.
\end{example}

Once such a trace is available, we can lift it to the trace-class operators of
$n$-Tate objects:

\begin{proposition}
\label{Prop_TraceExists}Suppose $\mathcal{C}$ is an idempotent complete exact
category and $\operatorname*{tr}\nolimits_{(-)}$ an exact trace with values in
an abelian group $Q$. Then for every object $X\in\left.  n\text{-}%
\mathsf{Tate}(\mathcal{C})\right.  $ there is a canonically defined morphism%
\[
\tau_{X}:I_{tr}\rightarrow Q
\]
and these morphisms are uniquely determined by the following properties:

\begin{enumerate}
\item If $X\in\mathcal{C}$ then $\tau_{X}(f)=\operatorname*{tr}\nolimits_{X}%
(f)$ for all $f\in\operatorname*{End}(X)$.

\item Suppose $N^{\prime}\hookrightarrow N\hookrightarrow X$ are any lattices
of $X$ such that $\varphi\in I_{tr}(X)$ admits a lift%
\[%
%
%
\bfig\node x(0,0)[X]
\node xp(500,0)[X]
\node np(500,500)[N]
\arrow/>/[x`xp;\varphi]
\arrow/^{ (}->/[np`xp;]
\arrow/.>/[x`np;{\varphi}]
\efig
\]
and for which $\varphi\mid_{N^{\prime}}$ is zero, and thus factors as
$N/N^{\prime}\overset{\overline{\varphi}}{\longrightarrow}N/N^{\prime}$.\ Then
$\tau_{X}(\varphi):=\tau_{N/N^{\prime}}(\overline{\varphi})$. This element is
independent of the choice of $N,N^{\prime}$.

\item \emph{(Zero on commutators)} For $f,g\in I_{tr}(X)$ we have $\tau
_{X}(fg-gf)=0$.

\item \emph{(Additivity)} For a short exact sequence $A\hookrightarrow
B\twoheadrightarrow A/B$ and $f\in I_{tr}(B)$ so that $f\mid_{A}$ factors over
$A$, we have%
\[
\tau_{B}(f)=\tau_{A}(f\mid_{A})+\tau_{A/B}(\overline{f})\text{.}%
\]
We automatically have $f\mid_{A}\in I_{tr}(A)$ and $\overline{f}\in
I_{tr}(A/B)$.
\end{enumerate}
\end{proposition}

The endomorphism group in (1) makes sense in view of Lemma
\ref{lemma_TraceClassOfObjectInCIsAll}.

\begin{proof}
\textit{(Step 1)} The case $n=0$ is trivial and directly reduces to the axioms
of an exact trace. We deal with the case of an elementary $1$-Tate object $X$
first, i.e. assume $n=1$. Suppose $X\overset{\varphi}{\rightarrow}X$ is any
endomorphism in $\mathsf{Tate}^{el}(\mathcal{C})$. We call a diagram%
\[%
%
%
\bfig\node x(0,0)[X]
\node xp(500,0)[X]
\node n(0,500)[N]
\node np(500,500)[N]
\arrow/^{ (}->/[n`x;]
\arrow/>/[x`xp;\varphi]
\arrow/^{ (}->/[np`xp;]
\arrow/>/[x`np;]
\arrow/.>/[n`np;\overline{\varphi}]
\efig
\]
a\emph{\ }$1$-\emph{factorization} if $N\hookrightarrow X$ is a lattice, the
diagram commutes, and $\overline{\varphi}$ factors over $\overline{\varphi
}:N/N^{\prime}\rightarrow N/N^{\prime}$ with $N^{\prime}\hookrightarrow
N\hookrightarrow X$ a further (smaller) lattice. For any $1$-factorization of
$\varphi$ we can define a preliminary trace by $\tau(\varphi
):=\operatorname*{tr}\nolimits_{N/N^{\prime}}(\overline{\varphi})\in Q$ for
the simple reason that any quotient of lattices, e.g. $N/N^{\prime}$, must be
an object in $\mathcal{C}$, \cite[Prop. 6.6]{TateObjectsExactCats}. Next, we
note that any \textit{finite} morphism has a $1$-factorization: Since
$\varphi$ is bounded, it factors as $X\rightarrow N\hookrightarrow X$. Now,
restrict this to $N$. Since $\varphi$ is also discrete, there exists some
lattice $V$ with $V\hookrightarrow X\rightarrow X$ being zero, so let
$N^{\prime}$ be any common sub-lattice of $V$ and $N$. Such exists because of
the co-directedness of the Sato Grassmannian \cite[Theorem 6.7]%
{TateObjectsExactCats}. Now $N$ and $N^{\prime}$ satisfy all necessary
criteria. Suppose we find a further $1$-factorization with $N^{\prime}$
replaced by a smaller lattice $N^{\prime\prime}$. From $N^{\prime\prime
}\hookrightarrow N^{\prime}\hookrightarrow N$ we get the short exact sequence%
\[
\frac{N^{\prime}}{N^{\prime\prime}}\hookrightarrow\frac{N}{N^{\prime\prime}%
}\twoheadrightarrow\frac{N}{N^{\prime}}\text{.}%
\]
Since $\overline{\varphi}$ already vanishes on $N^{\prime}$, the trace must be
zero on the left-hand side term. The additivity axiom of the trace, Equation
\ref{ladd1}, implies that $\operatorname*{tr}\nolimits_{N/N^{\prime\prime}%
}(\overline{\varphi})=\operatorname*{tr}\nolimits_{N/N^{\prime}}%
(\overline{\varphi})+0$. Similarly, if we replace $N$ by a larger lattice
$N^{+}$, we get the short exact sequence%
\[
\frac{N}{N^{\prime}}\hookrightarrow\frac{N^{+}}{N^{\prime}}\twoheadrightarrow
\frac{N^{+}}{N}%
\]
and as $\overline{\varphi}$ factors over $N$ by assumption, the trace must be
zero on the right-hand side term. Again, we get $\operatorname*{tr}%
\nolimits_{N^{+}/N^{\prime}}(\overline{\varphi})=\operatorname*{tr}%
\nolimits_{N/N^{\prime}}(\overline{\varphi})+0$ by Equation \ref{ladd1}. This
shows that our preliminary definition of a trace is independent under
replacing $N^{\prime}$ by a smaller lattice and $N$ by a larger one. Since any
two lattices have a common sub-lattice and over-lattice, \cite[Theorem
6.7]{TateObjectsExactCats}, and $\operatorname*{tr}\nolimits_{N/N^{\prime}%
}(\overline{\varphi})\in Q$ is unchanged under replacing lattices this way, we
conclude that $\operatorname*{tr}\nolimits_{N/N^{\prime}}(\overline{\varphi})$
is actually independent of $N$ and $N^{\prime}$. In the same way we can show
that the trace is linear: Pick $N_{i},N_{i}^{\prime}$ for $i=1,2$ and both
morphisms in consideration and then verify the claim by replacing $N_{1}%
,N_{2}$ by a joint over-lattice, resp. sub-lattice for $N_{1}^{\prime}%
,N_{2}^{\prime}$. For a general (not necessarily elementary) $1$-Tate object,
we proceed as in \S \ref{sect_GeneralTateObjects}: The morphism $\varphi$ is
called finite if the underlying morphism of elementary Tate objects
$(X,p)\rightarrow(X,p)$ is finite, so we can take the trace which we have just
constructed.\newline Summarizing our findings, we have seen that axioms (1)
and (2) actually dictate a well-defined construction of a group homomorphism
$\tau_{X}:I_{tr}\rightarrow Q$. This implies the uniqueness and it remains to
show that axioms (3) and (4) hold. Vanishing on $[I_{tr},I_{tr}]$ is
immediately clear since we find $\overline{\varphi\circ\varphi^{\prime}%
}=\overline{\varphi}\circ\overline{\varphi^{\prime}}$ for composable
trace-class morphisms $\varphi,\varphi^{\prime}$, so the $1$-factorization of
a commutator can be expressed as the commutator of $1$-factorizations. Now use
the vanishing of exact traces on commutators. This proves (3). For (4),
suppose%
\begin{equation}
A\hookrightarrow B\twoheadrightarrow A/B\label{lc4}%
\end{equation}
is a short exact sequence and $f\in I_{tr}(B)$ is such that $f\mid_{A}$
factors over $A$. In this situation, Prop.
\ref{Prop_IfMorphismFactorsPreservesBasicProperties} guarantees that
$f\mid_{A}\in I_{tr}(A)$ and $\overline{f}\in I_{tr}(A/B)$ are also
trace-class. Moreover, if $f$ factors over $N/N^{\prime}$, it supplies us with
a short exact sequence%
\[
N_{1}/N_{1}^{\prime}\hookrightarrow N/N^{\prime}\twoheadrightarrow N_{2}%
/N_{2}^{\prime}\text{,}%
\]
where $N_{1}^{\prime}\hookrightarrow N_{1}\hookrightarrow A$ and
$N_{2}^{\prime}\hookrightarrow N_{2}\hookrightarrow A/B$ are lattices. Of
course these quotients are objects in $\mathcal{C}$, so the additivity of the
exact trace tells us that%
\[
\operatorname*{tr}\nolimits_{N_{1}/N_{1}^{\prime}}(f)=\operatorname*{tr}%
\nolimits_{N/N^{\prime}}(f\mid_{A})+\operatorname*{tr}\nolimits_{N_{2}%
^{\prime}/N_{2}}(f)\text{,}%
\]
but by (1) and (2) each of these traces is just a way to evaluate our trace
$\tau_{(-)}$, and we get%
\[
\tau_{B}(f)=\tau_{A}(f\mid_{A})+\tau_{A/B}(\overline{f})\text{,}%
\]
which is exactly what we wanted to show. This settles axiom (4).\newline%
\textit{(Step 2)} For an elementary $n$-Tate object proceed exactly as above,
but combined with an induction: Define a\emph{\ }$n$\emph{-factorization} just
like a $1$-factorization%
\[%
%
%
\bfig\node x(0,0)[X]
\node xp(500,0)[X]
\node n(0,500)[N]
\node np(500,500)[N]
\arrow/^{ (}->/[n`x;]
\arrow/>/[x`xp;\varphi]
\arrow/^{ (}->/[np`xp;]
\arrow/>/[x`np;]
\arrow/.>/[n`np;{\varphi}]
\efig
\]
but with $\overline{\varphi}:N/N^{\prime}\rightarrow N/N^{\prime}$ an
endomorphism of an $(n-1)$-Tate object. By the definition of trace-class
operators, Definition \ref{Def_CubicalStructureForTateObjects}, this is now
again a trace-class operator for this $(n-1)$-Tate object:%
\[
\varphi\in I_{tr}=I_{1}^{+}\cap I_{1}^{-}\cap\left(
{\textstyle\bigcap\nolimits_{i=2,\ldots,n}}
I_{i}^{+}\cap I_{i}^{-}\right)  \text{,}%
\]
from which we deduce that $\overline{\varphi}\in I_{tr}(N/N^{\prime})$, as
$(n-1)$-Tate objects. Next, pick a $(n-1)$-factorization for $\overline
{\varphi}$ and proceed this way until we get a $1$-factorization. Define
$\tau(\varphi)$ as before by $\tau(\varphi):=\operatorname*{tr}%
\nolimits_{N/N^{\prime}}(\overline{\varphi})$ of this $1$-factorization. As in
the case of $1$-factorizations, verify that for any $j$-factorization
($j=1,\ldots,n$), replacing lattices by over- resp. sub-lattices does not
affect the value of $\operatorname*{tr}\nolimits_{N/N^{\prime}}(\overline
{\varphi})$:\ We prove this by induction starting from $j=1$. But this case
has already been dealt with in Step 1. For $j\geq2$ use the same argument as
in Step 1, adapted as follows: In Step 1 we essentially only used the
additivity property of the exact trace. Replace this by using the additivity
of our $\tau$, i.e. its axiom (4), of the previous induction step
$j-1$.\newline Now, by construction properties (1), (2) in our claim are
satisfied. Property (3) follows easily as in Step 1. In order to show axiom
(4), we can again just copy the proof in Step 1 since it only uses the
additivity of the exact trace, which we can again replace by the additivity of
our $\tau$, axiom (4), of the previous induction step.
\end{proof}

It would be very nice if one could prove the following result in greater generality.

\begin{proposition}
[Strong Commutator\ Vanishing]\label{Prop_StrongVanishing}Let $\mathcal{C}$ be
an abelian category. Suppose $X\in\left.  n\text{-}\mathsf{Tate}%
(\mathcal{C})\right.  $ and $R:=\operatorname*{End}(X)$, $I_{tr}\subseteq R$
the trace-class ideal. Then%
\[
\tau_{X}([I_{tr},R])=0\text{,}%
\]
i.e. a commutator of a trace-class endomorphism with an \emph{arbitrary}
endomorphism vanishes.
\end{proposition}

\begin{proof}
Let $X\in\left.  n\text{-}\mathsf{Tate}(\mathcal{C})\right.  $ and $\varphi\in
R$, $\varphi_{0}\in I_{tr}(R)$. Since $I_{tr}$ is a two-sided ideal,
$\varphi_{0}$, $\varphi\varphi_{0}$ and $\varphi_{0}\varphi$ are all
trace-class and thus we know that for each of them we can quotient
$X\rightarrow X$ step-by-step through lattices as in Prop.
\ref{Prop_TraceExists} (2), going from $n$-Tate objects to $0$-Tate objects
while preserving the value of $\tau$, so that we may assume from the outset
that $X\in\mathcal{C}$. We can also find such lattices simultaneously for all
three of them since in each step the directedness and co-directedness of the
Sato Grassmannian \cite[Theorem 6.7]{TateObjectsExactCats} assures us that we
can take a common over- (respectively sub-)lattice of the lattices we find for
each individual morphism. Now consider the commutative diagram%
\[%
%
%
\bfig\node x(0,0)[X]
\node xp(600,0)[X]
\node n(0,500)[{\mathrm{ker}(\varphi_0)}]
\node np(600,500)[{\mathrm{ker}(\varphi_0)}]
\node ns(0,-500)[{\mathrm{im}(\varphi_0)}]
\node nps(600,-500)[{\mathrm{im}(\varphi_0)}]
\arrow/^{ (}->/[n`x;]
\arrow/>/[x`xp;{\varphi\varphi_0}]
\arrow/^{ (}->/[np`xp;]
\arrow/>/[n`np;{\varphi\varphi_0}]
\arrow/>>/[x`ns;{\varphi_0}]
\arrow/>>/[xp`nps;{\varphi_0}]
\arrow/>/[ns`nps;{\varphi_0 \varphi}]
\efig
\]
The kernel and image exist since $\mathcal{C}$ is abelian. The top horizontal
arrow is clearly the zero map so that $\tau_{X}(\varphi\varphi_{0}%
)=\tau_{\operatorname*{im}(\varphi_{0})}(\varphi_{0}\varphi)$ by the
additivity axiom of the trace. Moreover, we have the commutative diagram%
\[%
%
%
\bfig\node x(0,0)[X]
\node xp(600,0)[X]
\node n(0,500)[{\mathrm{im}(\varphi_0)}]
\node np(600,500)[{\mathrm{im}(\varphi_0)}]
\node ns(0,-500)[X/{\mathrm{im}(\varphi_0)}]
\node nps(600,-500)[X/{\mathrm{im}(\varphi_0)}]
\arrow/^{ (}->/[n`x;]
\arrow/>/[x`xp;{\varphi_0 \varphi}]
\arrow/^{ (}->/[np`xp;]
\arrow/>/[n`np;{\varphi_0 \varphi}]
\arrow/>>/[x`ns;]
\arrow/>>/[xp`nps;]
\arrow/>/[ns`nps;{\varphi_0 \varphi}]
\efig
\]
Again, by additivity we must have $\tau_{X}(\varphi_{0}\varphi)=\tau
_{\operatorname*{im}(\varphi_{0})}(\varphi_{0}\varphi)$, since the bottom
horizontal arrow is the zero map.
\end{proof}

\section{\label{sect_TateExtension}The Tate extension}

Next, we recall a construction due to Beilinson \cite{MR565095}, generalizing
Tate's ingenious insight from \cite{MR0227171} for $n=1$:

\begin{construction}
\label{construction_1}Let $k$ be a field. For every $n$-fold cubically
decomposed algebra $(A,(I_{i}^{\pm}),\tau)$ over $k$, as in Definition
\ref{BT_DefCubicallyDecompAlgebra}, there is a canonically defined Lie
cohomology class%
\[
\phi_{\operatorname*{Beil}}\in H_{\operatorname*{Lie}}^{n+1}(\mathfrak{g}%
,k)\text{,}%
\]
where $\mathfrak{g}:=A_{Lie}$ is the Lie algebra associated to $A$ via the
commutator $[x,y]_{\mathfrak{g}}:=xy-yx$.
\end{construction}

This cohomology class was introduced in \cite{MR565095}. An explicit formula
and example computations can be found in \cite{MR3207578}, \cite{olliTateRes}.

\begin{example}
For $n=1$ Tate constructs \textquotedblleft the original\textquotedblright%
\ cubically decomposed algebra in \cite[Prop. 1]{MR0227171} $-$ this is the
example which has started the entire subject in a way. It was independently
found by many others, notably Kac-Peterson \cite{MR619827} or the Japanese
school, cf. Date-Jimbo-Kashiwara-Miwa \cite{MR688946}. See also the complete
cohomology computations by Feigin and Tsygan \cite{MR705056}. For a certain
field $K$, Tate's paper \cite[Theorem 1]{MR0227171} constructs a map,
following the notation of \textit{loc. cit.},%
\[
\operatorname*{res}:K\wedge K\longrightarrow k\text{,}\qquad\text{(the
\textquotedblleft abstract residue\textquotedblright)}%
\]
which can be re-intepreted as $\phi_{\operatorname*{Beil}}\in
H_{\operatorname*{Lie}}^{2}(K_{Lie},k)$. It produces a map $\Omega_{K/k}%
^{1}\rightarrow k$, $f\mathrm{d}g\mapsto\operatorname*{res}(f\wedge g)$ which
agrees with the usual one-dimensional residue of a rational $1$-form at a
point. Going well beyond the viewpoint in \cite[Prop. 1]{MR0227171}, one can
regard the Lie $2$-cocycle $\phi_{\operatorname*{Beil}}$ as defining a Lie
algebra central extension%
\[
k\longrightarrow\widehat{\mathfrak{g}}\longrightarrow K\text{.}%
\]
The Lie algebra $\widehat{\mathfrak{g}}$ is an example of what is nowadays
called \emph{Tate's central extension}. In this case, $\widehat{\mathfrak{g}}
$ is known as the \emph{Heisenberg Lie algebra}. The theory is presented and
used from this perspective for example in \cite[\S 2.4]{MR962493},
\cite[\S 2.10, \S 2.13]{MR1988970}, \cite[\S 2.7]{MR2058353}, \cite{MR2563803}%
, \cite{MR2972546}, etc\ldots\ Applications abound.
\end{example}

\begin{example}
For $n=2$ the cocycle $\phi_{\operatorname*{Beil}}\in H_{\operatorname*{Lie}%
}^{3}(\mathfrak{g},k)$, applied to a doubly infinite matrix Lie algebra, has
been studied in great detail by Frenkel and Zhu \cite{MR2972546}.
\end{example}

The construction of a trace for trace-class operators in the previous section
allows us to define a higher Tate extension class for the Lie algebras
underlying endomorphism algebras of suitable $n$-Tate objects. In
\cite{olliTateRes} Beilinson's construction was lifted from Lie cohomology to
Hochschild and cyclic homology. These generalize analogously to our present situation.

\begin{theorem}
\label{Thm_BTClassesExist}Let $\mathcal{C}$ be a $k$-linear abelian category
with a $k$-valued exact trace. For every $n$-sliced object $X\in\left.
n\text{-}\mathsf{Tate}(\mathcal{C})\right.  $ the endomorphism algebra
$E:=\operatorname*{End}(X)$ is a cubically decomposed algebra in the sense of
Definition \ref{BT_DefCubicallyDecompAlgebra}.

\begin{enumerate}
\item In particular, its Lie algebra $\mathfrak{g}_{X}:=E_{Lie}$ carries a
canonical Beilinson-Tate Lie cohomology class,%
\[
\phi_{\operatorname*{Beil}}\in H_{\operatorname*{Lie}}^{n+1}(\mathfrak{g}%
_{X},k)
\]
via Construction \ref{construction_1}. Alternatively, one may view this as a
functional $\phi_{\operatorname*{Beil}}:H_{n+1}^{\operatorname*{Lie}%
}(\mathfrak{g}_{X},k)\rightarrow k$.

\item There is also a canonical Hochschild homology and cyclic homology
functional%
\[
\phi_{HH}:HH_{n}(E)\rightarrow k\qquad\text{resp.}\qquad\phi_{HC}%
:HC_{n}(E)\rightarrow k\text{.}%
\]

\end{enumerate}
\end{theorem}

The Hochschild and the Lie invariant are not completely independent of each
other, cf. \cite{olliTateRes} for details on the interplay of these constructions.

\begin{example}
Recall that, by Theorem
\ref{thm_CSplitExactIdealsAddUpGetNFoldCubicalBeilAlgebra}, if $\mathcal{C}$
is split exact and idempotent complete, every countable $n$-Tate object is
automatically $n$-sliced. For example, if we consider the category
$\mathcal{C}:=\mathsf{Vect}_{f}$, then the above theorem applies to all
objects in $\left.  n\text{-}\mathsf{Tate}_{\aleph_{0}}(\mathcal{C})\right.  $.
\end{example}

\begin{proof}
(1) By Prop. \ref{cor_SlicedTateObjectsHaveEndNFoldCubical}, the endomorphisms
$E:=\operatorname*{End}(X)$ form a Beilinson cubical algebra, but so far
without a trace formalism $\tau$. Prop. \ref{Prop_TraceExists} promotes the
$k$-valued exact trace on $\mathcal{C}$ to a trace%
\[
\tau_{X}:I_{tr}/[I_{tr},I_{tr}]\rightarrow k
\]
for any $n$-Tate object $X$ and $I_{tr}=I_{tr}(X,X)$ its trace-class
operators. As $\mathcal{C}$ is abelian, Prop. \ref{Prop_StrongVanishing},
shows that this trace satisfies the stronger axioms of a cubically decomposed
algebra. Finally, Construction \ref{construction_1} constructs $\phi
_{\operatorname*{Beil}}$; here we refer to \cite{MR565095} for the actual
construction. (2) The construction of these maps just takes a cubically
decomposed algebra as its input, so we can directly feed $E$ into
\cite{olliTateRes}.
\end{proof}

\begin{example}
If $n=1$, this means that $\mathfrak{g}_{X}$ comes equipped with a canonical
Lie central extension%
\[
0\longrightarrow k\longrightarrow\widehat{\mathfrak{g}_{X}}\longrightarrow
\mathfrak{g}_{X}\longrightarrow0
\]
and if $\mathcal{C}:=\mathsf{Vect}_{f}$ and we employ the usual trace, this
produces most of the classical examples of Tate's central extension. For
example, if $\mathfrak{g}$ is a simple Lie algebra, its loop Lie algebra%
\[
\mathfrak{g}((t)):=\underset{i}{\underrightarrow{\operatorname*{colim}}%
}\underset{j}{\underleftarrow{\lim}}\,t^{-i}\mathfrak{g}[t]/t^{j}%
\]
can naturally be viewed as an object in $\left.  1\text{-}\mathsf{Tate}%
_{\aleph_{0}}\mathcal{C}\right.  $. The adjoint representation can be promoted
to a Lie algebra embedding%
\begin{align*}
\widetilde{\operatorname*{ad}}:\mathfrak{g}((t))  & \hookrightarrow
\mathfrak{E}:=\operatorname*{End}\nolimits_{\left.  1\text{-}\mathsf{Tate}%
_{\aleph_{0}}\mathcal{C}\right.  }(\mathfrak{g}((t)))_{Lie}\\
x  & \mapsto(y\mapsto\lbrack x,y])
\end{align*}
(on the left-hand side view $\mathfrak{g}((t))$ as a plain Lie algebra and not
as a $1$-Tate object). The pullback of $\phi_{\operatorname*{Beil}}\in
H_{\operatorname*{Lie}}^{2}(\mathfrak{E},k)$ along $\widetilde
{\operatorname*{ad}}$ is the Kac-Moody cocycle. This mechanism defines a
higher Lie cohomology class also for higher loop Lie algebras $\mathfrak{g}%
((t_{1}))\ldots((t_{n}))$, \cite{MR2972546}, \cite{MR3207578}.
\end{example}

\begin{example}
The classical residue for a rational $1$-form on a curve can be obtained as
follows: Let $X/k$ be a smooth integral curve and $x\in X$ a closed point.
Then the field of fractions of the completed local ring at $x$ has a canonical
structure as a $1$-Tate object in $\mathcal{C}:=\mathsf{Vect}_{f}$: To see
this, observe that%
\[
\widehat{\mathcal{O}}_{X,x}=\underset{i}{\underleftarrow{\lim}}\,\mathcal{O}%
_{X,x}/\mathfrak{m}_{X,x}^{i}%
\]
is a Pro-object of finite-dimensional $k$-vector spaces. The field of
fractions $\widehat{\mathcal{K}}_{X,x}:=\operatorname*{Frac}\widehat
{\mathcal{O}}_{X,x}$ can be written as the colimit over all finitely generated
$\widehat{\mathcal{O}}_{X,x}$-submodules of $\widehat{\mathcal{K}}_{X,x}$.
Combining both presentations allows us to view $\widehat{\mathcal{K}}_{X,x}%
\in\left.  1\text{-}\mathsf{Tate}_{\aleph_{0}}\mathcal{C}\right.  $. Of
course, this is only a special case of the Parshin-Beilinson ad\`{e}les, see
\S \ref{sect_ClassicalMotivatingExample}. The Hochschild functional of Theorem
\ref{Thm_BTClassesExist} supplies us with a canonical map%
\[
\phi_{HH}:HH_{1}(\operatorname*{End}\nolimits_{\left.  1\text{-}%
\mathsf{Tate}_{\aleph_{0}}\mathcal{C}\right.  }(\widehat{\mathcal{K}}%
_{X,x}))\longrightarrow k\text{.}%
\]
Since the multiplication map $z\mapsto\alpha\cdot z$ for any $\alpha
\in\widehat{\mathcal{K}}_{X,x}$ defines an endomorphism of this $1$-Tate
object, there is a canonical ring map from the rational function field
$k\left(  X\right)  $ (or $\widehat{\mathcal{K}}_{X,x}$) to the above
endomorphism algebra. Composing them, we get%
\[
HH_{1}(k\left(  X\right)  )\longrightarrow k
\]
and the Hochschild-Kostant-Rosenberg isomorphism identifies the left-hand side
with $\Omega_{k(X)/k}^{1}$. This map turns out to be the residue. This is the
Hochschild analogue of Tate's construction of the residue. See
\cite{olliTateRes} for details.\medskip\newline Note that Beilinson's paper
\cite{MR565095} would have used Beilinson's cubically decomposed algebra, see
Theorem \ref{Thm_BeilFlagInSchemeGivesCubDecompAlgebra}, instead of using a
Tate category. However, by our Theorem \ref{thm_intro_OnAdeles} these are
isomorphic. Alternatively, one could also use Yekutieli's cubically decomposed
algebra, see \cite{MR3317764}.
\end{example}

\begin{remark}
The structures produced by Theorem \ref{Thm_BTClassesExist} can be viewed as
\textquotedblleft linearizations\textquotedblright\ of a more involved
non-linear extension on the level of groups, resp. algebraic $K$-theory. See
\cite{IndexMapAlgKTheory} for the non-linear version. For $K$-groups in low
degrees, notably $K_{1}$ and $K_{2}$, this has been pioneered by
\cite{MR619827} and \cite{MR1013132}. See \cite{MR900587} for the analogue in
topological $K$-theory.
\end{remark}

\section{\label{sect_ApplicationsToAdeles}Applications to ad\`{e}les of
schemes}

We refer to \cite[\S 7.2]{TateObjectsExactCats} for a detailed treatment of
the relation between Parshin-Beilinson ad\`{e}les and $n$-Tate objects. Let
$\mathsf{Ab}$ (resp. $\mathsf{Ab}_{f}$) be the category of all (resp. finite)
abelian groups. Suppose $X$ is a scheme, finite type of pure Krull dimension
$n$ over $\operatorname*{Spec}R$ for $R$ some commutative ring. Moreover, let
$\mathcal{F}$ be a quasi-coherent sheaf on $X$ and we fix a subset
$\triangle\subseteq S\left(  X\right)  _{n}$ of the flags of scheme points.

The treatment of \cite[\S 7.2]{TateObjectsExactCats} views ad\`{e}les as an
elementary $n$-Tate object in coherent sheaves of $X$ with zero-dimensional
support, i.e. with a slight abuse of language, we could say%
\[
A(\triangle,\mathcal{F})\in\left.  n\text{-}\mathsf{Tate}^{el}(\mathsf{Coh}%
_{0}X)\right.  \text{.}%
\]
This gives an exact functor $\mathsf{QCoh}\left(  X\right)  \rightarrow\left.
n\text{-}\mathsf{Tate}^{el}(\mathsf{Coh}_{0}X)\right.  $, $\mathcal{F}\mapsto
A(\triangle,\mathcal{F})$. Of course, one might wish to distinguish between
$A(\triangle,\mathcal{O}_{X})$ as an $n$-Tate object of coherent sheaves, or
as the $\mathcal{O}_{X}$-module sheaf one obtains by carrying out the
respective limits and colimits in the bi-complete category of $\mathcal{O}%
_{X}$-module sheaves $A(\triangle,\mathcal{F})\in\left.  \mathsf{Mod}%
(\mathcal{O}_{X})\right.  $. However, this distinction will always be clear
from the context.

\begin{remark}
The category of quasi-coherent $\mathcal{O}_{X}$-module sheaves $\mathsf{QCoh}%
\left(  X\right)  $ is also bi-complete, but carrying out the limits in this
category instead would \textsl{not} form an exact functor $\left.
n\text{-}\mathsf{Tate}^{el}(\mathsf{Coh}_{0}X)\right.  \rightarrow
\mathsf{QCoh}\left(  X\right)  $, and furthermore the resulting objects do not
appear to be particularly interesting. In fact, in $\mathsf{QCoh}\left(
X\right)  $ even taking countable infinite products $\prod_{\mathbf{Z}}$ is
not an exact functor.
\end{remark}

However, if $R=k$ is a field, we may alternatively take global sections,
$\mathsf{Coh}_{0}X\overset{\Gamma}{\longrightarrow}\mathsf{Vect}_{f}$, and
view the ad\`{e}les as an $n$-Tate object in finite-dimensional $k$-vector
spaces,%
\[
A(\triangle,\mathcal{F})\in\left.  n\text{-}\mathsf{Tate}^{el}(\mathsf{Vect}%
_{f})\right.  \text{.}%
\]
For applications in number theory it is interesting to look at schemes over
$\operatorname*{Spec}\mathbf{Z}$. Then the global section functor allows to
formulate the ad\`{e}les as an $n$-Tate object in finite abelian groups, that
is%
\[
A(\triangle,\mathcal{F})\in\left.  n\text{-}\mathsf{Tate}^{el}(\mathsf{Ab}%
_{f})\right.  \text{.}%
\]
All these variations of the ad\`{e}les provide a rich source of examples of
higher Tate objects. We shall show:

\begin{theorem}
\label{thm_body_AdelesAreSlicedOverVectSpaces}Let $k$ be a field and $X/k$ a
reduced finite type scheme of pure dimension $n$. For any quasi-coherent sheaf
$\mathcal{F}$ and subset $\triangle\subseteq S\left(  X\right)  _{n}$ the
Beilinson-Parshin ad\`{e}les $A(\triangle,\mathcal{F})$, viewed as an
elementary $n$-Tate object in finite-dimensional $k$-vector spaces, i.e. so
that%
\[
A(\triangle,\mathcal{F})\in n\text{-}\mathsf{Tate}^{el}\left(  \mathsf{Vect}%
_{f}\right)  \text{,}%
\]
is $n$-sliced (cf. Definition \ref{def_GoodIdempotents}). In particular,%
\[
E_{\triangle}^{\operatorname*{Tate}}:=\operatorname*{End}\left(
A(\triangle,\mathcal{O}_{X})\right)
\]
carries the structure of an $n$-fold cubical\ Beilinson algebra (cf.
Definition \ref{def_BeilNFoldAlg}).
\end{theorem}

The claim of this theorem fails if we instead view the ad\`{e}les as $n$-Tate
objects over $\mathsf{Coh}_{0}\left(  X\right)  $ or $\mathsf{Ab}_{f}$. These
variations are usually not $n$-sliced. We defer the proof and begin with some
negative examples:

\begin{example}
Suppose $X:=\operatorname*{Spec}\mathbf{Z}[t]$. We shall only consider
singleton flags $\triangle=\{(\eta_{0}>\eta_{1}>\eta_{2})\}$, defining objects
in the category $\left.  2\text{-}\mathsf{Tate}^{el}(\mathsf{Ab}_{f})\right.
$. We shall look at some examples modelled after $2$-local fields of mixed
characteristics $(0,0,p)$ and $(0,p,p)$.

\begin{enumerate}
\item $A((0)>(t)>(p,t),\mathcal{O}_{X})$ evaluates to what could be called
$\mathbf{Q}_{p}((t))$. The objects $t^{n}\mathbf{Q}_{p}[[t]]$ for
$n\in\mathbf{Z}$ are lattices and the relative quotients%
\begin{equation}
\frac{t^{n}\mathbf{Q}_{p}[[t]]}{t^{m}\mathbf{Q}_{p}[[t]]}\simeq\mathbf{Q}%
_{p}^{\oplus(m-n)}\label{lcw8}%
\end{equation}
for $m\geq n$ lie in $\left.  1\text{-}\mathsf{Tate}^{el}(\mathsf{Ab}%
_{f})\right.  $. Here the sub-objects $p^{i}\mathbf{Z}_{p}^{\oplus(n-m)}$ are
examples of lattices, with respective quotients $\simeq(\mathbf{Z}%
/p^{j}\mathbf{Z})^{\oplus n-m}\in\mathsf{Ab}_{f}$. We have%
\begin{equation}
I_{i}^{+}\left(  X,X\right)  +I_{i}^{-}\left(  X,X\right)
=\operatorname*{End}(X)\label{lcw6}%
\end{equation}
for $i=1$ by the presence of the splitting $\mathbf{Q}_{p}%
((t))\twoheadrightarrow\mathbf{Q}_{p}[[t]]$ which chops off the principal part
of the Laurent series. On the other hand, for $i=2$ Equation \ref{lcw6} fails.
It suffices to apply Example \ref{Example_FinAbelianPGroups} to the lattice
quotients appearing in Equation \ref{lcw8}.

\item $A((0)>(p)>(p,t),\mathcal{O}_{X})$ evaluates to something interesting.
It is denoted by $\mathbf{Q}_{p}\{\{t\}\}$ in \cite{MR1804915}, and can be
described explicitly as doubly infinite $\mathbf{Q}_{p}$-valued sequences with
boundedness conditions, namely%
\[
\mathbf{Q}_{p}\{\{t\}\}=\left\{  \left.  \sum_{i=-\infty}^{+\infty}a_{i}%
t^{i}\right\vert \exists C\in\mathbf{R}:a_{i}\in\mathbf{Q}_{p},\left\vert
a_{i}\right\vert _{p}\leq C\text{,}\underset{i\rightarrow-\infty}{\lim
}\left\vert a_{i}\right\vert _{p}=0\text{ }\right\}  \text{.}%
\]
It carries the structure of a $2$-local field. The objects $L_{n}%
:=p^{n}A((p)>(p,t),\mathcal{O}_{X})$ for $n\in\mathbf{Z}$, which identify with%
\[
L_{n}=\left\{  \left.  \sum_{i=-\infty}^{+\infty}a_{i}t^{i}\right\vert
a_{i}\in\mathbf{Q}_{p},\left\vert a_{i}\right\vert _{p}\leq p^{-n}%
\text{,}\underset{i\rightarrow-\infty}{\lim}\left\vert a_{i}\right\vert
_{p}=0\text{ }\right\}  \text{,}%
\]
are lattices and the relative quotients%
\[
\frac{p^{n}A((p)>(p,t),\mathcal{O}_{X})}{p^{m}A((p)>(p,t),\mathcal{O}_{X}%
)}\simeq(\mathbf{Z}/p^{m-n}\mathbf{Z})((t))
\]
for $m\geq n$ lie in $\left.  1\text{-}\mathsf{Tate}^{el}(\mathsf{Ab}%
_{f})\right.  $. Here we are in the opposite situation. Equation \ref{lcw6}
holds for $i=2$, but fails for $i=1$. The argument of Example
\ref{Example_FinAbelianPGroups} can be adapted to show the latter. For this
note that $\mathbf{Q}_{p}\{\{t\}\}/L_{n}$ is a $p$-primary torsion group.
\end{enumerate}

See for example \cite[Ch. I]{MR1804915} or \cite{MorrowHLF} for a further
discussion of higher local fields. These sources also explain the construction
of $F\{\{t\}\}$ for $F$ a general complete discrete valuation field. We leave
it to the reader to formulate its Tate object structure in general. All these
higher local fields arise as special cases of ad\`{e}les of suitably chosen
singleton flags.
\end{example}

\begin{example}
Suppose $X:=\operatorname*{Spec}k[t]$ and we view its ad\`{e}les as an
elementary $1$-Tate object in $\left.  1\text{-}\mathsf{Tate}^{el}%
(\mathsf{Coh}_{0}X)\right.  $. We show that it cannot be sliced. For
simplicitly, let us look at $\triangle=\{(0)>(t)\}$ and $\triangle^{\prime
}=\{(t)\}$. Then%
\[
A(\triangle^{\prime},\mathcal{O}_{X})\hookrightarrow A(\triangle
,\mathcal{O}_{X})\twoheadrightarrow A(\triangle,\mathcal{O}_{X})/A(\triangle
^{\prime},\mathcal{O}_{X})
\]
is a short exact sequence. Roughly speaking, it identifies with%
\[
k[[t]]\hookrightarrow k((t))\twoheadrightarrow k((t))/k[[t]]\text{.}%
\]
The Ind-object $k((t))/k[[t]]=\underset{i}{\underrightarrow
{\operatorname*{colim}}}\frac{1}{t^{i}}k[[t]]/k[[t]]$ is $t$-torsion. In
particular, if there was a section to $k((t))$, the latter would have to
possess non-trivial $t$-torsion elements. This is a contradiction. Quite
differently, in $\left.  1\text{-}\mathsf{Tate}^{el}\left(  \mathsf{Vect}%
_{f}\right)  \right.  $ a section exists.
\end{example}

\begin{proof}
[Proof of Theorem \ref{thm_body_AdelesAreSlicedOverVectSpaces}]This is not
very difficult because we can produce the required idempotents explicitly. For
the sake of simplicity let us write $\eta^{i}$ to denote the $i$-th ideal
power of the ideal sheaf of the reduced closed sub-scheme $\overline{\{\eta
\}}$ for a given scheme point $\eta\in X$. Moreover, let us write
\textquotedblleft$\mathcal{O}\left\langle f^{-1}\right\rangle $%
\textquotedblright\ to denote coherent sub-sheaves of $\mathcal{O}_{X,\eta}$,
indexed by $f$, so that the quasi-coherent sheaf $\mathcal{O}_{X,\eta}$ is
presented as the $\mathcal{O}_{X}$-module colimit over them, i.e. as depicted
on the left below:%
\[
\mathcal{O}_{X,\eta}=\underset{f\notin\eta}{\underrightarrow
{\operatorname*{colim}}}\,\mathcal{O}\left\langle f^{-1}\right\rangle
\text{.}\qquad\qquad\qquad R_{P}=\underset{f\notin P}{\underrightarrow
{\operatorname*{colim}}}\,\frac{1}{f}R\subset R[\frac{1}{f}]
\]
(This notation is supposed to be suggestive of the corresponding presentation
if $R$ is a ring and $P\in\operatorname*{Spec}R$ a prime ideal, as depicted
above on the right). We unwind the formation of ad\`{e}les directly from the
definition; but recall that by the limits and colimits we really mean the
respective diagrams of $n$-Tate objects and do not refer to carrying out
actual limits in the categories themselves, see \cite[\S 7.2]%
{TateObjectsExactCats} for details on how this can be implemented explicitly.
We arrive at%
\begin{align*}
A(\triangle,\mathcal{O}_{X})  & =%
{\textstyle\prod\limits_{\eta_{0}\in X}}
\underset{i_{0}}{\underleftarrow{\lim}}\,A\left(  \left.  _{\eta_{0}}%
\triangle\right.  ,\frac{\mathcal{O}_{X,\eta_{0}}}{\eta_{0}^{i_{0}}}\right) \\
& =%
{\textstyle\prod\limits_{\eta_{0}\in X}}
\underset{i_{0}}{\underleftarrow{\lim}}\underset{f_{0}\notin\eta_{0}%
}{\underrightarrow{\operatorname*{colim}}}\,A\left(  \left.  _{\eta_{0}%
}\triangle\right.  ,\frac{\mathcal{O}\left\langle f_{0}^{-1}\right\rangle
}{\eta_{0}^{i_{0}}}\right) \\
& =%
{\textstyle\prod\limits_{\eta_{0}\in X}}
\underset{i_{0}}{\underleftarrow{\lim}}\underset{f_{0}\notin\eta_{0}%
}{\underrightarrow{\operatorname*{colim}}}%
{\textstyle\prod\limits_{\eta_{1}\in X}}
\underset{i_{1}}{\underleftarrow{\lim}}\,A\left(  \left.  _{\eta_{1}\eta_{0}%
}\triangle\right.  ,\frac{\mathcal{O}\left\langle f_{0}^{-1}\right\rangle
}{\eta_{0}^{i_{0}}}\underset{\mathcal{O}_{X}}{\otimes}\frac{\mathcal{O}%
_{X,\eta_{1}}}{\eta_{1}^{i_{1}}}\right) \\
& =%
{\textstyle\prod\limits_{\eta_{0}\in X}}
\underset{i_{0}}{\underleftarrow{\lim}}\underset{\mathsf{Tate}^{el}%
}{\underbrace{\underset{f_{0}\notin\eta_{0}}{\underrightarrow
{\operatorname*{colim}}}%
{\textstyle\prod\limits_{\eta_{1}\in X}}
\underset{i_{1}}{\underleftarrow{\lim}}}}\underset{f_{1}\notin\eta_{1}%
}{\underrightarrow{\operatorname*{colim}}}\,A\left(  \left.  _{\eta_{1}%
\eta_{0}}\triangle\right.  ,\frac{\mathcal{O}\left\langle f_{0}^{-1}%
\right\rangle }{\eta_{0}^{i_{0}}}\underset{\mathcal{O}_{X}}{\otimes}%
\frac{\mathcal{O}\left\langle f_{1}^{-1}\right\rangle }{\eta_{1}^{i_{1}}%
}\right) \\
& =%
{\textstyle\prod\limits_{\eta_{0}\in X}}
\underset{i_{0}}{\underleftarrow{\lim}}\underset{\mathsf{Tate}^{el}%
}{\underbrace{\underset{f_{0}\notin\eta_{0}}{\underrightarrow
{\operatorname*{colim}}}%
{\textstyle\prod\limits_{\eta_{1}\in X}}
\underset{i_{1}}{\underleftarrow{\lim}}}}\underset{\mathsf{Tate}^{el}%
}{\underbrace{\underset{f_{1}\notin\eta_{1}}{\underrightarrow
{\operatorname*{colim}}}%
{\textstyle\prod\limits_{\eta_{2}\in X}}
\underset{i_{2}}{\underleftarrow{\lim}}}}\underset{f_{2}\notin\eta_{2}%
}{\underrightarrow{\operatorname*{colim}}}\,\left(  \ldots\right) \\
& \text{and so forth\ldots,}%
\end{align*}
where we only run through those $\eta_{0},\ldots,\eta_{n}$ such that $\eta
_{0}>\cdots>\eta_{n}\in\triangle$. The underbraces emphasize which parts of
this expression are to be read as limits or colimits respectively, and how to
group them to form Tate diagrams. We need to justify why the left-most limits,
left of the initial underbrace, exist: Since our scheme is of finite type and
$\eta_{0}$ runs through the irreducible components of $X$, the product over
the $\eta_{0}$ is finite. Similarly, the ideal sheaf of each respective
irreducible component is necessarily nilpotent so that for each $\eta_{0}$ the
limit over $i_{0}$ is over an essentially finite diagram. Unwinding this
further presents the ad\`{e}le object as an elementary $n$-Tate object (in the
sketch above only the first two outer-most Tate category iterations is
visible). Our claim is proven if we can exhibit pairwise commuting idempotents
$P_{j}^{+}$, $j=1,\ldots,n$, projecting this object onto the respective
lattice, indexed by $f_{j-1}=1$. This reduces to constructing sections%
\[
\frac{\mathcal{O}\left\langle f_{0}^{-1}\right\rangle }{\eta_{0}^{i_{0}}%
}\otimes\cdots\frac{\mathcal{O}\left\langle f_{j-1}^{-1}\right\rangle }%
{\eta_{j-1}^{i_{j-1}}}\cdots\otimes\frac{\mathcal{O}\left\langle f_{n-1}%
^{-1}\right\rangle }{\eta_{n-1}^{i_{n}}}\twoheadrightarrow\frac{\mathcal{O}%
\left\langle f_{0}^{-1}\right\rangle }{\eta_{0}^{i_{0}}}\otimes\cdots
\frac{\mathcal{O}_{X}}{\eta_{j-1}^{i_{j-1}}}\cdots\otimes\frac{\mathcal{O}%
\left\langle f_{n-1}^{-1}\right\rangle }{\eta_{n-1}^{i_{n}}}%
\]
in the category of finite-dimensional $k$-vector spaces (since once these
exist, they define straight morphisms between the respective Tate diagrams and
therefore the desired idempotents in the category of $n$-Tate objects).
However, the latter is obvious since the category of vector spaces is split exact.
\end{proof}

Of course Theorem \ref{thm_body_AdelesAreSlicedOverVectSpaces}\ provokes a
question:%
\[
E_{\triangle}^{\operatorname*{Beil}}=E_{\triangle}^{\operatorname*{Tate}%
}\,\text{?}%
\]
Beilinson had already shown, see Theorem
\ref{Thm_BeilFlagInSchemeGivesCubDecompAlgebra}, that for flags $\triangle
=\{(\eta_{0}>\cdots>\eta_{n})\}$ in a scheme $X$ a cubical algebra
$E_{\triangle}^{\operatorname*{Beil}}$ can be formed from his notion of
lattices. Its definition \textit{hinges crucially} on geometric data of $X$
simply because the underlying notion of lattice depends on $X$. On the other
hand, we have just seen that $E_{\triangle}^{\operatorname*{Tate}}$ is also a
cubical algebra. It comes with its own notion of lattices, which now only
depends on the structure as a Tate object. One can show that these two types
of lattices are different, albeit very closely related to each other. We refer
to \cite[\S 5]{bgwGeomAnalAdeles} for an explicit example illustrating this discrepancy.

However, the answer to our question is still affirmative:

\begin{theorem}
\label{thm_RelationCubicallyDecompAlgsBeilToTate}Let $k$ be a field. Suppose
$X/k$ is a reduced scheme of pure dimension $n$, and view the
Beilinson-Parshin ad\`{e}les $A(\triangle,\mathcal{O}_{X})$, for
$\triangle=\{(\eta_{0}>\cdots>\eta_{n})\}$ with $\operatorname*{codim}%
\nolimits_{X}\overline{\{\eta_{i}\}}=i$, as an $n$-Tate object in
finite-dimensional $k$-vector spaces. Then there is a canonical isomorphism of
Beilinson cubical algebras%
\[
E_{\triangle}^{\operatorname*{Beil}}\cong E_{\triangle}^{\operatorname*{Tate}%
}\text{.}%
\]

\end{theorem}

We shall split the proof into several lemmata. For notational clarity, let us
temporarily introduce the following distinction:

\begin{definition}
Suppose $M$ is a finitely generated $\mathcal{O}_{\eta_{0}}$-module and
$\triangle=\{(\eta_{0}>\cdots>\eta_{n})\}$ a flag, $\triangle^{\prime
}:=\{(\eta_{1}>\cdots>\eta_{n})\}$.

\begin{enumerate}
\item A \emph{Beilinson lattice} is a lattice in the sense of Definition
\ref{def_BeilLattice}, i.e. a finitely generated $\mathcal{O}_{\eta_{1}}%
$-module $L\subseteq M$ such that $\mathcal{O}_{\eta_{0}}\cdot L=M$.

\item A \emph{Tate lattice} is a lattice in the sense of Tate objects, i.e. a
sub-object of the $n$-Tate object $M_{\triangle}:=A(\triangle,M)$ which is a
Pro-object with an Ind-quotient.
\end{enumerate}
\end{definition}

\begin{lemma}
\label{lemma_BeilLatticesAreTateLattices}For $\triangle=\{(\eta_{0}%
>\cdots>\eta_{n})\}$ and $M$ a finitely generated $\mathcal{O}_{\eta_{0}}$-module,

\begin{enumerate}
\item $M_{\triangle}$ is an elementary $n$-Tate object and

\item for every Beilinson lattice $L\subseteq M$ we have that $L_{\triangle
^{\prime}}\hookrightarrow M_{\triangle}$ is a Tate lattice,

\item for every Tate lattice $T$ there exist Beilinson lattices $L_{1}%
\subseteq L_{2}$ such that%
\[
L_{1\triangle^{\prime}}\subseteq T\subseteq L_{2\triangle^{\prime}}\text{.}%
\]

\end{enumerate}
\end{lemma}

\begin{proof}
(1) The statement about $M_{\triangle}=A(\eta_{0}>\cdots>\eta_{n},M)$ is clear
from the discussion opening the section, i.e. essentially nothing but
\cite[\S 7.2]{TateObjectsExactCats}.\newline(2) Unwinding Equation
\ref{lcw106} for $L_{\triangle^{\prime}}=A(\eta_{1}>\cdots>\eta_{n},L)$, we
see that $L$ is one of the Tate lattices in the outer-most colimit so that we
clearly have, just by unwinding definitions, a canonical morphism%
\begin{align*}
L_{\triangle^{\prime}}  & \hookrightarrow M_{\triangle}\\
A(\triangle^{\prime},L)  & \hookrightarrow A(\triangle,M)\\
\underset{L_{1}^{\prime}}{\underleftarrow{\lim}}\underset{(n-1)\text{-Tate
object}}{\underbrace{\underset{L_{2}}{\underrightarrow{\operatorname*{colim}}%
}\cdots\underset{L_{n}}{\underrightarrow{\operatorname*{colim}}}%
\underset{L_{n}^{\prime}}{\underleftarrow{\lim}}\frac{L_{n}}{L_{n}^{\prime}}%
}}  & \hookrightarrow\underset{L_{1}}{\underrightarrow{\operatorname*{colim}}%
}\underset{L_{1}^{\prime}}{\underleftarrow{\lim}}\underset{(n-1)\text{-Tate
object}}{\underbrace{\underset{L_{2}}{\underrightarrow{\operatorname*{colim}}%
}\cdots\underset{L_{n}}{\underrightarrow{\operatorname*{colim}}}%
\underset{L_{n}^{\prime}}{\underleftarrow{\lim}}\frac{L_{n}}{L_{n}^{\prime}}}%
}\text{,}%
\end{align*}
where on the left-hand side we have replaced the colimit over $L_{1}$ by the
single value for $L_{1}:=L$ the lattice at hand. This is visibly a Pro-object
with an Ind-quotient, thus a Tate lattice.\newline(3) Let $T\hookrightarrow
M_{\triangle}$ be a Tate lattice. Since we may write $M_{\triangle}$ as
$M_{\triangle}=\underrightarrow{\operatorname*{colim}}_{L}L_{\triangle
^{\prime}}$ (by definition), i.e. as an Ind-diagram over Pro-objects, where
$L$ runs through all Beilinson lattices, the Pro-subobject $T$ must factor
through one of these $L_{\triangle^{\prime}}$. If $L_{2}$ denotes one such
index, i.e. the underlying Beilinson lattice, this means that
$T\hookrightarrow L_{2\triangle^{\prime}}$. The other direction is a little
more complicated: Let $L$ be any Beilinson lattice. Then the composition%
\[
L_{\triangle^{\prime}}\hookrightarrow M_{\triangle}\twoheadrightarrow
M_{\triangle}/T
\]
is a morphism from a Pro-object to an Ind-object.\ Thus, it must factor
through an $(n-1)$-Tate object $C$, i.e.%
\begin{equation}
L_{\triangle^{\prime}}\rightarrow P\hookrightarrow M_{\triangle}%
/T\label{lwxu1}%
\end{equation}
(\textit{Proof:} Since Pro-objects are left filtering in Tate objects by
\cite[Prop. 5.8]{TateObjectsExactCats}, the composed arrow $L_{\triangle
^{\prime}}\rightarrow M_{\triangle}/T$ factors as $L_{\triangle^{\prime}%
}\rightarrow P\hookrightarrow M_{\triangle}/T$, with $P$ a Pro-object. But
$M_{\triangle}/T$ is an Ind-object, so $P$ must be an Ind-object, too. By
\cite[Prop. 5.9]{TateObjectsExactCats} it follows that $P$ is an $(n-1)$-Tate
object). The object $L_{\triangle^{\prime}}$ can be presented as the
Pro-diagram%
\[
L\mapsto(L/L^{\prime})_{\triangle^{\prime}}=L_{\triangle^{\prime}%
}/L_{\triangle^{\prime}}^{\prime}\text{,}%
\]
where $L^{\prime}\subseteq L$ runs through all Beilinson sub-lattices,
partially ordered by inclusion. The quotients $(L/L^{\prime})_{\triangle
^{\prime}}$ are $(n-1)$-Tate objects, and since these are right filtering in
Pro-objects over them, \cite[Theorem 4.2 (2)]{TateObjectsExactCats}, it
follows that the arrow $L_{\triangle^{\prime}}\rightarrow P$ factors through
the projection to an object in the Pro-diagram%
\[
L_{\triangle^{\prime}}\twoheadrightarrow(L/L^{\prime})_{\triangle^{\prime}%
}\rightarrow P
\]
for a suitable $L^{\prime}\subseteq L$. Thus, returning to Equation
\ref{lwxu1} the composition%
\[
L_{\triangle^{\prime}}^{\prime}\hookrightarrow L_{\triangle^{\prime}%
}\rightarrow M_{\triangle}/T
\]
is zero. Thus, $L_{\triangle^{\prime}}^{\prime}\hookrightarrow T$ follows from
the universal property of kernels.
\end{proof}

\begin{remark}
The apparent asymmetry in the complexity of proving the existence of $L_{1}$
resp. $L_{2}$ in $L_{1\triangle^{\prime}}\subseteq T\subseteq L_{2\triangle
^{\prime}}$ is caused by the fact that we view Tate objects as a sub-category
of $\mathsf{Ind}^{a}\mathsf{Pro}^{a}(\mathcal{C})$. This is the place where
this kind of asymmetry is built in.
\end{remark}

Next, we observe that one can present the limits and colimits underlying the
ad\`{e}les in a particularly convenient format:

\begin{lemma}
\label{lemma_decodeindprolimits}Suppose we are in the situation of the theorem.

\begin{itemize}
\item Then for any $j=1,\ldots,n$ the following describes the same object:%
\[
\underset{L_{1}}{\underrightarrow{\operatorname*{colim}}}\underset
{L_{1}^{\prime}}{\underleftarrow{\lim}}\cdots\underset{L_{j}}{\underrightarrow
{\operatorname*{colim}}}\underset{L_{j}^{\prime}}{\underleftarrow{\lim}%
}\,A\left(  \eta_{j+1}>\cdots>\eta_{n},\frac{L_{j}}{L_{j}^{\prime}}\right)
\]
\newline where all the $L_{\ell}$ run increasingly through all finitely
generated $\mathcal{O}_{\eta_{\ell}}$-submodules of $\frac{L_{\ell-1}}%
{L_{\ell-1}^{\prime}}$ (if $\ell\geq2$) or $\mathcal{O}_{\eta_{0}}$ (if
$\ell=1$); the $L_{\ell}^{\prime}\subseteq L_{\ell}$ run decreasingly through
all finitely generated $\mathcal{O}_{\eta_{\ell}}$-submodules of $L_{\ell}$
having full rank.

\item This statement holds true irrespective of whether we carry out the
limits and colimits in the category of all $k$-vector spaces, or interpret it
as an elementary $j$-Tate object with values in an elementary $(n-j)$-Tate
object of finite-dimensional $k$-vector spaces.
\end{itemize}
\end{lemma}

\begin{proof}
The immediate evaluation of $A(\eta_{0}>\cdots>\eta_{n},\mathcal{O}_{X})$
straight from the definition unravels easily to become the case $j=1$ in the
statement.\ Inductively, one can transform the expression into its counterpart
for $j+1$. For details, cf. \cite[Lemma 30]{bgwGeomAnalAdeles}. This procedure
works with both interpretations, verbatim.
\end{proof}

We may read this lemma as a kind of induction step. For its final step, $j=n$,
we arrive at%
\begin{equation}
A(\eta_{0}>\cdots>\eta_{n},\mathcal{O}_{X})=\underset{L_{1}}{\underrightarrow
{\operatorname*{colim}}}\underset{L_{1}^{\prime}}{\underleftarrow{\lim}}%
\cdots\underset{L_{n}}{\underrightarrow{\operatorname*{colim}}}\underset
{L_{n}^{\prime}}{\underleftarrow{\lim}}\frac{L_{n}}{L_{n}^{\prime}}%
\text{,}\label{lcw106}%
\end{equation}
presenting the ad\`{e}le object on the left-hand side entirely in terms of
Beilinson lattices. This presentation bridges from the definition of the
ad\`{e}les in Beilinson's original paper \cite{MR565095} (or \cite{MR1138291},
\cite{MR1213064}, \cite{MR1374916} for secondary sources) to the ideals in
Beilinson's cubical algebra structure as given in Definition
\ref{def_BeilNFoldAlg}:

\begin{lemma}
\label{Lemma_IdealsAndIndProLimits}We keep the assumptions as in the theorem.
Below, the `roof symbol' $\widehat{(\cdots)}$ will denote omission:

\begin{enumerate}
\item Suppose $M_{1},M_{2}$ are finitely generated $\mathcal{O}_{\eta_{0}}%
$-modules. Then a $k$-linear map $f\in\operatorname*{Hom}\nolimits_{k}%
(M_{1\triangle},M_{2\triangle})$ lies in $\operatorname*{Hom}%
\nolimits_{\triangle}(M_{1},M_{2})$ if and only if it stems from a compatible
system of $k$-linear morphisms%
\[
\frac{L_{n}}{L_{n}^{\prime}}\longrightarrow\frac{N_{n}}{N_{n}^{\prime}%
}\text{,}%
\]
with $L_{n}^{\prime}\subseteq L_{n}$ (in $M_{1}$) and $N_{n}^{\prime}\subseteq
N_{n}$ (in $M_{2}$) suitable Beilinson lattices, inducing a morphism in the
limit/colimit%
\begin{align*}
f:M_{1\triangle} &  \longrightarrow M_{2\triangle}\\
\underset{L_{1}}{\underrightarrow{\operatorname*{colim}}}\underset
{L_{1}^{\prime}}{\underleftarrow{\lim}}\cdots\underset{L_{n}}{\underrightarrow
{\operatorname*{colim}}}\underset{L_{n}^{\prime}}{\underleftarrow{\lim}}%
\frac{L_{n}}{L_{n}^{\prime}} &  \longrightarrow\underset{N_{1}}%
{\underrightarrow{\operatorname*{colim}}}\underset{N_{1}^{\prime}%
}{\underleftarrow{\lim}}\cdots\underset{N_{n}}{\underrightarrow
{\operatorname*{colim}}}\underset{N_{n}^{\prime}}{\underleftarrow{\lim}}%
\frac{N_{n}}{N_{n}^{\prime}}\text{.}%
\end{align*}

\item We remain in the situation of (1). We have $f\in I_{i\triangle}%
^{+}(M_{1},M_{2})$ if and only if $f$ factors as%
\[
\underset{L_{1}}{\underrightarrow{\operatorname*{colim}}}\underset
{L_{1}^{\prime}}{\underleftarrow{\lim}}\cdots\underset{L_{n}}{\underrightarrow
{\operatorname*{colim}}}\underset{L_{n}^{\prime}}{\underleftarrow{\lim}}%
\frac{L_{n}}{L_{n}^{\prime}}\longrightarrow\underset{N_{1}}{\underrightarrow
{\operatorname*{colim}}}\underset{N_{1}^{\prime}}{\underleftarrow{\lim}}%
\cdots\underset{N_{i}}{\underrightarrow{\widehat{\operatorname*{colim}}}%
}\cdots\underset{N_{n}}{\underrightarrow{\operatorname*{colim}}}%
\underset{N_{n}^{\prime}}{\underleftarrow{\lim}}\frac{N_{n}}{N_{n}^{\prime}%
}\text{,}%
\]
i.e. instead of the colimit over all $N_{i}$ we can take a fixed $N_{i}$
(depending on $N_{1},N_{1}^{\prime},\ldots,N_{i-1},N_{i-1}^{\prime}$). We have
$f\in I_{i\triangle}^{-}(M_{1},M_{2})$ if and only if $f$ factors as%
\[
\underset{L_{1}}{\underrightarrow{\operatorname*{colim}}}\underset
{L_{1}^{\prime}}{\underleftarrow{\lim}}\cdots\underset{L_{i}}{\underleftarrow
{\widehat{\lim}}}\cdots\underset{L_{n}}{\underrightarrow{\operatorname*{colim}%
}}\underset{L_{n}^{\prime}}{\underleftarrow{\lim}}\frac{L_{n}}{L_{n}^{\prime}%
}\longrightarrow\underset{N_{1}}{\underrightarrow{\operatorname*{colim}}%
}\underset{N_{1}^{\prime}}{\underleftarrow{\lim}}\cdots\underset{N_{n}%
}{\underrightarrow{\operatorname*{colim}}}\underset{N_{n}^{\prime}%
}{\underleftarrow{\lim}}\frac{N_{n}}{N_{n}^{\prime}}\text{,}%
\]
i.e. instead of the limit over all $L_{i}$ we can take a fixed $L_{i}$
(depending on $L_{1},L_{1}^{\prime},\ldots,L_{i-1},L_{i-1}^{\prime}$).
\end{enumerate}
\end{lemma}

\begin{proof}
This follows rather directly from the definition. Firstly, unravel $A(\eta
_{0}>\cdots>\eta_{n},\mathcal{O}_{X})$ in terms of iterated limits and
colimits of lattices as in Equation \ref{lcw106}. But then ideal membership
for $I_{i}^{\pm}$ is exactly the property to factor through a Beilinson
lattice of the target, resp. a Beilinson lattice of the source.
\end{proof}

\begin{proof}
[Proof of Thm. \ref{thm_RelationCubicallyDecompAlgsBeilToTate}]For the sake of
clarity, we shall denote a Beilinson lattice by the letter $\mathcal{L}$ in
this proof, and Tate lattices by the letter $L$. We know that every Beilinson
lattice gives rise to a Tate lattice, Lemma
\ref{lemma_BeilLatticesAreTateLattices}, and conversely by the same Lemma
every possible Tate lattice $L$ is sandwiched as $\mathcal{L}_{1\triangle
}\subseteq L\subseteq\mathcal{L}_{2\triangle}$ between Beilinson lattices
$\mathcal{L}_{1},\mathcal{L}_{2}$. We now claim that%
\begin{equation}
E_{\triangle}^{\operatorname*{Beil}}\cong E_{\triangle}^{\operatorname*{Tate}%
}\label{lcw107}%
\end{equation}
holds as sets. This is seen as follows: Given a $k$-linear map $f\in
E_{\triangle}^{\operatorname*{Beil}}$ the definition of the subgroup
$\operatorname*{Hom}\nolimits_{\triangle}(\mathcal{O}_{\eta_{0}}%
,\mathcal{O}_{\eta_{0}})\subseteq\operatorname*{End}\nolimits_{k}%
(\mathcal{O}_{X\triangle},\mathcal{O}_{X\triangle})$ in Definition
\ref{def_HigherAdeleOperatorIdeals} guarantees that there exist factorizations%
\[
\overline{f}:(\mathcal{L}_{1}/\mathcal{L}_{1}^{\prime})_{\triangle^{\prime}%
}\rightarrow(\mathcal{L}_{2}^{\prime}/\mathcal{L}_{2})_{\triangle^{\prime}}%
\]
over suitable Beilinson lattices $\mathcal{L}_{1},\mathcal{L}_{1}^{\prime
},\mathcal{L}_{2}^{\prime},\mathcal{L}_{2}$. By the exactness of the ad\`{e}le
functor $(-)_{\triangle^{\prime}}$, this is nothing but $\overline
{f}:\mathcal{L}_{1\triangle^{\prime}}/\mathcal{L}_{1}^{\prime}{}%
_{\triangle^{\prime}}\rightarrow\mathcal{L}_{2\triangle^{\prime}}^{\prime
}/\mathcal{L}_{2\triangle^{\prime}}$. Hence, we get a (straight) morphism of
the explicit Tate diagrams arising from the presentation of Equation
\ref{lcw106}. In particular, this datum induces a morphism of $n$-Tate
objects. Conversely, any morphism of $n$-Tate objects can be factored over
lattice quotients in the desired shape by Lemma
\ref{Lemma_CanFactorTateMorThroughLatticePairs}. This proves Equation
\ref{lcw107} as an equality of sets because both maps are inverse to each
other. However, it is easy to check that these maps are in fact group
homomorphisms and also respect composition, so we get an isomorphism of
associative algebras. Lemma \ref{Lemma_IdealsAndIndProLimits} then establishes
the equality of ideals $I_{i}^{\pm}$: Just unravel the ideal membership
conditions and use that all Beilinson lattices give rise to Tate lattices, and
conversely we have the sandwiching property.
\end{proof}

\bibliographystyle{amsalpha}
\bibliography{ollinewbib}

\end{document}